%% file: higherklein_v7_.tex
\documentclass[11pt]{amsart}
\usepackage[margin=1.5in]{geometry}
\usepackage[colorlinks,citecolor=cyan,linkcolor=teal]{hyperref}

\author[]{James Farre}

\author[]{Beatrice Pozzetti}

\author[]{Gabriele Viaggi}

\newcommand{\Addresses}{{
  \bigskip
  \footnotesize  
  \noindent James Farre, \textsc{Max Planck Institute for Mathematics in the Sciences, Leipzig}\par\nopagebreak
  \textit{E-mail address}: \texttt{james.farre@mis.mpg.de}
  
  \noindent Beatrice Pozzetti, \textsc{Mathematical Institute, Heidelberg University, Heidelberg}\par\nopagebreak
  \textit{E-mail address}: \texttt{pozzetti@mathi.uni-heidelberg.de}
  
  \noindent Gabriele Viaggi, \textsc{Mathematical Institute, Sapienza University of Rome, Rome}\par\nopagebreak
  \textit{E-mail address}: \texttt{gabriele.viaggi@uniroma1.it}
  }}

\usepackage{amsmath}
\usepackage{amsthm}
\usepackage{amssymb}
\usepackage[all,cmtip]{xy}
\usepackage{todonotes}

\DeclareMathOperator{\Area}{Area}
\newcommand{\C}{\mathbb C}
\renewcommand{\P}{\mathbb P}
\newcommand{\R}{\mathbb R}
\newcommand{\calF}{\mathcal F}
\newcommand{\SL}{{\rm SL}}

\newcommand{\Gr}{{\rm Gr}}

\newcommand{\G}{\Gamma}
\newcommand{\HH}{\mathbb H}
\newcommand{\g}{\gamma}

\newcommand{\calT}{\mathcal T}
\newcommand{\cal}{\mathcal }

\newcommand{\bP}{\mathbb {P}}

\newcommand{\deG}{\partial\Gamma}

\renewcommand{\H}{{\mathbb H}}

\newcommand{\calA}{\mathcal A}
\newcommand{\calL}{\mathcal L}

\newcommand{\qc}{\mathcal{QC}}
\newcommand{\CC}{\mathbb C}

\newcommand{\inverse}{^{-1}}

\DeclareMathOperator{\Mod}{Mod}

\newcommand{\CP}{\mathbb{CP}^1}
\newcommand{\RP}{\mathbb{RP}^1 }
\newcommand{\dH}{\partial \mb{H}^2}
\newcommand{\Aut}{{\rm Aut}}
\newcommand{\bR}{{\mb{R}}}
\newcommand{\bH}{{\mb{H}}}
\newcommand{\bC}{{\mb{C}}}

\DeclareMathOperator{\DE}{DE}
\DeclareMathOperator{\Homeo}{Homeo}
\DeclareMathOperator{\tr}{tr}
\input{macros.tex}

\newcommand{\ep}{
\epsilon
}
\newcommand{\mc}[1]{
\mathcal{#1}
}
\newcommand{\mb}[1]{
\mathbb{#1}
}
\newcommand{\T}{
\mc{T}
}

\newcommand{\PSL}{{\rm PSL}}

\newcommand{\Charn}{\Xi^{\rm hyp}(\G,\PSL(d,\C))}
\newcommand{\Char}{\Xi(\G,\PSL_3(\C))}

\newtheorem*{thm*}{Theorem}
\newtheorem*{cor*}{Corollary}
\newtheorem*{prop*}{Proposition}

\newtheorem{thmA}{Theorem}

\newtheorem{thm}{Theorem}[section]

\newtheorem{cor}[thm]{Corollary}
\newtheorem{lemma}[thm]{Lemma}
\newtheorem{lem}[thm]{Lemma}
\newtheorem{prop}[thm]{Proposition}

\theoremstyle{definition}
\newtheorem{claim}[thm]{Claim}

\newtheorem*{defi*}{Definition}
\newtheorem{dfn}[thm]{Definition}

\newtheorem{remark}[thm]{Remark}

\theoremstyle{remark}

\newcommand{\nocontentsline}[3]{}
\newcommand{\tocless}[2]{\bgroup\let\addcontentsline=\nocontentsline#1{#2}\egroup}
\title{Geometry of hyperconvex representations of surface groups}

\begin{document}

\begin{abstract}
We study the geometry of hyperconvex representations of surface groups in ${\rm PSL}(d,\mb{C})$ and their deformation spaces: We produce a natural holomorphic extension of the classical Ahlfors--Bers map to a product of Teichmüller spaces of a canonical Riemann surface lamination and prove that the limit set of a hyperconvex representation in the full flag space has Hausdorff dimension 1 if and only if the representation is conjugate in ${\rm PSL}(d,\mb{R})$. 
\end{abstract}

\maketitle

\tableofcontents

\section{Introduction}
A quasi-Fuchsian representation of the fundamental group $\pi_1(\Sigma)$ of a (closed, oriented) surface is a discrete and faithful representation $\rho:\pi_1(\Sigma)\to{\rm PSL}(2,\mb{C})$ that preserves a Jordan curve $\Lambda\subset\mb{CP}^1$ on the Riemann sphere. These objects lie at the crossroad of several different areas of mathematics such as complex dynamics, Teichmüller theory, and 3-dimensional hyperbolic geometry. They have been fruitfully studied from all these perspectives displaying a very rich structure. 

From a dynamical point of view, an important invariant associated with a quasi-Fuchsian representation is the Hausdorff dimension of the invariant Jordan curve $\Lambda$. It is elementary to show that this number is always at least 1 and at most 2. One expects that generically $\Lambda$ is a complicated non-rectifiable fractal curve. A celebrated result of Bowen \cite{BowenQC} shows that this is indeed the case: The Hausdorff dimension is 1 if and only if the quasi-Fuchsian representation is Fuchsian, that is, it is conjugate in ${\rm PSL}(2,\mb{R})$.     

Our first contribution is to show that this phenomenon persists in a suitable form in the much larger class of {\em (fully) hyperconvex} representations of surface groups in ${\rm PSL}(d,\mb{C})$ with $d\ge 2$ arbitrary. 

Let us briefly introduce these objects: The prototype of a hyperconvex curve in $\mb{CP}^{d-1}$ is the image $\nu^1(\mb{CP}^1)$ of the Veronese embedding $\nu^1:\mb{CP}^1\to\mb{CP}^{d-1}$. This curve comes together with an osculating frame 
$\nu^1(p)\subset\cdots\subset \nu^{d-1}(p)$, where $\nu^k(p)\subset\mb{C}^d$ is the unique $k$-dimensional plane whose projectivization is tangent to the Veronese curve at $\nu^1(p)$ with order $k-1$. These planes satisfy the following transversality property: For every $p\in\mb{CP}^1$ and $k\le d-2$ we have that the projection 
\begin{equation}\label{e.Veronese}
\begin{array}{ccc}
\mb{CP}^1\setminus\{p\}&\to&\mb{P}\left(\nu^{k+1}(p)/\nu^{k-1}(p)\right)\\
q&\mapsto &[\nu^{d-k}(q)\cap\nu^{k+1}(p)]
\end{array}
\end{equation}
is injective. 
Mimicking this picture, we say that a representation $\rho:\pi_1(\Sigma)\to{\rm PSL}(d,\mb{C})$ is \emph{$k$-hyperconvex} if it admits an equivariant dynamics preserving boundary map $\xi_\rho:\partial\pi_1(\Sigma)\to\mc{F}_{k-1,k+1,d-k}(\mb{C}^d)$ in the partial flag manifold of $\mb{C}^d$ consisiting of subspaces of dimensions $\{k-1,k+1,d-k\}$ satisfying the  transversality property \eqref{e.Veronese} (see also \S\ref{subsec: hyperconvex}). Here $\partial\pi_1(\Sigma)\cong \mb{S}^1$ denotes the Gromov boundary of the word-hyperbolic group $\pi_1(\Sigma)$. We furthermore say that a representation is \emph{fully hyperconvex} if it is $k$-hyperconvex for all $1\leq k\leq d-1$, in this case its boundary map naturally has image in the full flag manifold $\mc F(\C^d)$.  The image $\Lambda\subset\mc{F}(\mb{C}^d)$, called the limit set, is a $\rho$-invariant topological circle. Generalizing Bowen's result, we prove:

\begin{thmA}\label{thmINTRO:Hff}
Let $\rho:\pi_1(\Sigma)\to{\rm PSL}(d,\mb{C})$ be a {\rm fully hyperconvex} representation with {\rm limit set} $\Lambda\subset{\calF}(\mb{C}^d)$. Then ${\rm Hdim}(\Lambda)= 1$ if and only if $\rho$ is conjugate into ${\rm PSL}(d,\mb{R})$.
\end{thmA}

The fact that the limit set $\Lambda\subset\mc{F}(\mb{R}^d)$ of hyperconvex representations in ${\rm PSL}(d,\mb{R})$ has always Hausdorff dimension 1 comes from work of Pozzetti-Sambarino-Wienhard \cite{PSW:HDim_hyperconvex}. A {\em local} rigidity statement for the Hausdorff dimension of hyperconvex representations in ${\rm PSL}(d,\mb{C})$ near the Hitchin locus was proven by Bridgeman-Pozzetti-Sambarino-Wienhard \cite{BPSW22} with completely different techniques. Our result is instead of global nature.

Special examples of fully hyperconvex representations are given by any (small perturbation of a) {\em Hitchin homomorphism} $\rho:\pi_1(\Sigma)\to{\rm PSL}(d,\mb{R})<{\rm PSL}(d,\mb{C})$. Recall that these are arbitrary continuous deformations of the composition of a Fuchsian representation $\pi_1(\Sigma)\to{\rm PSL}(2,\mb{R})$ with the irreducible embedding ${\rm PSL}(2,\mb{R})\to{\rm PSL}(d,\mb{R})$. By work of Labourie \cite{Labourie:Anosov} and Guichard \cite{Guichard}, Hitchin representations are characterized by more refined hyperconvexity assumptions, which are implied by our hyperconvexity assumptions if $d=3$. This implies that in this low dimension Theorem \ref{thm:IntroReal} can be strengthened to say that ${\rm Hdim}(\Lambda)= 1$ if and only if $\rho$ is a Hitchin homomorphism. 

One of the main novelties of our approach is the fact that we bring back tools from 2-dimensional quasi-conformal analysis and 3-dimensional hyperbolic geometry to the realm of Anosov representations contributing to the problem raised in \cite[Section 13]{WieICM}. Key to the proof is the notion of a {\em hyperbolic surface lamination} which is a generalization of a { hyperbolic surface}. The relevance of hyperbolic surface laminations in complex dynamics and Anosov representations of surface groups has been highlighted by the pioneering work of Sullivan \cite{Sullivan} and  Tholozan \cite{Tholozan:notes}.

Roughly speaking, a hyperbolic surface lamination is a compact space with a local product structure $U\times X$ where $U$ is a subset of  $\mb{H}^2$ and $X$ is a transverse space. A basic example is the following: choose a torsion-free cocompact Fuchsian realization $\Gamma<{\rm PSL}(2,\mb{R})$ of $\pi_1(\Sigma)$. $\Gamma$ acts on $\mb{H}^2\times\partial\mb{H}^2$ by homeomorphism preserving the product structure and isometrically with respect to the slices $\mb{H}^2\times\{t\}$. The quotient $M_\Gamma=\mb{H}^2\times\partial\mb{H}^2/\Gamma$ is a hyperbolic surface lamination whose leaves, the projections of the slices $\mb{H}^2\times\{t\}$, can either be isometric copies of $\mb{H}^2$ or annuli $L_{[\gamma]}=\mb{H}^2/\langle\gamma\rangle\times\{\gamma^+\}$ corresponding to conjugacy classes $[\gamma]$ of primitive elements of $\Gamma$ (here $\gamma^+$ is the attracting fixed point of $\gamma$ on $\partial\mb{H}^2$). 

We briefly outline the proof of Theorem \ref{thmINTRO:Hff}. By \cite{PSW:HDim_hyperconvex} and \cite{PSamb}, we can interpret the Hausdorff dimension of $\Lambda$ as 
\[
{\rm Hdim}(\Lambda)=\max_{1\le k\le d-1}\{h^k(\rho)\}
\]
where 
\[
h^k(\rho):=\limsup_{R\to\infty}\frac{\log|\{[\gamma]\text{ conjugacy class of }\pi_1(\Sigma)|\,\log|L_\rho^k(\gamma)|\le R\}|}{R}
\]
is the $k$-th root entropy of the representation $\rho$, defined as follow. Every element $\rho(\gamma)$ is diagonalizable with eigenvalues $\lambda^1,\cdots,\lambda^d\in\mb{C}$ ordered in decreasing order according to their modulus, and we denote by $L_\rho^k(\gamma)=\lambda^{k}(\rho(\gamma))/\lambda^{k+1}(\rho(\gamma))$ the $k$-th eigenvalue gap.

\begin{thmA}
\label{thm:IntroReal}
Let $\rho:\pi_1(\Sigma)\to{\rm PSL}(d,\mb{C})$ be  $k$-hyperconvex. If $h^k(\rho)=1$ then the $k$-th root spectrum of $\rho$ is real.
\end{thmA}

Thus, if $h^k(\rho)=1$ for every $k$, the $k$-th root spectrum of $\rho$ is real for every $k$. We show (Theorem \ref{thm:real gaps real}) that this implies that $\rho$ must be conjugate into ${\rm PSL}(d,\mb{R})$.

The proof of Theorem \ref{thm:IntroReal} is inspired by work of Deroin-Tholozan \cite{DT16}. Indeed, we show that unless the $k$-th root spectrum of $\rho$ is real, we can find a hyperbolic surface lamination $E$ (topologically isomorphic to $M_\Gamma$) and a number $\kappa>1$ such that we have the following strict domination
\[
\ell_E(\cdot)\ge\kappa\cdot\log|L_\rho^k(\cdot)|.
\]
The conclusion then follows from the fact that the entropy of a hyperbolic Riemann surface lamination is bounded below by one, a result that was proven by Tholozan in \cite{Tholozan:notes} and of which we give an independent elementary proof.

\begin{thmA}
Let $E$ be a hyperbolic surface lamination topologically isomorphic to $M_\Gamma$. Then
\[
h(E):=\limsup_{R\to\infty}{\frac{\log|\{L_{[\gamma]}\text{ \rm annular leaf }|\,\ell_E(\gamma)\le R\}|}{R}}\ge 1.
\]
\end{thmA}
\noindent Hence, directly from the definitions, we get $h^k(\rho)\ge\kappa\cdot h(E)\ge\kappa>1$. 

The main observation behind our proof is that the hyperbolic surface lamination $E$ has a natural leaf-preserving flow which exponentially expands the Lebesgue measure on the leaves. Using this property, we can estimate the topological entropy (which coincides with the exponential growth rate of the closed orbits) using simple ideas from Manning \cite{Manning:entropy}.

We construct the hyperbolic Riemann surface lamination $E$ by doing foliated hyperbolic geometry extracting effective estimates from the work of Benoist--Hulin \cite{BH17} on the existence of harmonic maps at bounded distance from isometries. More specifically, we first construct, using hyperconvexity, an action of $\pi_1(\Sigma)$ on $\mb{CP}^1\times\partial\pi_1(\Sigma)$ with the following properties: 
\begin{itemize}
    \item{Every element $\gamma$ preserves $\mb{CP}^1\times\{\gamma^+\}$ (here $\gamma^+\in\partial\pi_1(\Sigma)$ is the attracting fixed point for $\gamma$) on which it acts as a loxodromic transformation with dilation $\lambda^{k}/\lambda^{k+1}(\rho(\gamma))$. }
    \item{There is an equivariant family of uniform marked quasicircles $\xi_t:\partial\pi_1(\Sigma)\to\Lambda_t\subset\mb{CP}^1\times\{t\}$, one for every $t\in\partial\pi_1(\Sigma)$.}
\end{itemize} 

Then, using work of Benoist--Hulin \cite{BH17}, we produce an equivariant family of harmonic fillings $f_t:\tilde{\Sigma}\to\mb{H}^3$ of $\xi_t$. Pulling back the metric of $\mb{H}^3$ via $f_t$ we produce an invariant metric on $\tilde{\Sigma}\times\partial\pi_1(\Sigma)$. By construction, such a metric 1-dominates $\lambda^{k}/\lambda^{k+1}$. We conclude by showing that the Riemann uniformization strictly dominates the metric. 

Our second main contribution connects hyperconvex representations to Teichmüller theory providing a natural holomorphic extension of the Ahlfors--Bers map. Recall that every quasi-Fuchsian representation leaves invariant a Jordan curve $\Lambda\subset\mb{CP}^1$ whose complement $\Omega\sqcup\Omega'=\mb{CP}^1-\Lambda$ is a union of two topological disks. The action on $\Omega,\Omega'$ is properly discontinuous, free, and biholomorphic so that the quotients $E=\Omega/\rho(\pi_1(\Sigma))$ and $F=\Omega'/\rho(\pi_1(\Sigma))$ are two Riemann surfaces that come equipped with an isomorphism between their fundamental groups and $\pi_1(\Sigma)$. The classical Ahlfors--Bers map is the map ${\rm AB}:\rho\to (E,F)\in\T(\Sigma)\times\T(\overline \Sigma)$.  Bers' Simultaneous Uniformization \cite{Bers60} establishes that ${\rm AB}$ is a biholomorphism between the space of quasi-Fuchsian representations up to conjugacy and $\T(\Sigma)\times\T(\overline \Sigma)$.

In the context of hyperconvex representations we consider the Teichmüller space $\T(M_\Gamma)$ of abstract marked hyperbolic surface laminations $h:M_\Gamma\to E$ ($h$ is a homeomorphism sending leaves to leaves; see \S\ref{s:surlam}) up to the Teichmüller equivalence relation that identifies $E$ and $F$ when there exists a homeomorphism $\phi:E\to F$ sending leaves to leaves isometrically and such that $h\phi$ is homotopic to $h'$ (through maps that send leaves to leaves). The classical Teichmüller space $\T(\mb{H}^2/\Gamma)$ canonically embeds in $\T(M_\Gamma)$.
The following structural result is due to Sullivan.

\begin{thm*}[{\cite{Sullivan}}]
The space $\T(M_\Gamma)$ has a natural structure of complex Banach manifold, realized as a bounded domain in a complex Banach space. 
\end{thm*}

\begin{thmA}
\label{thm:IntroAB}
Let $\Xi^k(\G,\PSL(d,\C))$ be the space of $k$-hyperconvex representations. For every $1\le  k\le d-1$, there is a canonical $k$-th Ahlfors--Bers map 
\[
{\rm AB}^k:\Chark\to\T(M_\Gamma)\times\T(\overline M_\Gamma).
\]
with the following properties:
\begin{enumerate}
\item{${\rm AB}^k$ is holomorphic.}
\item{If ${\rm AB}^k(\rho)=(E^k_\rho,F^k_\rho)$, then $2\min\{\ell_{E^k_\rho}(\cdot),\ell_{F^k_\rho}(\cdot)\}\ge\log|L^k_\rho(\cdot)|$.}
\item{When $\rho$ is the composition $\rho=\iota_d\circ \sigma$ of a quasi-Fuchsian representation $\sigma:\Gamma\to{\rm PSL}(2,\mb{C})$ with the irreducible embedding $\iota_d:{\rm PSL}(2,\mb{C})\to{\rm PSL}(d,\mb{C})$, the Ahlfors--Bers parameters ${\rm AB}^k(\rho)$ coincide with the classical ones of $\sigma$.}
\end{enumerate}
\end{thmA}

Property (2) generalizes the classical Ahlfors' Lemma (see \cite[Lemma 5.1.1]{Otal:hyperbolization}) which is important to control the geometry of the deformations of quasi-Fuchsian representations. As above, $L_\rho^k(\cdot)$ is the $k$-th root length spectrum of $\rho$. 

When a representation is \emph{fully hyperconvex}, namely $k$-hyperconvex for all $1\leq k\leq d-1$, combining all the Ahlfors--Bers parameters ${\rm AB}^k$ we get a map 
\[
{\rm AB}=({\rm AB}^1,\cdots,{\rm AB}^{d-1}).
\]

\begin{thmA}
\label{thm:ab intro 2}
Let $\Charn$ be the space of fully hyperconvex representations. Then
\[
{\rm AB}:=({\rm AB}^1,\cdots,{\rm AB}^{d-1}):\Charn\to(\T(M_\Gamma)\times\T(\overline M_\Gamma))^{d-1}
\]
is holomorphic, injective, and closed. Moreover  
\[
{\rm AB}^{-1}(\Delta\times\cdots\times\Delta)=\Xi^{\rm hyp}(\G,\PSL(d,\R))
\]
where $\Delta\subset\T(M_\Gamma)\times\T(\overline M_\Gamma)$ is the (diagonal) set of pairs which are mirror images of one another.
\end{thmA}

Let us briefly sketch the main steps of the proof, beginning with injectivity. It follows from our construction that the Ahlfors--Bers parameter ${\rm AB}^j(\rho)$ captures the $j$-th root gap $\lambda^j/\lambda^{j+1}$ of  the matrices $\rho(\gamma)\in{\rm PSL}(d,\mb{C})$. Given this input, the rest of the argument is mostly algebraic: we show that the collection of all such gaps distinguishes points in $\Charn$ (see Theorem \ref{thm:same gaps conjugate}). This rests on the fact that the trace functions distinguish points on the character variety $\Xi(\Gamma,{\rm SL}(d,\mb{C}))$.

As for closedness, we actually show a stronger property. The space $(\T(M_\Gamma)\times\T(\overline  M_\Gamma))^{d-1}$ is endowed with a natural (product) Teichmüller metric, also introduced by Sullivan \cite{Sullivan}. We prove that the pre-image under ${\rm AB}$ of a bounded set in such a metric is pre-compact in $\Charn$. This property implies closedness.

Lastly, we discuss the relation between the diagonal and real hyperconvex representations. Again, it follows from our construction that if ${\rm AB}^j(\rho)\in\Delta$, then the $j$-th root gap $\lambda^j/\lambda^{j+1}$ is real for all matrices $\rho(\gamma)\in{\rm PSL}(d,\mb{C})$. Once we observe this, again the rest of the proof is mostly algebraic. We prove that such a representation must be conjugate into ${\rm PSL}(d,\mb{R})$ (Theorem \ref{thm:real gaps real}). We use a result of Acosta \cite{Aco19}  that if all the trace functions of a representation in ${\rm SL}(d,\mb{C})$ are real then the representation is conjugate in ${\rm SL}(d,\mb{R})$ or ${\rm SL}(d/2,{\bf H})$ (where ${\bf H}$ are the real quaternions). 

We conclude with a couple of words on the complex geometry of $\Charn$. As bounded domains in complex Banach spaces (such as $(\T(M_\Gamma)\times\T(\overline M_\Gamma))^{d-1}$, thanks to Sullivan's work) are well-studied objects, Theorems \ref{thm:IntroAB} and \ref{thm:ab intro 2} allow us to deduce that the space of hyperconvex representations has many good properties from the point of view of its complex geometry. For example, it is Kobayashi hyperbolic and the Kobayashi metric is complete.

\subsection*{Acknowledgements}
We thank Yves Benoist, Dominique Hulin, Peter Smillie, and Nicolas Tholozan for useful discussions. This work was funded through the DFG Emmy Noether project 427903332 of Beatrice Pozzetti. Beatrice Pozzetti acknowledges additional support by the DFG under Germany’s Excellence Strategy EXC-2181/1-390900948.
James Farre acknowledges support from DFG – Project-ID 281071066 – TRR 191. Gabriele Viaggi acknowledges the support of the funding RM123188D816D4A5 of the Sapienza University of Rome. 

\section{Riemann surface laminations}
\label{sec:rsl}

In this section, we introduce smooth hyperbolic and Riemann surface laminations.
We discuss a Riemann surface lamination structure $M_\Gamma$ on $T^1S$ whose leaves are the weakly unstable leaves for the geodesic flow associated to $\bH^2/\Gamma$, a hyperbolic surface homeomorphic to a closed surface $S$.

The main result of this section is Theorem \ref{thm: orbit growth rate}:  the growth rate of closed orbits of a certain (geodesic) flow defined on a smooth (hyperbolic) Riemann surface lamination smoothly equivalent to $M_\Gamma$ is bounded below by $1$.
Theorem \ref{thm: orbit growth rate} follows from a geometric bound on the topological entropy of this flow (Theorem \ref{t.Tholozan}) after a technical excursion into thermodynamical formalism, relegated to Appendix \ref{sec:appendix thermo}.

Theorem \ref{thm: orbit growth rate} was established in unpublished notes of Tholozan \cite[Theorem 0.4]{Tholozan:notes} using different but related techniques.  
Our proof supplies a geometric alternative to some of the more sophisticated dynamical arguments given in \cite{Tholozan:notes}.
 
\subsection{Surface laminations}\label{s:surlam}
A surface lamination is a topological space that locally looks like the product of a topological surface and a transversal metric space. 
We begin by discussing the basic definitions of surface laminations with additional structure, mostly following Candel \cite[Section 1]{Candel}.

\begin{dfn}[Surface lamination]
A ({\em Riemann} or {\em hyperbolic}) {\em surface lamination} is a topological space $M$ together with an atlas of charts
\[
\calA=\{u_\alpha: U_\alpha\to D_\alpha\times T_\alpha\}_{\alpha\in J}
\]
which are homeomorphisms from open subsets $U_\alpha\subset M$ to the product of an open subset $D_\alpha\subset\H^2$ of the hyperbolic plane and an open subset $T_\alpha$ of a compact metric space  whose change of charts have expression, choosing coordinates $(z,t)$ on $D_\alpha \times T_{\alpha}$,
\[
u_\beta\circ u_\alpha^{-1}(z,t)=(f_{\alpha\beta}(z,t),g_{\alpha\beta}(t))
\]
where: 
\begin{itemize}
    \item{The component $f_{\alpha\beta}$ is a ({\em holomorphic} or {\em isometric}) {\em diffeomorphism} for every $t\in T_\alpha$.}
    \item{In the transverse direction we have that $t\mapsto \frac{\partial^nf_{\alpha\beta}}{\partial x^j\partial y^{n-j}}(x+iy,t)$ is continuous for all $n\ge 0$ and $z=x+iy$.}
\end{itemize}
\end{dfn}

By the structure of the change of charts, the fibers $D_\alpha\times\{t\}\subset D_\alpha\times T_\alpha$ glue together in a unique way to form a decomposition of $M$ into a disjoint union of connected (Riemann or hyperbolic) surfaces, the {\em leaves} of the lamination.

\begin{dfn}[Morphisms of laminations]\label{d.autsurlam}
A {\em morphism} of surface laminations is a continuous map $f:M\to M'$ that respects the lamination structure, that is, if $u:U\subset M\to D\times T$ and $u':U'\subset M'\to D'\times T'$ are lamination charts, then $u'\circ f\circ u\inverse(z,t)=(h(z,t),q(t))$. Furthermore, we say that: 
\begin{itemize}
\item{It is {\em smooth} if for every $t\in T$ the map $h(\cdot,t):D\to D'$ is smooth and its derivatives vary continuously in $t\in T$.}
\item{If $M,M'$ are Riemann surface laminations then $f$ is {\em holomorphic} if for every $t\in T$ the map $h(\cdot,t):D\to D'$ is holomorphic and its derivatives vary continuously in $t\in T$.}
\item{If $M,M'$ are hyperbolic surface laminations then $f$ is {\em locally isometric} if for every $t\in T$ the map $h(\cdot,t):D\to D'$ is locally isometric and its derivatives vary continuously in $t\in T$.}
\end{itemize} 

An {\em isomorphism} between (Riemann or hyperbolic) surface laminations is a smooth (holomorphic or locally isometric) morphism with a smooth (holomorphic or locally isometric) inverse.  
\end{dfn}

A Riemannian metric $g$ on a smooth surface lamination is a Riemannian metric on the leaves that varies smoothly in the transversal direction.
A Riemannian metric $g$ on a Riemann surface lamination is \emph{conformal} if the conformal structure induced by $g$ on the leaves coincides with the leafwise complex structure.
The following theorem of Candel is an analog of the classical uniformization theorem in the smooth, laminated setting.
\begin{thm}[\cite{Candel}]\label{thm: candel}
Every Riemann surface lamination has a conformal  Riemannian metric $g$ with constant curvature.
\end{thm}

\subsection{Laminations from hyperbolic surfaces}

We will be interested in a special type of Riemann surface laminations coming from hyperbolic surfaces. 

\begin{dfn}[Hyperbolic surface]
A closed orientable {\em hyperbolic surface} is a quotient $\mb{H}^2/\Gamma$ of the hyperbolic plane $\mb{H}^2$ by a discrete, torsion free, cocompact subgroup $\Gamma<{\rm Isom}^+(\mb{H}^2)={\rm PSL}(2,\mb{R})$.
\end{dfn}

The diagonal action
$\Gamma\curvearrowright\mb{H}^2\times\partial\mb{H}^2$
is free and properly discontinuous as it is so on the first factor. Since it preserves the product structure, the quotient 
\[M_\Gamma:=\mb{H}^2\times\partial\mb{H}^2/\Gamma\] is a Riemann surface lamination whose leaves,  the projection to $M_\Gamma$ of the sets $\mb{H}^2\times\{t\}$, are of two types:
\begin{itemize}
\item{If $t\in\partial\mb{H}^2$ is the fixed point of some  hyperbolic element $\gamma\in\Gamma$, then $\mb{H}^2\times\{t\}$ is stabilized by the cyclic group $\langle\gamma\rangle$. The image of the projection in $M_\Gamma$ is the hyperbolic annulus $\mb{H}^2\times\{t\}/\langle\gamma\rangle$ whose core geodesic has length $\ell(\gamma)$.}
\item{If $t\in\partial\mb{H}^2$ is not one of the (countably) many fixed points of primitive hyperbolic elements of $\Gamma$, then $\mb{H}^2\times\{t\}$ has trivial stabilizer and the projection to $M_\Gamma$ is injective so that the corresponding leaf is a copy of $\mb{H}^2$.}
\end{itemize}
Every leaf of $M_\Gamma$ is dense since $\Gamma$ acts {\em minimally} on $\partial\mb{H}^2$, namely with dense orbits.
\begin{dfn}[Associated lamination]\label{d.MGamma}
We call $M_\Gamma$ the {\em Riemann surface lamination associated with the hyperbolic surface} $\mb{H}^2/\Gamma$. 
\end{dfn}

Note that the unit tangent bundle $T^1\bH^2/\Gamma$ has a \emph{weakly unstable} foliation for the geodesic flow, where the leaf passing through the point $(p,v)$ consists of the points $(q,w)$ whose backward geodesic trajectories are asymptotic.
There is a smooth map  
\[\Phi:T^1\mb{H}^2/\Gamma\to M_\Gamma\] 
defined by $\Phi(p,v)=[(\tilde p,t)]$, where $\tilde p\in\mb{H}^2$ is a lift of $p$ and $t\in\partial\mb{H}^2$ is the point at infinity of the corresponding lift of the geodesic ray $\gamma:[0,\infty)\to\mb{H}^2/\Gamma$ starting at $p$ with velocity $-v$.
Then $\Phi$ is an isomorphism of Riemann surface laminations that identifies the leaf of the unstable foliation for the geodesic flow corresponding to $t$ with  $[\mb{H}^2\times\{t\}]\subset M_\Gamma$.

\begin{dfn}[Marking]\label{def: marking}
    Let $W$ be a Riemann surface lamination.  A smooth lamination equivalence $f: M_\Gamma \to W$ that is leafwise orientation preserving is called a \emph{marking} or \emph{marked Riemann surface lamination}.
\end{dfn}

\subsection{Entropy}\label{subsec: entropy}

Let $f: M_\Gamma \to W$ be a marked Riemann surface lamination.
Using Candel's Theorem \ref{thm: candel}, $W$ has a conformal Riemannian metric $g_W$ (in the category of laminations) with constant curvature equal to $-1$, which we call the Poincar\'e metric.

Consider the covers $\H^2 \times \partial \H^2$ and $\widetilde W$ corresponding to $\Gamma$ and $f_*\Gamma$, respectively.
Then $f$ lifts to a smooth lamination map $\tilde f$, which is equivariant with respect to the covering actions by conformal lamination automorphisms.
Topologically, $\widetilde W$ is a disk bundle $\zeta$ over $\partial \Gamma$, and $g_W$ gives each disk fiber its Poincar\'e metric.  

\begin{lemma}\label{lem: uniform bi-Lipschitz on leaves}
    The map $\tilde f $ is uniformly bi-Lipschitz on leaves.
    The ideal boundary of $\zeta\inverse (s)$ is canonically identified with $\partial \Gamma$, and is pointed by $s \in \partial \Gamma$.
\end{lemma} 
\begin{proof}
    Since $f$ is smooth, it is bi-Lipschitz on compact subsets of leaves.
    Every leaf of $M_\Gamma$ is dense and the holomorphic structure on the leaves is transversely constant.
    Compactness implies that $f$ is uniformly bi-Lipschitz on every leaf.
    Thus the map $\tilde f$ on leaves extends continuously to a bi-H\"older map between ideal boundaries of leaves, and the ideal boundary of every leaf of $M_\Gamma$ is canonically identified with $\partial \Gamma$. 
\end{proof}

We define a continuous, leaf-preserving flow $\psi= \{\psi_t\}$ on $W$ using the Poincar\'e metric $g_W$ as follows.
The leaf $\zeta\inverse (s)$ is a copy of the hyperbolic plane, foliated by oriented geodesics emanating from $s$ on its ideal boundary.
The leaves of $W$ are similarly foliated by oriented complete $g_W$-geodesics.
We define $\psi_t(z,s)$ as the point in the same fiber obtained by traveling distance $t$ from $z$ in the future along the relevant geodesic line.

\begin{dfn}[Entropy]\label{def:entropy RS}
    The \emph{entropy} $H(W)$ of the marked hyperbolic surface lamination $f: M_\Gamma \to W$ is the topological entropy of the flow $\psi$ defined above.
\end{dfn}

The topological entropy $H(\phi)$ of a continuous flow $\phi$ on a compact metric space $(X,d)$ can be computed as follows:
A $(t,\delta)$-separated set $Q\subset X$ is a finite subset such that for every $x,x'\in Q$ there exists $s\in[0,t]$ such that $d(\phi_s(x),\phi_s(x'))\ge\delta$. Denote by $N(t,\delta)$ the maximal cardinality of a $(t,\delta)$-separated set. The topological entropy is computed by
\begin{equation}\label{eqn: separated sets entropy}
    H(\phi)=\lim_{\delta\to 0^+}\limsup_{t\to\infty}\frac{\log N(t,\delta)}{t}.
\end{equation}
The topological entropy of a flow does not depend on the choice of metric on $X$ inducing the topology (see \cite[\S3]{Manning:entropy} or \cite[\S9]{VO:ergodic}).
\medskip

The following theorem states, in our setting, that the entropy is bounded from below by 1.
Manning used a similar argument to show that the topological entropy of the geodesic flow of a compact Riemannian manifold is bounded below by the \emph{volume entropy}, i.e., the exponential growth rate of metric balls in the universal cover of the manifold \cite[Theorem 1]{Manning:entropy}.

\begin{thm}\label{t.Tholozan}
For every marked hyperbolic surface lamination $f:M_\Gamma\to W$ we have $H(W)\ge1$.
\end{thm}

\begin{proof}
The hyperbolic metric on $\H^2/\Gamma$ induces a natural  \emph{Sasaki Riemannian metric} $h_\Gamma$ on $M_\Gamma \cong T^1\H^2/\Gamma$. The Sasaki metric has the property that the restriction to leaves of the unstable foliation is a Riemannian metric $g_\Gamma$ in the category of laminations and the tangent projection $M_\Gamma \to \H^2/\Gamma$ is a locally isometric covering on each leaf. In other words, $g_\Gamma$ is the laminated Poincar\'e metric on $M_\Gamma$.

We can compute entropy with respect to any distance on $W$ that induces its topology, e.g., for the distance $d$ associated with $f_* h_\Gamma$.
Denote by $B(z,r)$ the $d$-ball of radius $r$ about $z\in W$.

By Lemma \ref{lem: uniform bi-Lipschitz on leaves}, the restriction $d_L$ of $d$ to a leaf $L$ of $W$ is (uniformly) $C$-bi-Lipschitz equivalent to the Poincar\'e metric $g_W$ on $L$.
Denote by $B_L(q,r)$ be the $d_L$-disk of radius $r$ around $q \in L$.
Note that $d_L$ is a hyperbolic metric and has a naturally associated notion of area coming from $f_*h_\Gamma$ (equivalently $f_*g_\Gamma$).

Let $\epsilon>0$ be given, and find a positive $\delta$ much smaller than the injectivity radius of $(W,d)$ such that 
\[H(W) \ge \limsup_{t\to \infty} \frac{\log |N(t,\delta)|}{t}-\epsilon \]
holds.

Now we work in the cover $\widetilde W$ corresponding to $f_* \Gamma$ with metric $\tilde d$ and flow $\tilde \psi$.
Choose $L\subset \widetilde W$, a leaf with distinguished point $\xi$ on the ideal boundary, and assume that $L$ maps injectively into $W$ under the covering projection.
We identify $L$ with the upper half plane $\H^2$ using $g_W$ and place $\xi$ at infinity.
In these coordinates, we have $\tilde \psi_t(x+iy) = x+ie^{-t}y \in L$.

Consider a box $\tau = \{x +i y : x \in [0,1] \text{ and } y \in [1,2]\}\subset L$ so that 
\[\tilde \psi_t(\tau) = \{x+iy: x\in [0,1] \text{ and } y \in [e^{-t}, 2e^{-t}]\}.\]
A computation shows that the $g_W$-area  of $\tilde \psi_t(\tau)$ is $e^t/2$, so the  $d_L$-area (coming from $f_*h_\Gamma$ or equivalently $f_*g_\Gamma$) satisfies
\[\frac{e^t}{2C^2}\le \Area(\tilde \psi_t(\tau)) \le \frac{C^2e^t}2.\]

Let $Q_t'\subset \tilde \psi_t(\tau)$ be a maximal set of points such that $\tilde d(q_i, q_j) \ge \delta$ for all $q_i\not = q_j \in Q_t'$.  
Then 
\[ \bigcup_{q_i \in Q_t'}B(q_i,\delta) \subset \tilde \psi_t(\tau) \subset \bigcup_{q_i \in Q_t'}B(q_i,2\delta). \]

The inclusion of $L$ into $\widetilde W$ is $1$-Lipschitz, so we have
\[B(q,r) \cap L \supset B_L(q,r)\]
for all $q \in L$.
By continuity of $g_W$ and compactness of $W$, there is a $\delta'\ge 2\delta$ such that $d(p,q)<2\delta$ implies that $d_L(p,q)<\delta'$ for all $p,q\in L$.
In particular,
we have 
\[B(q,2\delta)\cap L \subset B_L(q,\delta')\]
for all $q \in L$, and hence
\[\tilde \psi_t(\tau) \subset \bigcup _{q_i\in Q_t'}B_L(q_i,\delta').\]

Thus 
\[\sum \Area(B_L(q_i,\delta')) \ge \Area \tilde \psi_t (\tau).\]
Since the area of a hyperbolic disk (in the $d_L$ metric) of radius $\delta'$ is $\pi\sinh(\delta')$, 
we obtain 
\begin{equation}\label{eqn: lower bound Q_t'}
    |Q_t'| \ge \frac{e^t}{2C^2\pi\sinh(\delta')}.
\end{equation}

Let $Q_t$ be the projection of $\psi_{-t}(Q_t')$ to $W$, and note that $|Q_t'| = |Q_t|$.
We claim that $Q_t$ is $(t,\delta)$ separated for $\psi$.
Indeed, suppose points $q_i$ and $q_j \in Q_t$ satisfy $d(\psi_s(q_i), \psi_s(q_j))<\delta$ for all $s\in [0,t]$.
Since $\delta$ is smaller than the injectivity radius of $W$, the corresponding paths $s\in [0,t]\mapsto \tilde \psi_s(\tilde q_i)$ and $s\in [0,t]\mapsto \tilde \psi_s (\tilde q_j)$ stay $\delta$ close in $L$.
Since pairs of points in $Q_t'$ are $\delta$-separated and $\tilde \psi_t (\tilde q_i), \tilde \psi_t (\tilde q_j) \in Q_t'$, it follows then that $q_i = q_j$, which proves the claim.

Equation \eqref{eqn: lower bound Q_t'} then gives a bound $N(\delta, t) \ge |Q_t| \ge e^t/2C^2\pi\sinh(\delta')$.
In conclusion, we have
\[
H(W)=H(\psi)\ge\limsup_{t\to\infty}{\frac{\log|Q_t|}{t}}\ge 1-\ep.
\]
As $\ep$ is arbitrary, the theorem follows.
\end{proof}

\subsection{Length spectrum and orbital growth rate}
Every marked hyperbolic surface lamination $f:M_\Gamma\to W$ has an associated length spectrum. Denote by $[\Gamma]$ the set of conjugacy classes of elements in $\Gamma$.

\begin{dfn}[Length spectrum]\label{d.reallength}
The {\em length spectrum} of $f:M_\Gamma\to W$ is the function $\ell_W:[\Gamma]-\{1\}\to(0,\infty)$ that associates to an element $[\gamma]\in[\Gamma]-\{1\}$ the length of the closed geodesic of the leaf $f([\mb{H}^2\times\{\gamma^+\}/\langle\gamma\rangle])$ where $\gamma^+\in\partial\mb{H}^2$ is the {\em attracting} fixed point of $\gamma$.
\end{dfn}

By definition, the periods of the closed orbits of $\psi$ correspond exactly to the (marked) length spectrum.
Note that $\ell_W(\gamma)$ can be different from $\ell_W(\gamma^{-1})$, and that $\ell_W(\gamma)$ can be read as the logarithm of the ratio of the eigenvalues of any matrix representing the projective action of $\gamma$ on the leaf of $W$ corresponding to $\bH^2 \times \{\gamma^+\}$.

We will deduce the following theorem from Theorem \ref{t.Tholozan} by observing that the topological entropy $H(W)$ coincides with the exponential growth rate of closed orbits for the flow $\psi$ defined in the previous subsection.

\begin{thm}\label{thm: orbit growth rate}
Let $f:M_\Gamma\to W$ be a marked hyperbolic surface lamination. Then
\[
{h}(W):=\limsup_{R\to\infty}{\frac{\log \left|\{[\gamma]\in [\Gamma]\left|\,\ell_W(\gamma)\le R\right.\}\right|}{R}}\ge 1.
\]
\end{thm}

The proof of the theorem is somewhat technical, borrowing some tools from thermodynamical formalism.  It is carried out in Appendix \ref{sec:appendix thermo}.

\medskip
Here is an outline of the argument:
The first step is to observe that $\psi$ is conjugate to a reparameterization of the geodesic flow $\phi$ on $M_\Gamma \cong T^1\H^2/\Gamma$ (Lemmas \ref{lem: orbit equivalence} and \ref{lem: reparam conj}).

Classical results of Bowen and Pollicott imply that the orbital growth rate $h$ and the topological entropy $H$ coincide when $\psi$ is a H\"older reparameterization of $\phi$.
We show that $h$ and $H$ are continuous in the reparameterization potential (\S\ref{subsec: entropy of reparam}) and conclude the proof by density of H\"older reparameterizations in continuous reparameterizations.

\section{Teichmüller spaces of Riemann surface laminations}

In this section we discuss Sullivan's construction \cite{Sullivan} of the Teichm\"uller space $\T(M_\Gamma)$ of the Riemann surface lamination $M_\Gamma$.
We review some features and tools from classical Teichm\"uller theory, and spend some time discussing (smooth) quasi-conformal maps, Beltrami differentials, and the Measurable Riemann Mapping Theorem \ref{thm:mrm}.
We conclude by explaining how a foliated analog of Bers' embedding (due to Sullivan) gives $\T(M_\Gamma)$ a natural complex structure that makes it biholomorphic to a bounded domain in a complex Banach space.

\subsection{Teichm\"uller equivalence}
We fix a closed hyperbolic surface $\Sigma :=\H^2/\Gamma$ and consider the associated Riemann surface lamination $M_\Gamma$ (Definition \ref{d.MGamma}). 

Recall from Definition \ref{def: marking} that a marking $f: M_\Gamma \to W$ of a Riemann surface lamination $W$ is a leafwise orientation preserving smooth lamination equivalence.

\begin{dfn}[Teichmüller equivalence relation]
We say that two marked Riemann surface laminations $f:M_\Gamma\to W$ and $f':M_\Gamma\to W'$ are {\em Teichmüller equivalent} if there exists an isomorphism of Riemann surface laminations $\phi:W\to W'$ such that $\phi \circ f$ is {\em leafwise homotopic} to $f'$. This means that there is a continuous map $H:M_\Gamma\times[0,1]\to W'$ such that:
\begin{itemize}
    \item{$H(\cdot,0)= \phi\circ f $ and $H(\cdot,1)=f'$.}
    \item{For every $t\in[0,1]$, the map $H(\cdot,t):M_\Gamma\to W'$ is a morphism of surface laminations.}
\end{itemize}

Denote by $\T(M_\Gamma)$ the  Teichm\"uller set of equivalence classes of marked Riemann surface laminations $f:M_\Gamma\to W$.
\end{dfn}

The \emph{classical} Teichm\"uller space $\T(\Sigma)$ consists of equivalence classes of \emph{marked Riemann surfaces}, i.e., orientation preserving diffeomorphisms $f: \Sigma \to \Sigma'$ where $\Sigma'$ is a Riemann surface.\footnote{Since $\Sigma$ is closed, any orientation preserving diffeomorphism is quasi-conformal.}  Two diffeomorphisms $f_i : \Sigma \to \Sigma_i$, $i =1,2$ are \emph{Teichm\"uller equivalent} if there is a biholomorphism $\phi: \Sigma_1 \to \Sigma_2$ such that $\phi\circ f_1$ is homotopic to $f_2$.

\begin{prop}\label{prop: classical includes}
There is an inclusion 
$\T(\Sigma)\to \T(M_\Gamma).$
\end{prop}
\begin{proof}
Let $f: \Sigma \to \Sigma'$ be a marked Riemann surface.
Then $\Sigma' \cong \H^2/\Gamma'$ for a discrete group $\Gamma' = f_*\Gamma \le \PSL(2,\R)$.
Choose a lift 
\[\widetilde f : \H^2 \to \H^2.\]
Then $\widetilde f$ is smooth and bi-Lipschitz, hence extends to a $(\Gamma, \Gamma')$-equivariant homeomorphism 
\[\partial \widetilde f: \partial \H^2 \to \partial \H^2.\]
The map 
\[(\widetilde f, \partial \widetilde f):\H^2 \times \partial \H^2 \to \H^2 \times \partial \H^2\]
is a $(\Gamma, \Gamma')$-equivariant equivalence of smooth surface laminations.
Since $\Gamma$ and $\Gamma'$ act by Riemann surface automorphisms on $\H^2 \times \partial \H^2$, we obtain a smooth lamination equivalence
\[F: M_\Gamma \to M_{\Gamma'} \]
on the orbit spaces.

To see that this assignment respects the Teichm\"uller equivalence relations, suppose that $\phi: \Sigma_1 \to \Sigma_2$ is a bi-holomorphism and $f_1: \Sigma \to \Sigma_1$ and $f_2 : \Sigma \to \Sigma_2$ are markings such that $\phi\circ f_1 \sim f_2$. 
As in the first paragraph, we obtain a corresponding maps 
\[F_i: M_\Gamma \to M_{\Gamma_i},\]
where $\Sigma_i = \H^2/\Gamma_i$ for $i = 1$, $2$.
By a similar argument, we obtain a (leafwise conformal) Riemann surface lamination equivalence 
\[\Phi: M_{\Gamma'} \to M_{\Gamma'},\]
and our claim is that $\Phi \circ F_1$ and $F_2$ are homotopic as lamination maps.
The proof is an exercise in covering space theory and is left to the reader.
\end{proof}

The symmetry group of $M_\Gamma$ as a smooth surface lamination defines an action on the Teichm\"uller set.

\begin{dfn}[Mapping class group of $M_\Gamma$]
We define $\Mod(M_\Gamma)$ as the group of leaf-preserving homotopy classes of orientation preserving smooth lamination equivalences $M_\Gamma \to M_\Gamma$.
\end{dfn}

The group $\Mod(M_\Gamma)$ acts on $\T(M_\Gamma)$ as follows. If $[\phi]\in{\rm Mod}(M_\Gamma)$ and $[f:M_\Gamma\to W]\in\T(M_\Gamma)$ then we set
\[
[\phi]\cdot[f:M_\Gamma\to W]:=[f\circ\phi^{-1}:M_\Gamma\to W].
\]
It is a routine check to see that the action is well-defined.

Similarly, the mapping class group $\Mod(\Sigma)$ can be defined as the group of homotopy classes of orientation preserving diffeomorphisms $\Sigma \to \Sigma$ \cite{FM:primer}.
\begin{lem}\label{lem: classical Mod includes}
There is a canonical injective homomorphism ${\rm Mod}(\Sigma)\to{\rm Mod}(M_\Gamma)$.    
\end{lem}

\begin{proof}
The proof is essentially contained in the proof of Proposition \ref{prop: classical includes}.
Briefly, suppose $\phi: \Sigma \to \Sigma$ is an orientation preserving diffeomorphism.
Choose a lift $\phi: \H^2 \to \H^2$ with boundary extension $\partial \phi$.
Then $\widetilde \phi \times \partial \widetilde \phi$ is $\Gamma$-equivariant and so descends to a smooth lamination equivalence 
\[\Phi: M_\Gamma \to M_\Gamma,\]
as before.
That homotopic maps $\phi_1$ and $\phi_2$ give rise to leafwise homotopic lamination mappings $\Phi_1$ and $\Phi_2$ is an exercise in covering space theory.
\end{proof}

\subsection{Smooth quasi-conformal maps and Teichm\"uller distance}

For an orientation preserving smooth diffeomorphism $f : U \to V $ between domains $U ,V \subset \CP$, the \emph{complex dilatation} $\mu_f$ at a point $z\in U$ is 
\begin{equation}\label{eqn: complex dilitation}
    \mu_f(z) = \frac{\partial_{\overline z}f(z)}{\partial_z f(z)},
\end{equation}
or more compactly
\[\mu_f = \frac{f_{\overline z}}{f_z}.\]

The \emph{conformal distortion} of $f: U \to V$ at $z \in U$  is defined as
\[K_f(z) = \frac{1+|\mu_f(z)|}{1-|\mu_f(z)|}.\]
This is the ratio between the sizes of the major and minor axes of the ellipse $f_* C \subset T_{f(z)}V$, where $C$ is a circle in $T_zU$.
This quantity is invariant under pre- and post-composition by conformal mappings and satisfies $1\le K_f(z) <\infty$. 

An orientation preserving smooth homeomorphism $\phi: X \to Y$ between Riemann surfaces is {\em $K$-quasi-conformal} if 
\[K_\phi  =\sup_{z \in \Sigma} K_\phi(z) \le K.\]
Similarly, an orientation preserving smooth lamination equivalence 
$\phi: W \to W'$ 
of Riemann surface laminations is $K$-quasi-conformal if is leafwise $K$-quasi-conformal.
Denote by $K_\phi$ the supremum of quasi-conformal distortion of $\phi$ over the leaves of $W$.  

\begin{dfn}[Teichmüller distance]\label{d.Teichdis}
The {\em Teichmüller distance} $d_\T(\cdot,\cdot)$ on $\T(M_\Gamma)$ is defined to be
\[
d_\T\left([f:M_\Gamma\to W],[f':M_\Gamma\to W']\right):=\frac{1}{2}\inf\{\log(K_\phi)\left|\,\phi:W\to W', \phi\circ  f\sim f'\right.\}
\]
where the infimum is taken over all leafwise orientation preserving smooth lamination equivalences $\phi:W\to W'$ such that $\phi \circ  f$ is leafwise homotopic to $f'$.
\end{dfn}
That $d_\T$ is a distance is immediate from basic properties of the conformal distortion, namely, that $K_{\phi} \equiv 1$ if and only if  $\phi$ is conformal, $K_{f\inverse}(f(z)) = K_f(z)$, and $K_{f\circ g}(z) \le K_f(g(z))K_g(z)$.

Observe that the classical Teichm\"uller space $\T(\Sigma)$, equipped with its Teichm\"uller metric, includes (Proposition \ref{prop: classical includes}) as a totally geodesic subspace of $\T(M_\Gamma)$.

\subsection{Smooth Beltrami differentials}
Suppose $f$ and $g$ are orientation preserving smooth maps between domains in $\CP$ on which $g\circ f$ is defined.
We record here a useful formula for the complex dilatation for composite maps \cite[Chapter I.C]{Ahlfors:lectures}
\begin{equation}\label{eqn: beltrami composite}
    \mu_g\circ f  ~\frac{\overline {f_z}}{f_z} =\frac{\mu_{g\circ f} - \mu_f}{1-\overline{\mu_f}\mu_{g\circ f}}.
\end{equation}
If $g$ is conformal, we obtain
\[\mu_{g\circ f} = \mu_f, \]
and if $f$ is conformal, we get
\[\mu_g \circ f ~\frac{\overline {f'}}{f'} = \mu_{g\circ f}.\]
In particular, if $g: U\to V$ is equivariant with respect to (discrete, torsion free) groups $G\le \Aut(U)$ and $G' \le \Aut(V)$, then 
for all $\gamma \in G$,
 we have 
\begin{equation}\label{eqn: invariant Beltrami}
    \mu_g \circ \gamma ~\frac{\overline {\gamma'}}{\gamma'} = \mu_g.
\end{equation}
Thus, the expression 
\[\mu_g \frac{d\overline z}{dz}\] 
is invariant by $G$ and defines a $(-1, 1)$-differential form on the Riemann surface $U/G$.

For a smooth map $\phi: X \to Y$ 
between Riemann surfaces, 
there is a $(-1,1)$-\emph{Beltrami differential} $\mu_\phi d\overline z/ dz$ on $X$ recording the complex dilatation of $\phi$ as in \eqref{eqn: complex dilitation}.

In a similar fashion, we can define the Beltrami differential associated with a smooth map of Riemann surface laminations.
Indeed, a marked Riemann surface lamination $f: M_\Gamma \to W$ defines a {\em leafwise Beltrami differential} that can be written in local coordinates as
\begin{equation}\label{eqn: laminated Beltrami differential}
\mu_f(\cdot,t):=\frac{\partial f/\partial\overline{z}}{\partial f/\partial z}(\cdot,t)\frac{d\overline z}{dz}.    
\end{equation}
The rules \eqref{eqn: beltrami composite} apply leafwise for composite maps.

\subsection{Quasi-conformal homeomorphisms}\label{subsec: qc homeo}
The class of smooth quasi-conformal maps is too small for our purposes. 
Ahlfors gives a geometric definition of quasi-conformal homeomorphisms as orientation preserving homeomorphisms that distort the conformal modulus of quadrilaterals by a bounded multiplicative factor \cite{Ahlfors:lectures}.
Here is another (equivalent) formulation of quasi-conformality.
\begin{dfn}[Quasi-conformal homeomorphism]
\label{dfn:quasi-conformal}
An orientation preserving homeomorphism $f: U \to V$ between domains $U$ and $V\subset \CP$ is $K$-quasi-conformal for some constant $K\ge 1$ if in all affine charts it satisfies the following 
\begin{equation}\label{eqn: def qc}
   \limsup_{r\to 0}{\frac{\sup_{|z-w|=r}{|f(z)-f(w)|}}{\inf_{|z-w|=r}{|f(z)-f(w)|}}}\le K
\end{equation}
for every $z$ in the affine chart. If the domain $K\subset \CP$ of $f$ is instead an arbitrary subset $K\subset \CP$ we say that $f$ is \emph{$K$-quasi-M\"obius}.
\end{dfn}

A $1$-quasi-conformal map is conformal \cite[Chapter II.A]{Ahlfors:lectures}.
The space of normalized quasi-conformal homeomorphisms of $\CP$ is compact:
any accumulation point of a sequence of $K$-quasi-conformal homeomorphisms of $\CP$ fixing three distinct points is a $K$-quasi-conformal homeomorphism \cite[Chapter II.C]{Ahlfors:lectures}.

\begin{dfn}[Quasi-circle]\label{def: quasicircle}
    A \emph{$K$-quasi-circle} $\mathcal C \subset \CP$ is the image of $\RP$ under a $K$-quasi-conformal homeomorphism $\CP \to \CP$.
    A \emph{marking} of $\mathcal C$ is a homeomorphism $\RP \to \mathcal C$.
\end{dfn} 

We will need the following removability criterion for quasi-conformal maps. See  \cite[\S8.3]{LehtoVirtanen}.
\begin{prop}\label{prop: quasicircle removable}
    Suppose $\mathcal C_1 \subset \CP$ and $\mathcal C_2\subset \CP$ are quasi-circles and 
    \[f: \CP\setminus \mathcal C_1 \to \CP\setminus \mathcal C_2\] is $K$-quasi-conformal and extends continuously to homeomorphism $g: \CP \to \CP$.
    Then $g$ is a $K$-quasi-conformal homeomorphism.
\end{prop}

An important result in the theory of quasi-conformal maps asserts that quasi-conformal solutions $f: \CP \to \CP$ to the \emph{Beltrami differential equation}
\begin{equation}\label{eqn: beltrami}
    f_{\overline z} \mu = f_z
\end{equation}
exist, are essentially unique, and vary nicely in $\mu$.
The \emph{Beltrami coefficient} 
\[\mu: \CP \to \mathbb D\] 
in \eqref{eqn: beltrami} is only required to be defined on a full measure set, i.e., $\mu \in L^\infty (\CP)$  and $\|\mu\|_\infty <1$.

The following is known as the Measurable Riemnann Mapping Theorem; see \cite[Chapter V]{Ahlfors:lectures}.

\begin{thm}[Measurable Riemann Mapping]
\label{thm:mrm}
For every essentially bounded Beltrami coefficient $\mu\in L^\infty(\mb{CP}^1)$ with $\|\mu\|_\infty<1$, there exists a unique quasi-conformal map $g^\mu:\mb{CP}^1\to\mb{CP}^1$ fixing $0$, $1$, and $\infty$ such that $\mu_{g^\mu}=\mu$, a.e. Furthermore, $g^\mu$ depends holomorphically\footnote{Holomorphic dependence of solutions $g^\mu$ on $\mu$ means the following.
First, a map \[t\in \mathbb D \mapsto \mu(t) \in L^\infty(\CP)\] 
is holomorphic if for all $t \in \mathbb D$ and sufficiently small $s \in \mathbb D$, we have
    \[\mu(s+t)(z) = \mu(t)(z) +s \nu(t)(z) + s\epsilon(s,t)(z) \]
    for some $\nu(t), \epsilon(s,t) \in L^\infty(\CP)$ with $\|\epsilon(s,t)\|_\infty \to 0$ as $s\to 0$.
    In other words, for almost every $z \in \CP$, the map $t\mapsto \mu(t)(z)$ is holomorphic.} on $\mu$.
\end{thm}
In particular, if $\mu_n\to \mu$ is a convergent sequence of Beltrami coefficients in $L^\infty(\CP)$, then $g^{\mu_n}\to g^\mu$ uniformly on $\CP$ (\cite[Chapter V.B]{Ahlfors:lectures}).
Theorem \ref{thm:mrm} asserts that if $\mu: \mathbb D \to L^\infty (\CP)$ is holomorphic with $\|\mu(t)\|_\infty \le k <1$ for all $t$, then for all $z\in \CP$,  \[t \in \mathbb D\mapsto g^{\mu(t)}(z) \in \CP\] is holomorphic.

\subsection{Bers' embedding and complex structure}\label{sec:Bersembedding}
Bers constructed a biholomorphism from $\T(\Sigma)$ to a bounded domain in the complex vector space  $Q(\overline \Sigma)$ of holomorphic quadratic differentials on $\overline{\Sigma}$.
We recall Bers' construction and explain Sullivan's observation that uniform continuity of solutions of the Beltrami equation \eqref{eqn: beltrami} allow us to carry out a similar construction to embed $\T(M_\Gamma)$ as a bounded domain in a complex Banach space \cite{Sullivan}.

\subsubsection{Holomorphic quadratic differentials}

Let $Q(\Sigma)$ be the complex vector space of holomorphic quadratic differentials on $\Sigma$.
Lifting to $\H^2$,  $q \in Q(\Sigma)$ is the data of a holomorphic function $\varphi: \H^2 \to \C$ satisfying 
\[\varphi(\gamma(z)) \gamma'(z)^2 = \varphi(z)\] for all $z\in \H^2$ and $ \gamma \in \Gamma$, i.e., $q$ is a holomorphic $(2,0)$-differential form on $\Sigma$.

Denote by $\rho(z) |dz|^2$ the Poincar\'e metric on $ \H^2$.  Then the ratio $|\varphi|/\rho$ is invariant by $\Gamma$, hence descends to a function on $\Sigma$. 
We endow $Q(\Sigma)$ with the norm\footnote{Some authors appear to prefer to divide by $\rho/4$, rather than $\rho$; in particular, the bound  \eqref{eqn: nehari krauss} and in \cite{Sullivan} differ by a factor of $4$.} 
\[\|q\|_\infty = \sup \frac{|\varphi(z)|}{\rho(z)}.\]

There is a natural analog of leafwise holomorphic quadratic differential in the category of Riemann surface laminations:
An element $q \in Q(M_\Gamma)$ is a holomorphic $(2,0)$-differential form on the leaves of $M_\Gamma$ that varies transversely continuously.  Note that $C^0$-transverse continuity and holomorphicity guarantee that all (complex) derivatives automatically vary transversely continuously. 
We similarly endow the infinite dimensional $\C$-vector space $Q(M_\Gamma)$ with the leafwise $\|\cdot\|_\infty$-norm obtained by dividing by the Poincar\'e metric on the leaves of $M_\Gamma$.
Then $Q(M_\Gamma)$ is a complex Banach space with this norm.

\subsubsection{Classical Bers' embedding}
Our discussion of Bers' embedding follows  \cite[\S6.1]{IT:ItroTeich}.
Throughout, we denote by $\overline \Sigma$ and $\overline \H^2$  the mirror images of $\Sigma$ and $\H^2$, respectively.

To an orientation preserving diffeomorphism $f: \Sigma \to \Sigma'$, we assign  $\beta(f) \in Q(\overline \Sigma)$, as follows.
Choose a lift $\tilde f: \H^2 \to \H^2$, and define
\[\mu = \begin{cases}
\mu_{\tilde f}(z), & z \in \H^2\\
0, & z \in \overline \H^2
\end{cases},\]
so that $\mu\frac{d\overline z}{dz}$ is $\Gamma$-invariant. 
The normalized solution $g^\mu$ of the Beltrami equation \eqref{eqn: beltrami} is quasi-conformal on $\CP$ and conformal on $\overline \H^2$.
Let $G^\mu: \overline \H^2 \to g^\mu(\overline \H^2)$ be the restriction of $g^\mu$ to $\overline \H^2$.
 
We associate to $G^\mu$ a holomorphic quadratic differential on $\overline \Sigma$ via the \emph{Schwarzian derivative}, which is defined, for a conformal mapping $f : U \to V$, by 
\begin{equation}\label{eqn: def schwarz}
    \mathcal S(f) (z) = \frac{f'''(z)}{f'(z)} - \frac32 \left( \frac{f''(z)}{f'(z)}\right)^2.
\end{equation}
Using $\Gamma$-invariance of $\mu\frac{d\overline z}{dz}$, a computation verifies that
\[\mathcal S(G^\mu)(\gamma(z)) \gamma'(z)^2 = \mathcal S (G^\mu)(z)\]
for all $\gamma \in \Gamma$ and $z \in \overline \H^2$.
Thus $\mathcal S(G^\mu)$ defines a holomorphic quadratic differential $\beta(f) \in Q(\overline \Sigma)$.

The assignment $f\mapsto \beta(f)$ \emph{respects the Teichm\"uller equivalence relation}.
Thus $\beta$ defines a continuous injection 
\[\beta: \T(\Sigma) \to Q(\overline \Sigma) \cong \C^{3g-3}\]
called \emph{Bers' embedding}.

Classical estimates on the Schwarzian derivative for univalent functions due to Nehari and Krauss give 
\begin{equation}\label{eqn: nehari krauss}
    \|\beta([f])\|_\infty \le 3/2,
\end{equation}
for all $[f: \Sigma \to \Sigma'] \in \T(\Sigma)$; see \cite[\S6.1.4]{IT:ItroTeich}.

The complex structure on $\T(\Sigma)$ may be defined by pulling back the complex structure along $\beta$:
for a different basepoint $\Sigma'$ and Bers' embedding $\beta' : \T(\Sigma) \to Q(\overline \Sigma ')$, 
the map $\beta'\circ \beta\inverse: \beta(\T(\Sigma) ) \to \beta'(\T(\Sigma))$ is a bi-holomorphism \cite[\S6.2]{IT:ItroTeich}

\subsubsection{Laminated Bers' embedding}
We conclude the section outlining Sullivan's construction of a complex structure on $\T(M_\Gamma)$ via a laminated analog of Bers' embedding.

Let $f: M_\Gamma \to W$ be a marked Riemann surface lamination with Beltrami differential $\mu_f$ defined as in \eqref{eqn: laminated Beltrami differential}.
We lift $\mu_f$ to a $\Gamma$-invariant leafwise Beltrami coefficient on $\mb{H}^2\times\partial\mb{H}^2$ which we still denote by $\mu_f$. Extend $\mu_f$ to a leafwise Beltrami coefficient $\mu$ defined on the Riemann surface lamination $\mb{CP}^1\times\partial\mb{H}^2$ by setting it to be 0 on $\overline \H^2\times\partial\mb{H}^2$. 

Solving the Beltrami equation 
\[
\frac{\partial g}{\partial\overline{z}}(\cdot,t)=\mu(\cdot,t)\frac{\partial g}{\partial z}(\cdot,t)
\]
leaf by leaf using Theorem \ref{thm:mrm} gives a quasi-conformal continuous lamination equivalence 
\[g^\mu: \CP \times \partial \H^2 \to \CP \times \partial \H^2.\]

The restriction $G^\mu$ of $g^\mu$ to the lower hemisphere $\overline \H^2 \times\partial\mb{H}^2$ is leafwise {\em holomorphic} as $\mu \equiv 0$ on $\overline \H^2 \times\partial\mb{H}^2$. 
By continuity of solutions to the Beltrami differential equation, the leafwise Schwarzian derivative $\beta(f) : =\mathcal S(G^\mu)$ defines a holomorphic quadratic differential on $\overline M_\Gamma$ (the Riemann surface lamination $M_\Gamma$ with the orientation of each leaf reversed).  
As before, $\|\beta(f)\|_\infty \le 3/2$.

We record here some more of the properties, due to Sullivan \cite[\S5]{Sullivan}, of the construction outlined above.

\begin{thm}\label{thm: Sullivan Bers embedding}
    The assignment $f: M_\Gamma \to W \mapsto \beta(f)\in Q(\overline M_\Gamma)$ respects the Teichm\"uller equivalence relation, hence defines a map 
    \[\beta : \T(M_\Gamma) \to Q(\overline M_\Gamma)\]
    called the \emph{Bers embedding}
    which is continuous with respect to the topology on $\T(M_\Gamma)$ defined by the Teichm\"uller distance (Definition \ref{d.Teichdis}).  Moreover, $\beta$ is injective with bounded image and $\beta(\T(M_\Gamma))$ contains $\{q \in Q(\overline M_\Gamma) : \|q \|_\infty < 1/2\}$.
\end{thm}

There is a Bers embedding $\beta_W : \T(M_\Gamma) \to  Q(\overline W)$ associated to any point $[f: M_\Gamma \to W] \in \T(M_\Gamma)$.
The following theorem provides a complex analytic structure on $\T(M_\Gamma)$; again see \cite[\S5]{Sullivan}.
\begin{thm}\label{thm: Bers holomorphic}
    Bers' embedding $\beta$ is holomorphic, i.e., 
    \[\beta_W \circ \beta\inverse: \beta(\T(M_\Gamma)) \to \beta_W (\T(M_\Gamma))\]
    is a bi-holomorphism for any $[f: M_\Gamma \to W] \in \T(M_\Gamma)$.
\end{thm}

Finally, we point out that every leaf $L \subset M_\Gamma$ is dense in $M_\Gamma$. The restriction map recording the conformal structure on $L$ up to bounded homotopy defines a continuous injection $\T(M_\Gamma) \to \T(L)$ \cite[\S3]{Sullivan}.

\section{Laminated quasi-Fuchsian theory}

In this section we develop a laminated quasi-Fuchsian theory for laminated actions of a surface group $\Gamma$ on $\mb{CP}^1\times\partial\mb{H}^2$ that parallels in many aspects the classical one. Every such action comes together with a pair of laminated Ahlfors--Bers parameters that are a pair of marked Riemann surface laminations. Generalizing the classical results of quasi-Fuchsian theory, we show that any pair of parameters is realized (Theorem \ref{thm:ab parametrization}). We then consider the complex dilation spectrum of a laminated action and prove two properties: a generalization of Ahlfors' Lemma (Proposition \ref{p.bound}) and a characterization of the preimage of the diagonal (Proposition \ref{prop:diag real}).

\subsection{Automorphisms of $\CP$-laminations}
We consider the Riemann surface lamination $\fCP$, which is, in particular, also a surface lamination in the smooth and topological categories.

Denote by $\Aut^0(\fCP)$ the group of continuous lamination automorphisms of $\fCP$.
Recall that $g\in \Aut^0(\fCP)$ is a homeomorphism 
\[g : \fCP\to \fCP,\] 
which has the form 
\begin{equation}\label{eqn:automorphism}
g(z,t) = (h_t(z), f(t)).   
\end{equation} 
For each $t$, it holds that
\begin{itemize}
    \item $h_t:\CP\to \CP$ is a homeomorphism;
    \item the maps $h_t$ vary continuously in $t$; and
    \item $f: \dH\to \dH$ is a homeomorphism.
\end{itemize} 

We also consider subgroups that preserve more structure.
Namely, $\Aut^\infty(\fCP)$ is the group of smooth lamination automorphisms, where in \eqref{eqn:automorphism}, for each $t$, we require $h_t$ to be smooth and vary continuously in the $C^\infty$ topology. See Definition \ref{d.autsurlam}.

\begin{dfn}[Laminated M\"obius group]
    The subgroup $\MG\le \Aut^\infty(\fCP)$ preserving the structure of $\fCP$ as a Riemann surface lamination and where the map $f$ in \eqref{eqn:automorphism} is M\"obius is called the \emph{laminated M\"obius group}.
\end{dfn}
An element $g\in \mathcal {MG}$ is of the form $g(z,t)= (h_t(z), f(t))$, where $h_\cdot: \partial \mb H^2 \to \Aut (\mb{CP}^1) \cong \PSL(2,\mb C)$  is continuous, and $f\in \PSL(2,\mb R)$.
The choice of the standard $\RP \subset \CP$ fixes a diagonal action of $\PSL(2,\bR)$ on $\fCP$, which induces an embedding 
\[\iota:\PSL(2,\bR)\hookrightarrow\MG.\]
By fiat, $\PSL(2,\bR)$ is identified with its image in
$\MG$.
Summarizing, we have
\begin{equation}\label{eqn:group inclusions}
    \PSL(2,\bR)\le \MG\le \Aut^\infty(\fCP)\le \Aut^0(\fCP).
\end{equation}
We will be interested in the space 
$$\Hom(\Gamma, \MG)$$
of \emph{laminated conformal actions}: these give rise to actions of $\G$ on the Riemann surface lamination $\fCP$ which are conformal when restricted to leaves. Restricting the standard embedding $\iota$ to $\G$ we get an example of one such action.

\subsection{Quasi-conformal deformations}
Let $\Gamma \le \PSL(2,\bR)$ be a uniform lattice, so that $\bH^2/\Gamma$ is a closed hyperbolic surface, and let $M_\Gamma=\quotient{\mb{H}^2\times\partial\mb{H}^2}{\Gamma}$ its associated Riemann surface lamination (see Definition \ref{d.MGamma}). 

We call \emph{standard laminated limit set} the torus
\begin{equation}\label{eqn:fL}
    \fL = \RP\times \dH \subset \fCP,
\end{equation} 
 which is preserved by $\iota(\Gamma)$.
Then $\iota(\Gamma)$ also preserves the disconnected Riemann surface lamination $\fCPL$, acting by Riemann surface lamination automorphisms.

Each component of 
\[\fCPLmG\]
can be identified with $M_\Gamma$ by an isomorphism of smooth surface laminations that is leafwise either holomorphic or anti-holomorphic.

We would like to consider deformations of \[\Gamma \hookrightarrow \PSL(2,\bR) \xrightarrow{\iota} \MG\] by laminated quasi-conformal maps.

\begin{dfn}[Laminated quasi-conformal map]\label{def:qc map}
    A map $g \in \Aut^0(\fCP)$ is $K$-quasi-conformal if
    \begin{itemize}
        \item For all $t\in \dH$, the map $z \mapsto h_t(z)$ is $K$-quasi-conformal.
        \item If $g(z,t) = (h_t(z), f(t))$, then $f$ is M\"obius.
    \end{itemize}
\end{dfn}

In order to stay inside of the category of (pairs of) smooth Riemann surface laminations, we will restrict our attention to conjugations of $\iota$ by certain special laminated quasi-conformal homeomorphisms that are smooth outside of a (quasi)-circle in each leaf.

\begin{dfn}[Quasi-conformal deformation]\label{def:qc deformation}
    A homomorphism $\rho: \Gamma \to \MG$ is a \emph{$K$-quasi-conformal deformation} of $\iota: \Gamma \to \MG$ if there is a  $K$-quasi-conformal $g \in \Aut^0(\fCP)$ such that $\rho = g \iota g\inverse$ and the restriction 
    \[g:\fCPL\to\fCPgL\] 
    is an isomorphism of \emph{smooth} surface laminations. 
\end{dfn}
While it follows from the definition that the restriction $g|_{\fCPL}$ of a quasi-conformal deformation to the complement of the laminated limit set is smooth and all its derivatives depend continuously on the transversal direction, we can, in general, only expect that the restriction of $g$ to the laminated limit set is continuous. 

We will want to consider quasi-conformal deformations only up to suitable equivalence.
Say that $\rho_1$ is \emph{equivalent} to $\rho_2$ if there is a $g\in \MG$ satisfying $\rho_2 = g\rho_1g\inverse$.

\begin{dfn}[Quasi-conformal deformation space]
    Denote by 
    \[
    \mathcal {QC}(\iota)\subset \Hom(\Gamma, \MG)/\MG
    \]
    the \emph{quasi-conformal deformation space of $\iota$}, i.e.,
$[\rho]\in \mathcal {QC}(\iota)$ if $\rho$ is a $K$-quasi-conformal deformation of $\iota$ for some $K\ge 1$. 
\end{dfn}

Note that $\Hom(\Gamma, \MG)$ is equipped with the compact-open topology, $\Hom(\Gamma, \MG)/\MG$ has the quotient topology, and $\mathcal {QC}(\iota)$ is equipped with the subspace topology.

\begin{prop}
    Let $\mathcal {QF}$ denote the classical quasi-conformal deformation space of $\Gamma \hookrightarrow \PSL(2,\bR)\hookrightarrow\PSL(2,\bC)$.
    There is a continuous embedding $\mathcal {QF}\to \mathcal {QC}(\iota)$.
\end{prop}

\begin{proof}
 The classical quasi-Fuschian space is realized as an open subset of the character variety $\Hom(\Gamma, \PSL(2,\bC))/\!\!/\PSL(2,\bC)$.
    Let $[\rho]$ be as such; by definition, there is a quasi-conformal homeomorphism $g_1$ of $\CP$ satisfying $\rho(\Gamma) = g_1\Gamma g_1\inverse \le \PSL(2,\bC)$.
    We may assume that $g_1$ is smooth outside of $\mathbb {RP}^1$.\footnote{See also \S\ref{s.invAB}, where we do this in higher generality.}
    
    Then $\widehat g_1 : \CP \times \dH \to \fCP$ defined by $\widehat g_1 (z,t) = (g_1(z), t)$ is a laminated quasi-conformal homeomorphsim that restricts to  an isomorphism of smooth surface laminations 
    \[\fCPL \to \CP \times \dH-g_1(\fL).\]
    In particular, 
    \[\widehat \rho : = \widehat g_1\iota\widehat g_1\inverse: \Gamma \to \MG\]
    is a quasi-conformal deformation of $\iota$ (in the sense of Definition \ref{def:qc deformation}), and 
    \[\widehat \rho (\gamma) (z,t) = (\rho(\gamma)z, \gamma t),\]
    for $\gamma \in \Gamma\le \PSL(2,\bR)$.
    
    For $g_2\in \PSL(2,\bC)$, we have 
    \[\widehat {g_2\rho g_2\inverse}(\gamma)(z,t) = (g_2\rho(\gamma)g_2\inverse z, \gamma t).\]
    Thus $\PSL(2,\bC)$ maps into $\MG$, acting trivially on the circle factor and $\widehat \rho$ is conjugated to $\widehat {g_2\rho g_2\inverse}$ by  $\widehat g_2\in \MG_0$. So the map $[\rho] \mapsto [\widehat \rho]$ is well-defined.

    To see that this assignment is injective, we suppose that $[\widehat \rho_1] = [\widehat \rho_2]$.  Then there exists $g \in \MG$ of the form $g(z,t) = (h_t(z), f (t))$ conjugating $\widehat \rho_1$ to $\widehat \rho_2$ where $f \in \PSL(2,\bR)$.
    The action $\widehat \rho_1$   on the circle factor is $t \mapsto \gamma t$, while the $g\widehat \rho_2 g\inverse$ action on the circle factor is $t\mapsto f \gamma f\inverse t$ for all $\gamma \in \Gamma$.
    This implies that $f$ commutes with $\Gamma$ and is hence trivial.
    
    Then 
    \[g \widehat \rho_2 g\inverse(\gamma) (z,t) = (h_{\gamma t} \rho_2(\gamma) h_t\inverse z, \gamma t),\]
    while 
    \[\widehat \rho_1(\gamma)(z,t) = (\rho_1(\gamma)z, \gamma t).\]
    Taking $t = \gamma^+$, the attracting fixed point of $\gamma$, we see that 
    \[ h_{\gamma^+} \rho_2(\gamma) h_{\gamma^+}\inverse = \rho_1(\gamma)\]
    for all $\gamma\in \Gamma$.
    Since $h_{\gamma^+}\in \PSL(2,\bC)$, we have that $\tr^2(\rho_1(\gamma)) = \tr^2(\rho_2(\gamma))$ for all $\gamma\in \Gamma$.
    This implies that $\rho_1$ and $\rho_2$ are conjugate in $\PSL(2,\bC)$.

    Continuity follows immediately from the definitions.
    This completes the proof.
\end{proof}

We will see in Section \ref{sec:marking} that hyperconvex representations $\Gamma \to \PSL(d,\C)$ give rise to exotic quasi-conformal deformations of $\iota$.

\begin{remark}
    Our proof of injectivity of the Ahlfors--Bers parameters associated with certain \emph{fully hyperconvex} representations follows the same strategy of the proof given above essentially by recovering eigenvalue ratios (hence traces).  See \S\ref{subsubsection: injectivity}.
\end{remark}

\subsection{Ahlfors--Bers parameters and double uniformization}\label{s.ABuniversal}

From a quasi-conformal deformation $\rho=g\iota g^{-1}$, we extract a pair $(E_\rho, F_\rho)$ of marked Riemann surface laminations $M_\Gamma\to E_\rho$ and $\overline M_\Gamma \to F_\rho$. 

Namely, $\fCPL$ consists of two components $\bH^2\times \dH$ and $\overline \bH^2\times \dH$.
By definition, the laminated quasi-conformal map $g$ restricts to a pair of smooth lamination isomorphisms
\[\bH^2\times \dH \to g(\bH^2\times \dH)
\]and 
\[\overline\bH^2\times \dH \to g(\overline \bH^2\times \dH),\]
and the image of $\rho = g\iota g\inverse$ preserves the images acting by Riemann surface lamination automorphisms.

Define Riemann surface laminations 
\[E_\rho := g(\bH^2\times \dH)/\rho(\Gamma) \]
and 
\[ F_\rho: = g(\overline \bH^2\times \dH)/\rho(\Gamma).\]
Since $g\iota (\gamma) = \rho(\gamma) g$ for all $\gamma \in \Gamma$,  $g$ induces smooth markings \[M_\Gamma \to E_\rho\] 
and 
\[ \overline M_\Gamma \to F_\rho.\]
Thus $[ M_\Gamma \to E_\rho]$ defines a point in $\mathcal T( M_\Gamma)$ and similarly $[ \overline M_\Gamma \to E_\rho]\in \mathcal T(\overline M_\Gamma )$.
\begin{dfn}[Ahlfors--Bers parameters]
    The \emph{Ahlfors--Bers} parameters associated to $[\rho]\in \mathcal {QC}(\iota)$ are 
    \[\mathcal {AB}([\rho]) :=([E_\rho], [F_\rho]) \in  \T(M_\Gamma)\times\T(\overline{M}_\Gamma).\]
\end{dfn}

We must verify that our definition respects the corresponding equivalence relations.
\begin{prop}
    The map
    \[\mathcal{AB}:\mathcal {QC}(\iota) \to \T(M_\Gamma)\times\T(\overline{M}_\Gamma)\]
    is well-defined.
\end{prop}

\begin{proof}
Consider two equivalent $\rho_1,\rho_2\in{\rm Hom}(\Gamma,\mc{MG})$. By definition $\rho_j=g_j\iota g_j^{-1}$ for $j=1,2$ with $g_j$ quasi-conformal (as in Definition \ref{def:qc deformation}) and, as $\rho_1$ is equivalent to $\rho_2$, the composition $h:=g_2g_1^{-1}:\mb{CP}^1\times\partial\mb{H}^2\to\mb{CP}^1\times\partial\mb{H}^2$ belongs to $\mc{MG}$. Notice that $h$ is $(\rho_1,\rho_2)$-equivariant, conformal on each $\mb{CP}^1\times\{t\}$, and restricts to an isomorphism of Riemann surface laminations  
\[h:g_1(\fCPL)\to g_2(\fCPL).\] 
Hence, it induces an isomorphism between $M_\Gamma\sqcup\overline{M}_\Gamma\to E_{\rho_1}\sqcup F_{\rho_1}$ and $M_\Gamma\sqcup\overline{M}_\Gamma\to E_{\rho_2}\sqcup F_{\rho_2}$ in the right homotopy class with respect to the markings induced by $g_1,g_2$.
\end{proof}

Generalizing Bers' Simultaneous Uniformization \cite{Bers60}, we have the following:

\begin{thm}
\label{thm:ab parametrization}
Let $\mb{H}^2/\Gamma$ be a closed hyperbolic surface with associated lamination $M_\Gamma$.  There is a natural, i.e., mapping class group equivariant, homeomorphism
\[
\mc{AB}:\mc{QC}(\iota)\to\T(M_\Gamma)\times\T(\overline M_\Gamma).
\]
\end{thm}
The proof is given in \S\S\ref{subsubsec:continuity of AB} - \ref{ss.mcg} as Propositions \ref{prop:AB continuous}, \ref{p.ABinjective},  \ref{p.invAB} and \ref{p.mcg}.

\subsubsection{Continuity of $\mathcal {AB}$}\label{subsubsec:continuity of AB}
We begin the proof with an elementary observation:
\begin{lemma}\label{l.seqcontinuity}
It is enough to check sequential continuity of $\mc{AB}$.
\end{lemma}
\begin{proof}
The group $\MG$ can be endowed with a metric inducing the compact open topology.  
Then $\Hom(\Gamma, \MG)$ is metrizable by choosing a finite generating set for $\Gamma$. Infimizing the distance over representatives in a given $\MG$-conjugacy class yields a metric on $\mathcal {QC}(\iota)$, which shows the claim.
\end{proof}

 Suppose then that $[\rho_n]\to [\rho_\infty] \in \mc{QC}(\iota)\subset \Hom(\Gamma, \MG)/\MG$.
    Up to conjugating by suitable elements in $\MG$, we may assume that $\rho_n = g_n \iota g_n\inverse$, where $g_n$ is a laminated quasi-conformal homeomorphism that is a smooth lamination map away from $\fL=\mathbb {RP}^1 \times \dH$, $g_n$ is identity on $\dH$, and \[\rho_n (\gamma) \to \rho_\infty(\gamma)\]
    for all $\gamma \in \Gamma$.

\begin{lemma}\label{l.convonL}
Under the above assumptions $ g_n$ converges pointwise to $g_\infty$ on $\fL$.
\end{lemma}
\begin{proof}

    Indeed, let $\gamma\in \Gamma$ be non-trivial and choose $a\in \RP\setminus \{\g^-\}$.
    Using equivariance, we obtain
    \[\rho_n(\gamma^k)g_n(a,\gamma^-) = g_n(\gamma^ka, \gamma^k\gamma^-) \to g_n(\gamma^+, \gamma^-), ~ k\to \infty.\]
    It follows that $g_n(\gamma^+, \gamma^-)$ is the attracting fixed point for the M\"obius action of $\rho_n(\gamma)$ on the leaf $\CP\times \{\gamma^-\}$. 
    Since $\rho_n (\gamma) \to \rho_\infty (\gamma)$, 
    it follows that the attracting fixed points converge, i.e., 
    $g_n(\gamma^+, \gamma^-) \to g_\infty (\gamma^+, \gamma^-)$. 
    Using the ergodicity of the geodesic flow and the closing lemma, the set  $\{(\gamma^+, \gamma^-): \gamma \in \Gamma\}$ is dense in $\calL$. 
    We have shown the convergence $g_n\to g_\infty$ of the continuous maps on a dense subset of the compact metric space $\calL$, 
    so $g_n \to g_\infty$ (uniformly) on $\fL$.
\end{proof}
In general the quasi-conformal maps $g_n$ might not converge pointwise on the complement $\fCPL$. Our strategy will be to replace them with better behaved quasi-conformal maps $\widehat{\DE}(g_n)$ extending $g_n|_\fL$. To construct the maps $\widehat{\DE}(g_n)$, we will use the Douady--Earle extension operator:
\begin{thm}[{\cite{DE86}}]
\label{t.DE}
There is a \emph{Douady--Earle extension operator}
    \[\DE: \Homeo^+ (\dH) \to {\rm Diff}^+(\bH^2)\]
    with the following properties 
    \begin{enumerate}
        \item $\DE(f)$ extends continuously to $f$ on $\dH$.
        \item $\DE(f)$ is conformally natural, i.e., for all $\alpha, \beta\in \PSL(2,\bR)$, we have $\DE(\alpha\circ f\circ \beta) = \alpha\circ \DE(f)\circ \beta$.
        \item $\DE$ is continuous with respect to the natural topologies.
        Concretely, the map 
    \begin{align*}
    \bH^2 \times \Homeo^+(\dH) &\to \bH^2 \\
    (z, f) &\mapsto \DE(f)(z)
    \end{align*}
    is continuous, as are the derivatives in $z$ of all orders.
    \item For every $K$, there is a $K^*$ such that if $f\in \Homeo^+(\dH)$ admits a $K$-quasi-conformal extension to $\bH^2$, then $\DE(f)$ is $K^*$-quasi-conformal.
    \end{enumerate}
\end{thm}

The key step in our proof of continuity is the following laminated extension of Theorem \ref{t.DE}.
For this, we consider the space of \emph{laminated markings}
\begin{equation}\label{e.foliatedmarking}
\fM=\left\{f:\fL=\RP\times\deH\to\fCP\left|\;\begin{array}{l}
    f \text{ continuous,}\\
    f(\RP\times\{t\})\subset \CP\times \{t\}, \\
    f|_{\RP\times\{t\}} \text{ injective }\; \forall t\in\dH
\end{array} \right.\right\}.
\end{equation}
A laminated marking gives rise to a continuous family of marked Jordan curves.
\begin{prop}\label{prop:DE}
There exists an extension map 
\[\widehat{\DE}:\fM\to\Aut^0(\fCP)\]
such that, for $g \in \fM$, the extension $\widehat{\DE}(g)$ is a smooth lamination isomorphism  \[\fCPL\to \fCP-g(\fL)\] and has the following properties 
\begin{enumerate}
    \item $\widehat{\DE}(g)|_{\fL}=g$;
    \item if $g_n\to g$ pointwise, then $\widehat{\DE}(g_n)\to \widehat{\DE} (g)$ pointwise on $\fCP$;
    \item for every $K$ there exists $K^*$ such that if $g|_{\RP \times \{t\}}$ is $K$-quasi-M\"obius for every $t$, then $\widehat{\DE}(g)$ is $K^*$-quasi-conformal and smooth away from $\fL$;
    \item if $g$ is $(\iota,\rho)$ equivariant, namely $g\iota =\rho g$ for $\rho\in\Hom(\G,\MG)$, then $\widehat{\DE}(g)$ is also $(\iota,\rho)$ equivariant.
\end{enumerate}
\end{prop}

\begin{proof}
{\bf Step 1: Construction.}
 Choose distinct points $x,y,z \in \RP$. We can and will assume, up to conjugating both $g$ and $\widehat{\DE}(g)$ by an element of $\MG$ covering the identity on $\dH$ that  the points  $x, y,z \in \RP$ are fiberwise fixed, i.e., we assume that $g(x,t) = (x,t),\; g(y,t) = (y,t)$, and $g(z,t) = (z,t)$ for all $t\in \dH$.

We write $g=(h_t,t)$ and denote by $\Lambda_t =h_t(\RP)$ and by 
$H_t $ and $\overline H_t \subset \CP$  the two connected components of $\CP\setminus \Lambda_t$ (recall that for every $t$, $\Lambda_t$ is a Jordan curve). 
Let $u_t: H_t \to \bH^2$ and $\overline u_t: \overline H_t \to \overline\bH^2$ be  uniformizing maps.  By Carath\'eodory's extension theorem (e.g., \cite{GM:extension_thm}), since $\Lambda_t$ is a Jordan curve, $u_t$ and $\overline u_t$ have  continuous extensions $\partial u_t: \Lambda_t \to \dH$ and $\partial \overline u_t: \Lambda_t \to \overline\dH$.
    We will assume, up to postcomposing $u_t$ with an element of $\PSL(2,\R)$, that $\partial u_t$ fixes $x$, $y$, and $z$.
    Define 
    \[\zeta_t: \dH \cong \RP   \xrightarrow{h_t} \Lambda_t  \xrightarrow{\partial u_t} \dH, \]
    which also fixes $x$, $y$, and $z$, and $\overline \zeta_t$ analogously.

  We define $\widehat{\DE}(g)$ on $\fCPL$ by $\widehat{\DE}(g)(z,t) = (h_t'(z), t)$,
    where
    \[h_t'(z) = \begin{cases}
         u_t\inverse \circ \DE(\zeta_t)(z), & z \in \bH^2 \\
         \overline u_t\inverse \circ \DE(\overline \zeta_t)(z), & z\in \overline \bH^2.      
    \end{cases}\]

We need to verify that  $\widehat{\DE}(g)$ is indeed an element of $\Aut^0(\fCP)$, which induces a smooth lamination isomorphism  $\fCPL\to \fCP-g(\fL)$. The extension is leafwise smooth outside $\fL$ because of Theorem \ref{t.DE}. In view of Theorem \ref{t.DE} (3), order to verify the derivatives of $h'_t$ vary continuously with $t\in\deH$ it is enough to verify that, for $t\to t_0$ the holomorphic maps $u_{t}^{-1}$ and $\overline u_{t}^{-1}$ converge uniformly to  $u_{t_0}^{-1}$ and $\overline u_{t_0}^{-1}$  on compact subsets.

The last fact follows by Carath\'eodory's Convergence Theorem (see \cite[Theorem 3.1]{Duren:univalent} and note that in our setup, convergence of kernels is equivalent to Hausdorff convergence of the complementary domains): 
since $\Lambda_t \subset \CP$ are 
Jordan curves running through points $x, y$ and $z$, and $\Lambda_{t} \to \Lambda_{t_0}$ as $t\to t_0$ in the Hausdorff topology, the maps $(u_t\cup \partial u_t)\inverse : \bH^2 \cup \partial \bH^2\to \CP $ converge uniformly to  $(u_{t_0} \cup \partial u_{t_0})\inverse$ as $t\to t_0$.
    Similarly, $(\overline u_{t} \cup \partial \overline u_{t})\inverse\to (\overline u_{t_0} \cup \partial \overline u_{t_0})\inverse$ as $t\to t_0$.

    \noindent{\bf Step 2: First properties. }
    By construction,  $\widehat{\DE}(g)$ extends continuously to $\fCP$ and agrees with $g$ on $\fL$ (Property (1)). Property (2), namely the pointwise convergence, follows directly from the continuity property, Property (3), of the Douady-Earle extension operator $\DE$.
    Continuity of $\DE$ (and the maps $\zeta_t$ and $\overline \zeta_t$ in $t$) gives that $\widehat{\DE}(g)$ is a $K^*$-quasi-conformal isomorphism of smooth laminations on $\fCPL$.  
    Since $g(\RP \times \{t\})$ is a a quasi-circle for each $t$, $\widehat{\DE}(g)$ is $K^*$-quasi-conformal everywhere (see Proposition \ref{prop: quasicircle removable}), this shows Property (3).

    \noindent{\bf Step 3: Equivariance. }
        Property (4) 
        essentially boils down to conformal naturality of $\DE$ and $(\iota,\rho)$-equivariance of $g$, though the computation is somewhat lengthy.
        We write $\rho(\gamma)(z,t) = (\rho(\gamma,t) z, \gamma t)$, where $\rho(\gamma,t) \in \PSL(2,\bC)$. Since the uniformizations $u_t$ and $u_{\gamma t}$, as well as the map $\rho(\gamma,t)$, are bi-holomorphisms on their domains, 
        there is $\alpha(\gamma,t) \in \PSL(2,\bR)$ such that 
        \begin{equation}\label{eqn:SL2R cocycle}
            u_{\gamma t} \circ \rho(\gamma,t) = \alpha(\gamma,t) \circ u_t,
        \end{equation}
        which, using the $(\iota,\rho)$-equivariance of $g$ also implies
\begin{equation}\label{eqn:SL2R cocycle2}
            \zeta_{\gamma t} \circ \gamma = \alpha(\gamma,t) \circ \zeta_t.
        \end{equation}
$$\xymatrix{\RP\ar[r]_{h_t}\ar@/^1pc/[rr]^{\zeta_t}\ar[d]^{\g}&\Lambda_t\ar[r]_{\partial u_t}\ar[d]^{\rho(\g,t)}&\dH\ar[d]^{\alpha(\g,t)}\\
\RP\ar[r]^{h_{\g t}}\ar@/_1pc/[rr]_{\zeta_{\g t}}&\Lambda_{\g t}\ar[r]^{\partial u_{\g t}}&\dH}$$        
         We compute, for $z\in \bH^2$, 
        \[\rho(\gamma) \widehat{\DE}(g)(z, t) = \left(\rho(\gamma,t) u_{ t}\inverse \DE(\zeta_{ t})(z), t\right)\]
        Using \eqref{eqn:SL2R cocycle} and \eqref{eqn:SL2R cocycle2}, we have
        \begin{align*}
         \rho(\gamma,t) u_{ t}\inverse \DE(\zeta_{ t})(z)&=
         u_{\gamma t}\inverse \alpha(\gamma,t) \DE( \zeta_t)( z) \\
         &= u_{\gamma t}\inverse  \DE(\alpha(\gamma,t) \zeta_t)(z)\\
            &= u_{\gamma t}\inverse  \DE( \zeta_{\gamma t})(\gamma z)
               \end{align*}
        concluding the proof that $\rho(\g)\widehat{\DE}(g)(z,  t) =  \widehat{\DE}(g)(\gamma z,\gamma t)$, for $z\in \bH^2$.

    For $z\in \overline \bH^2$, one computes analogously as above and comes to the same conclusion.
    Finally, $\widehat{\DE}(g)$ is clearly equivariant on $\RP\times \dH$, because $g$ was, and the two maps coincide there.
    This completes the proof of the claim.
    \end{proof}

We can now conclude the proof of the main result of the section:
\begin{prop}\label{prop:AB continuous}
    $\mc{AB}$ is continuous.
\end{prop}
\begin{proof}
Thanks to Lemma \ref{l.seqcontinuity} it is enough to check sequential continuity. As above, we can assume that $\rho_n = g_n\iota g_n\inverse : \Gamma \to \MG$ are $K$-quasi-conformal deformations of $\iota$ converging to $\rho_\infty = g_\infty \iota g_\infty\inverse$, with the further assumption that $g_n,g_\infty$ are the identity on $\dH$. Thanks to Lemma \ref{l.convonL} we than know that  $g_n \to g_\infty$ pointwise on $\fL$, and up to further composing $g_n,g_\infty$ with suitable elements in $\MG$ we can assume that the points $x,y,z$ are fiberwise fixed.

Since the restrictions $G_n = g_n|_{\fL}$ and $G_\infty = g_\infty|_{\fL}$ to $\fL$ induces an element of $\fM$, we can now apply the above result to replace $g_n$ and $g_\infty$ with $\widehat{\DE}(G_n)$ and $\widehat{\DE}(G_\infty)$.
    The advantage here is that the hypothesis that $g_n \to g_\infty$ pointwise on $\RP \times \dH$ now implies that $\widehat{\DE}(G_n) \to \widehat{\DE}(G_\infty)$ pointwise on $\fCP$. 
    
    By construction $\widehat{\DE}(G_n)$ is $\rho_n$-equivariant.  Since $\widehat{\DE}(G_n)$ converge to $\widehat{\DE}(G_\infty)$ pointwise, the lamination charts for the quotient Riemann surface laminations marked by $M_\Gamma$ and $\overline M_\Gamma$ converge also.
    This shows that  $[E_{\rho_n}]\to [E_{\rho_\infty}]\in \T(M_\Gamma)$ and $[F_{\rho_n}]\to [F_{\rho_\infty}]\in \T(\overline M_\Gamma)$, as $n\to \infty$ and completes the proof of the proposition.
    \end{proof}

\subsubsection{Injectivity}

\begin{prop}\label{p.ABinjective}
    $\mathcal {AB}$ is injective.
\end{prop}
\begin{proof}
    Suppose that $\mc{AB}([\rho])=\mc{AB}([\rho'])$, where $\rho = g \rho_0 g\inverse$ and $\rho' = g' \rho_0 (g')\inverse$.
    Up to conjugation in $\MG$ we may assume that $g$ and $g'$ cover the identity on the $\partial \H^2$ factor. 
    Denote by $\bar g : M_\Gamma \sqcup \overline M_\Gamma \to E_\rho \sqcup F_\rho$ the markings induced by $g$, and define $\bar g'$ similarly.
    Note that  $g'\circ g\inverse$ restricts to a homeomorphism 
\[\partial (g'\circ g\inverse): g(\RP\times \dH) \to g'(\RP\times \dH)\]
that is $(\rho, \rho')$-equivariant and identity on $\dH$.
    
    There is a conformal equivalence of Riemann surface laminations
    \[\bar \phi: 
\quotient{(\fCPgL)}{\rho(\Gamma)}\to\quotient{(\fCPgLprime)}{\rho'(\G)}
\]
such that $\bar \phi \circ \bar g $ is leafwise homotopic to $\bar g'$. 
Then $\bar \phi$ lifts to a $(\rho, \rho')$-equivariant equivalence of Riemann surface laminations 
\[\phi : \fCPgL \to \fCPgLprime\]
which is the identity on the $\dH$-factor.

We claim that $\partial (g'\circ g\inverse)$ extends $\phi$ continuously.
Indeed, the leafwise homotopy between $g'\circ g\inverse$ and $\phi$ has uniformly bounded length trajectories in the Poincar\'e metric on each leaf of $g'(\bH^2\times \dH)$.
Thus $\partial (g'\circ g\inverse)$ induces the same map as the Carath\'eodory extension of $\phi$ on each leaf.

In summary, $\phi$ induces a continuous surface lamination map $\CP \times \dH \to \CP \times \dH$ that is trivial on $\dH$ and conformal on each leaf away from a quasicircle $g(\RP\times \{t\})$, where is it quasi-conformal.
But a quasi-conformal map that is conformal outside a quasicircle is conformal (see Proposition \ref{prop: quasicircle removable}).
Hence $\phi \in \MG$, and $\phi \circ \rho(\gamma) = \rho'(\gamma)\circ \phi$ for all $\gamma \in \Gamma$.
Hence $[\rho] = [\rho']$, completing the proof of injectivity of $\mathcal {AB}$. 
\end{proof}

\subsubsection{Conformal welding: a continuous inverse to $\mc{AB}$}\label{s.invAB}
Consider two marked Riemann surface laminations 
$$[\hat f:M_\Gamma\to E], \quad [\hat f':\overline M_\Gamma\to F];$$ abusing notation, we sometimes suppress the markings for brevity and write $[E]$ and $[F]$ for the equivalence classes of those marked laminations.
Our goal is to construct  a quasi-conformal deformation $\rho_{E,F}$ of $\iota$ satisfying 
$$\mathcal {AB}([\rho_{E,F}]) = ([ E], [ F]).$$

Fix representatives $\hat f:M_\Gamma\to E$ and $\hat f': \overline M_\Gamma \to F$.
Consider the normal covering space $\bH^2\times \dH$ of $M_\Gamma$ corresponding to $\Gamma$.
We have equivariant smooth lamination maps 
$$f:\mb{H}^2\times\partial\mb{H}^2\to U, \quad f': \overline \bH^2\times\partial\mb{H}^2 \to V$$ lifting $\hat f$ and $\hat f'$, respectively, to the corresponding covers $U \to E$ and $V\to F$.
\begin{remark}
    By Candel's Uniformization Theorem \ref{thm: candel}, $E$ and $F$ have smooth laminated Riemannian metrics in their conformal classes such that the metric on each fiber is complete with sectional curvature everywhere equal to $-1$. 
    So, we could identify $U$ and $V$ with $\bH^2\times \dH$ and $\overline \bH^2\times \dH$, respectively.
\end{remark}

We define a Beltrami coefficient $\mu_{f\sqcup f'}$ leafwise on $\fCPL$ in coordinates by
\begin{equation}\label{eqn:def beltrami}
\mu_{f\sqcup f'}(z,t):=\left\{
\begin{array}{l l}
{\displaystyle\frac{\partial f/\partial\overline{z}}{\partial f/\partial z}(\cdot,t)} &\text{\rm for $z\in\mb{H}^2$},\\
{\displaystyle \frac{\partial f'/\partial\overline{z}}{\partial f'/\partial z}(\cdot,t)} &\text{\rm for $z\in\overline \bH^2$}.\\
\end{array}
\right.
\end{equation}
Since $f$ lifts $\hat f: M_\Gamma \to E$, we have 
\begin{equation}\label{eqn:automorphic}
    \mu_{f\sqcup f'}(z,t) = \mu_{f\sqcup f'}(\gamma z, \gamma t) \frac{\overline {\gamma '(z)} }{\gamma'(z)}, \text{ for all $\gamma\in \Gamma $ and $z\in \bH^2$},
\end{equation}
and similarly for $z\in \overline{\bH}^2$, because $f'$ lifts $\hat f'$; see \eqref{eqn: invariant Beltrami} and \eqref{eqn: laminated Beltrami differential}.
Another way to express this condition is to say that 
\[\mu_{f\sqcup f'}\frac{d\overline z}{dz} \text{ is invariant by $\iota$.}\]

Since $\hat f$ and $\hat f'$ are orientation preserving diffeomorphisms on leaves whose derivatives vary continuously transversally, $\mu_{f\sqcup f'}$ is smooth (in the lamination sense) on $\fCPL$.  The compactness of $M_\Gamma$ implies that $\|\mu_{f\sqcup f'}\|_\infty<1$.
By Theorem \ref{thm:mrm}, we can find a continuous, quasi-conformal isomorphism of surface laminations 
$G:\mb{CP}^1\times\partial\mb{H}^2\to\mb{CP}^1\times\partial\mb{H}^2$ satisfying
    \begin{equation}\label{eqn:beltrami equation}
\frac{\partial G}{\partial\overline{z}}(\cdot,t)=\mu_{f\sqcup f'}(\cdot,t)\frac{\partial G}{\partial z}(\cdot,t) \text{ a.e.}
\end{equation}

\begin{lemma}\label{lem: rhoEF defined}
    The rule 
    \[\gamma\in \Gamma \mapsto G\circ \gamma \circ G\inverse\]
    defines a representation $\rho_{E,F}: \Gamma \to \Hom(\Gamma, \MG)$.
\end{lemma}
\begin{proof}
    We only need to check that $\rho(\gamma)$ is leafwise conformal.
    This follows from the invariance property \eqref{eqn:automorphic} and uniqueness of normalized solutions \eqref{eqn:beltrami equation}; see \cite[Chapter VI.B]{Ahlfors:lectures}.
\end{proof}

We will show that $G$ is smooth away from $\fL$, so that $\rho_{E,F}$ is a quasi-conformal deformation of $\iota$ with the desired properties:
\begin{prop}\label{p.invAB}
    Given $([E], [F]) \in \T(M_\Gamma)\times \T(\overline M_\Gamma)$, the representation $\rho_{E,F}$ from Lemma \ref{lem: rhoEF defined} is a quasi-conformal deformation of $\iota$ and satisfies $\mathcal {AB}([\rho_{E,F}]) = ([E], [F])$. 
    Furthermore, the assignment $([E],[F])\mapsto [\rho_{E,F}] \in\mathcal {QC}(\iota)$ is continuous.
\end{prop}

\begin{proof}
We have an equivariant map 
\[\Phi = (f\sqcup f' ) \circ G\inverse : \fCPgL \to U \sqcup V,\]
whose leafwise Beltrami coefficient satisfies \eqref{eqn: beltrami composite}
\[\mu_\Phi \circ G=\frac{\partial G/\partial z}{\partial\overline{G}/\partial\overline{z}}\cdot\frac{\mu_{f\sqcup f'}-\mu_G}{1-\overline{\mu}_G\mu_{f\sqcup f'}}\] 
in the sense of distributions.
By construction, $\mu_G=\mu_{f\sqcup f'}$ a.e. and $|\overline \mu_G \mu_{f\sqcup f'}|<1$, so $\mu_\Phi=0$ a.e. 
This implies that $\Phi$ is leafwise conformal.
Since $f$ and $f'$ are smooth on their domains of definition, we conclude that  $G$ is smooth away from $\fL$.
Furthermore, $G$ descends to a pair of markings $M_\Gamma \to E$ and $\overline M_\Gamma \to F$ the homotopy classes of $\hat f$ and $\hat f'$, respectively.
This proves the claim that $\mc{AB}([\rho_{E,F}]) = ([E], [F])$.

Continuity follows directly from continuity of solutions of the Beltrami equation in the Measurable Riemann Mapping Theorem \ref{thm:mrm}.
\end{proof}

\begin{remark}
\label{rmk:conformal welding}
Note the following features of the conformal welding construction.
\begin{itemize}
\item{Given $([E],[F])\in\T(M_\Gamma)\times\T(\overline{M}_\Gamma)$ let $g_{E,F}:\mb{CP}^1\times\partial\mb{H}^2\to\mb{CP}^1\times\partial\mb{H}^2$ be the solution to the corresponding equation \eqref{eqn:def beltrami}. Then $K_{g_{E,F}}$ is bounded by $\max \{ K_{\hat f}, K_{\hat f'}\}$.}
\item{When $F$ is the mirror image of $E$, we can choose the markings $\hat f$ and $\hat f'$ to satisfy $f'(z):=\bar{f}(\bar{z})$ where $\bar{f}$ is the complex conjugate of $f$.  The corresponding solution $g_{E,F}$ to \eqref{eqn:def beltrami} preserves $\fL$.}
\end{itemize}
\end{remark}

\subsubsection{Mapping class group equivariance}\label{ss.mcg}

Recall that the mapping class group ${\rm Mod}(M_\Gamma)$ of the lamination is the group of smooth surface lamination automorphism up to leafwise homotopy. We briefly describe the canonical action of ${\rm Mod}(M_\Gamma)$ on $\mc{QC}(\iota)$.

Consider $[g\iota g^{-1}]\in\mc{QC}(\iota)$ a quasi-conformal deformation and $[\phi]\in{\rm Mod}(M_\Gamma)$ a mapping class. Let $\phi$ be a (any) representative of $[\phi]$. Lift it to a smooth $\Gamma$-equivariant lamination automorphism $\Phi:\mb{H}^2\times\partial\mb{H}^2\to\mb{H}^2\times\partial\mb{H}^2$. By cocompactness of the action $\Gamma\curvearrowright\mb{H}^2\times\partial\mb{H}^2$ every map $\Phi(\cdot,t):\mb{H}^2\to\mb{H}^2$ is a $K$-quasi-conformal map. As the boundary extension operator ${\rm QC}(\mb{H}^2)\to{\rm Homeo}(\partial\mb{H}^2)$ is continuous with respect to the natural topologies on domain and target, we get that $\Phi$ extends continuously to the closure of $\mb{H}^2\times\partial\mb{H}^2$ in $\mb{CP}^1\times\partial\mb{H}^2$ via a homeomorphism 
\[
\partial\Phi:\mb{RP}^1\times\partial\mb{H}^2\to\mb{RP}^1\times\partial\mb{H}^2.
\]

Define $\Phi^e$ by
\[
\Phi^e:=\left\{
\begin{array}{ll}
\Phi^{-1} &\text{on $\mb{H}^2\times\partial\mb{H}^2$}  \\
\overline{\Phi}^{-1} &\text{on $\overline{\mb{H}}^2\times\partial\mb{H}^2$}\\
\partial\Phi^{-1} &\text{on $\mb{RP}^1\times\partial\mb{H}^2$}\\
\end{array}
\right.
\]
where $\bar{\Phi}$ is the complex conjugate of $\Phi$.

Note that $\Phi^e(\cdot,t)$, being continuous and quasi-conformal on $\mb{CP}^1-\mb{RP}^1$, it is quasi-conformal everywhere (by Proposition \ref{prop: quasicircle removable}). Thus $(g\Phi^e)\iota(g\Phi^e)^{-1}$ is a quasi-conformal deformation of $\iota$ and we can define 
\[
[\phi]\cdot[g\iota g^{-1}]:=[(g\Phi^e)\iota (g\Phi^e)^{-1}].
\]

It is straightforward to check that the action is well-defined. 

We prove the following.

\begin{prop}\label{p.mcg}
The map $\mc{AB}$ is ${\rm Mod}(M_\Gamma)$-equivariant.
\end{prop}

\begin{proof}
Let $[g\iota g^{-1}]\in\mc{QC}(\iota)$ be a quasi-conformal deformation.
By definition $\mc{AB}(\rho)$ is the marked Riemann surface lamination defined by
\[
g:(\mb{CP}^1\times\partial\mb{H}^2-\mc{L})_{/\iota(\Gamma)}\to(\mb{CP}^1\times\partial\mb{H}^2-g(\mc{L}))_{/g\iota g^{-1}(\Gamma)}.
\]

Let $[\phi]\in{\rm Mod}(M_\Gamma)$ be a mapping class. By definition $[\phi]\cdot[g\iota g^{-1}]=[(g\Phi^e)\iota(g\Phi^e)^{-1}]$ where $\Phi^e$ is defined above.
Hence $\mc{AB}([\phi]\cdot[g\iota g^{-1}])$ is the marked Riemann surface lamination defined by
\[
g\circ\Phi^e:(\mb{CP}^1\times\partial\mb{H}^2-\mc{L})_{/\iota(\Gamma)}\to(\mb{CP}^1\times\partial\mb{H}^2-g(\mc{L}))_{/(g\Phi^e)\iota(g\Phi^e)^{-1}(\Gamma)}
\]

By construction, the restrictions of $g\Phi^e$ to $M_\Gamma$ and $\overline{M}_\Gamma$ coincide with $g\phi^{-1}$ and $g\overline{\phi}^{-1}$. Hence $\mc{AB}([\phi]\cdot[g\iota g^{-1}])=[\phi]\cdot\mc{AB}([g\iota g^{-1}])$ as desired.
\end{proof}

\subsection{Complex dilation spectrum}

Given a loxodromic element $h\in\PSL(2,\C)$ we say that its \emph{complex dilation} $L(h)$ is the ratio $\lambda_1(h)/\lambda_2(h)$ where $\lambda_i(h)$ are the eigenvalues of any lift of $h$ ordered so that $|\lambda_1(h)|>|\lambda_2(h)|$. Of course, this means that $h$ can be expressed, in an appropriate chart, as the transformation $z\mapsto L(h) z$, and the translation length of $h$ is computed by $\ell(h)=\log|L(h)|$.

Let now $\rho=g\iota g^{-1}:\Gamma\to\mc{MG}$ be a quasi-conformal deformation. For every $\gamma\in\Gamma-\{1\}$  the  restriction $\rho(\g,\g^+)$ of $\rho(\gamma)$ to the leaf $\mb{CP}^1\times\{\gamma^+\}$ is loxodromic, since it  fixes the distinct points $g(\g^+,\g^+)$ and $g(\g^-,\g^+)$, and has north-south dynamics.

\begin{dfn}[Complex dilation spectrum]\label{d.Clength}
The {\em complex dilation spectrum} of the quasi-conformal deformation $\rho=g\iota g^{-1}$ is the function $L_\rho:[\Gamma]-\{1\}\to\mb{C}^*$ that associates to $[\gamma]\in[\Gamma]$ the complex dilation of $\rho(\g)|_{\CP\times \{\g^+\}}$.
\end{dfn}

More generally if $\rho\in\Hom(\G,\MG)$ is a laminated conformal action covering the standard action of $\G$ on $\deH$, we define its complex dilation spectrum $L_\rho$ to be given by $L_\rho(\g)=\lambda_1(h)/\lambda_2(h)$ if $h=\rho(\g,\g^+)$ is loxodromic and equal to 1 if it is parabolic or elliptic, which coincides with the exponential of its translation length for the action on $\bH^3$.

The following is a generalization of Ahlfors' Lemma \cite[Lemma\,5.1.1]{Otal:hyperbolization}. Recall that the length spectrum $\ell_E$ of a hyperbolic surface lamination $E$ was defined in Definition \ref{d.reallength}, and associates to $\g\in\G$ the length of the closed geodesic of the leaf corresponding to $\g^+$.
\begin{prop}
\label{p.bound}
Let $\rho=g\iota g^{-1}:\Gamma\to\mc{MG}$ be a quasi-conformal deformation with Ahlfors--Bers parameters $\mc{AB}(\rho)=(E_\rho,F_\rho)$. For every $\gamma\in\Gamma$ we have 
\[
\log\left|L_\rho(\gamma))\right|\le 2\min\{\ell_{E_\rho}(\gamma),\ell_{F_\rho}(\gamma)\}.
\]
\end{prop}

\begin{proof}
We follow the proof given in \cite[Lemma\,5.1.1]{Otal:hyperbolization}. Assume for simplicity that $\ell_{E_\rho}(\gamma)\le\ell_{F_\rho}(\gamma)$. Let us show that $\log\left|L_\rho(\gamma))\right|\le 2\ell_{E_\rho}(\gamma)$.

Consider $g(\mb{H}^2_+\times\{\gamma^+\})\subset\mb{CP}^1\times\{\gamma^+\}$. We normalize the action of $\rho(\gamma)=g\gamma g^{-1}$ on $\mb{CP}^1\times\{\gamma^+\}$ so that $\rho(\gamma)z=L_\rho(\gamma)z$.
Let $u:\mb{H}^2\to g(\mb{H}^2_+\times\{\gamma^+\})$ be a uniformization whose Caratheodory extension to the boundary maps $0$ to $0=g(\gamma^-,\gamma^+)$ and $\infty$ to $\infty=g(\gamma^+,\gamma^+)$. This way $u^{-1}\rho(\gamma)u$ acts on $\mb{H}^2$ as the hyperbolic motion $z\to e^{\ell_{E_\rho}(\gamma)}z$. By equivariance, $u(e^{\ell_{E_\rho}(\gamma)}z)=L_\rho(\g)u(z)$.

Now recall that Koebe 1/4 Theorem says that if $v:\pD\to\mb{C}$ is an injective holomorphic map from the unit disk to the complex plane, then the image $v(\pD)$ contains the Euclidean ball of radius $|v'(0)|/4$ centered around $v(0)$. For every $z\in i(0,\infty)\subset\mb{H}^2$ (we identify $\mb{H}^2$ with the upper half-plane), we choose $v_z:=u\circ w_z$ where $w_z:\pD\to\mb{H}^2$ is the biholomorphism $w_z(s)=-|z|\frac{is+1}{s+i}$, which sends $0$ to $z$ with $|w'_z(0)|=2|z|$. As $0$ lies on the boundary of $v_z(\mb{H}^2)$, Koebe's theorem gives that $|v_z'(0)|/4{\leq}|v_z(0)|$ since the ball of radius $|v_z'(0)|/4$ centered at $v_z(0)$ must be contained in the image of $v$. Observe that $v_z(0)=u(w_z(0))=u(z)$ and that $v_z'(0)=u'(z)w_z'(0)=u'(z)2|z|$. Thus, for every $z\in i(0,\infty)$, we have
\[
\frac{|u'(z)|}{|u(z)|}\le\frac{2}{|z|}.
\]
We are now able to conclude the proof. Consider the path $\alpha:[0,1]\to\mb{H}^2$ given by $\alpha(t)=e^{\ell_E(\gamma)t}i$. Note that, by equivariance of $u$, we have $u(\alpha(1))=L_\rho(\gamma)u(\alpha(0))$. Thus
\begin{align*}
\log|L_\rho(\gamma)| &=\log\left|\frac{u(\alpha(1)}{u(\alpha(0))}\right|\\
 &=\int_0^1{\frac{{\rm d}}{{\rm d}t}\log|u(\alpha(t))|{\rm d}t}\\
 &=\int_0^1{\frac{|u'(\alpha(t))|}{|u(\alpha(t)|}|\alpha'(t)|{\rm d}t}\\
 &\le\int_0^1{\frac{2}{|\alpha(t)|}|\alpha'(t)|{\rm d}t}=2\ell_E(\gamma).
\end{align*}
\end{proof}

Lastly, we prove that if the Ahlfors--Bers parameters lie on the diagonal, then the complex dilation spectrum has only real values.

\begin{prop}
\label{prop:real dilation} 
Let $\rho=g\iota g^{-1}:\Gamma\to\mc{MG}$ be a quasi-conformal deformation with Ahlfors--Bers parameters $\mc{AB}(\rho)\in\Delta\subset\T(M_\Gamma)\times\T(\overline{M}_\Gamma)$. For every $\gamma\in\Gamma$ we have $L_\rho(\gamma)\in\mb{R}$. 
\end{prop}

\begin{proof}
By the discussion in Remark \ref{rmk:conformal welding}, if $\mc{AB}(\rho)=(E,E)$, then $\rho$ is equivalent to a representation that leaves invariant $\mb{RP}^1\times\partial\mb{H}^2$. It follows that $\rho(\gamma)$ preserves $\mb{RP}^1\times\{\gamma^+\}$ and, hence, it has real dialtion spectrum.
\end{proof}

\section{Hyperconvex representations}
\label{sec:hyperconvex}

In this section, we introduce and discuss preliminaries on hyperconvex representations $\rho:\Gamma\to\PSL(d,\CC)$.

\subsection{Anosov representations}\label{s.Anosov}
Anosov representations of hyperbolic groups were introduced by Labourie \cite{Labourie:Anosov} and Guichard-Wienhard \cite{GW:Anosov_DOD} using dynamical properties of the lift of the geodesic flow on the flat bundle associated to the representation. 
We will need an equivalent characterization combining the work of Kassel-Potrie \cite{KP} and the work of Kapovich-Leeb-Porti \cite{KLP}. 
 For a finitely generated group $\Gamma$ we choose a symmetric finite generating set and we denote by $|\cdot|_\infty$ the stable length, namely $|\gamma|_\infty=\lim_{n\to\infty}\frac{|\gamma^n|}{n}$, where $|\cdot|$ is the word-length associated to our choice of finite generating set.
\begin{dfn}[Anosov representation]
  Let $\G$ be hyperbolic, $1\leq k\leq d$. A representation $\rho:\G\to\PSL(d,\C)$ is \emph{$k$-Anosov} if and only if there exist positive constants $c_1,c_2$ such that
\[
\log|\lambda^{k}(\rho(\g))|-\log|\lambda^{k+1}(\rho(\g))|\geq c_1|\gamma|_\infty-c_2.
\]   
\end{dfn}
Here for $g\in\PSL_d(\C)$ we denote by $\lambda_1(g),\ldots, \lambda_d(g)$ the eigenvalues of any lift of $g$ to $\SL_d(\C)$ ordered so that their moduli are non-decreasing.

It follows from the definition that a $k$-Anosov representation is also $(d-k)$-Anosov. Any $k$-Anosov representation admits continuous, equivariant, transverse boundary maps $\xi^k:\partial\Gamma\to\Gr_k(\C^d)$, $\xi^{d-k}:\partial\Gamma\to\Gr_{d-k}(\C^d)$ \cite{KLP,BPS}. The latter means that for every $x\neq y$, the sum $\xi^k(x)+\xi^{d-k}(y)$ is direct. We will sometimes   just write for a point $x\in\partial\Gamma$
\[
x^k:=\xi^k(x)\in\Gr_k(\C^d).
\]
If $\rho$ is $k_1$ and $k_2$-Anosov with $k_1<k_2$ then for every $x\in\deG$ it holds $x^{k_1}<x^{k_2}$.

The boundary map $\xi^k$ parametrizes the \emph{limit set} $\Lambda^k_\rho:=\xi^k(\deG)\subset\Gr_k(\C^d)$ of the $\rho$-action on $\Gr_k(\C^d)$, namely the smallest closed $\rho(\G)$-invariant set.

\subsection{Hyperconvex representations}\label{subsec: hyperconvex}

Hyperconvex representations, the main focus of this article, form a subclass of Anosov representations in ${\rm PSL}(d,\mb{C})$ satisfying additional transversality properties.  For the next definition we declare any representation in $\PSL(d,\C)$ to be both $0$ and $d$-Anosov with the degenerate boundary maps $\xi^0(x)=\{0\}$ and $\xi^d(x)=\C^d$ for all $x\in\deG$.

\begin{dfn}[$k$-hyperconvex]
\label{d.hyp}
A representation $\rho: \Gamma \to{\rm PSL}(d,\mb{C})$ is {\em $k$-hyperconvex} for $1\le k\le d-1$ if it is $\{k-1,k+1,d-k\}$-Anosov, and, for all distinct triples $x, y, z\in \partial \Gamma$, we have 
\begin{equation}\label{e.hyp}\left((x^{d-k}\cap z^{k+1})+z^{k-1}\right)\cap\left((y^{d-k}\cap z^{k+1})+z^{k-1}\right) = z^{k-1}.
\end{equation}
\end{dfn}

\begin{remark}
This notion, studied in \cite{FPV2}, is the dual notion to property $H_k$ as defined in \cite[Definition 6.1]{PSW:HDim_hyperconvex}, requiring that the sum $(x^k\cap z^{d-k+1})+(y^k\cap z^{d-k+1})+z^{d-k-1}$ is direct, see \cite[Remark 2.5]{FPV2}. 
\end{remark}

A useful tool to study hyperconvex representations is the tangent projection construction, which we discuss in the next proposition. 
They will provide a marking of the laminated limit set of the conformal action which we will associate to a hyperconvex representation (recall Equation \eqref{e.foliatedmarking} in Section \ref{subsubsec:continuity of AB}).

\begin{prop}[Tangent projections]\label{p.tangentproj}
Let $\rho:\Gamma\to{\rm PSL}(d,\mb{C})$ be $\{k-1,k+1,d-k\}$-Anosov. Then, for each $z\in \partial \Gamma$, the function 
$\xi_z^k:\partial \Gamma\to\bP(z^{k+1}/z^{k-1})$
\[
\xi_z^k(x):=\left\{
\begin{array}{l l}
{[x^{d-k}\cap z^{k+1}]} &\text{if $x\neq z$,}\\
{[z^k]} &\text{if $x=z$}\\
\end{array}
\right.
\]
is well-defined. The representation $\rho$ is $k$-hyperconvex if and only if for every $z\in\deG$, $\xi_z^k$ is injective. In this case, for every $z$, $\xi_z$ is continuous.
\end{prop}
\begin{proof}
Since the representation is Anosov, for every $x,z$ the intersection $x^{d-k}\cap z^{k+1}$ is one dimensional and not contained in $z^k$, as a result $[x^{d-k}\cap z^{k+1}]$ gives a well defined point in $\bP(z^{k+1}/z^{k-1})$ distinct from $[z^k]$.
The injectivity of the maps $\xi^k_z$ outside the point $z$ is precisely Equation \eqref{e.hyp}, from which the second claim follows.
Continuity (at $z$) of $\xi^k_z$ was proven in \cite[Proposition 4.9]{BeyP1} building on ideas from \cite[Proposition 6.7]{PSW:HDim_hyperconvex}, see also \cite[Proposition 2.6]{FPV2}.
\end{proof}

Specialize now to surface groups $\Gamma$.
We denote by 
$$\Lambda^k_z:=\xi_z^k(\deG)\subset \P(z^{k+1}/z^{k-1})$$
the image of the boundary map $\xi_z$, we proved in \cite{FPV2} that the Jordan curves $\Lambda^k_z$ are uniform quasi-circles.
\begin{prop}[{\cite[Proposition 4.12]{FPV2}}]\label{p.qc}
There exists $K$ only depending on $\rho$ such that, for every $z\in\partial\Gamma$ the Jordan curve  $\xi_z:\deG\to\CP$ is  $K$-quasi-Möbius.
\end{prop}

Since transversality is an open condition, and $\G$ acts co-compactly on $\deG^{(3)}$, hyperconvexity is an open condition on representations \cite[Proposition 6.2]{PSW:HDim_hyperconvex}.
We denote by $$\Xi^k(\G,\PSL(d,\C))\subset\Hom(\G,\PSL(d,\C))/\!\!/\PSL(d,\C)$$ 
the set of conjugacy classes of $k$-hyperconvex representations, and by $$\Xi^{\rm hyp}(\G,\PSL(d,\C))\subset\Hom(\G,\PSL(d,\C))/\!\!/\PSL(d,\C)$$ the set of conjugacy classes of \emph{fully hyperconvex} representation, namely those representations that are $k$-hyperconvex for every $1\leq k\leq d-1$, so that we have 
$$\Xi^{\rm hyp}(\G,\PSL(d,\C))\subset\Xi^k(\G,\PSL(d,\C))\subset\Hom(\G,\PSL(d,\C))/\!\!/\PSL(d,\C).$$
Since the character variety $\Hom(\G,\PSL(d,\C))/\!\!/\PSL(d,\C)$ is a complex variety, the open subsets  $\Xi^{\rm hyp}(\G,\PSL(d,\C))\subset\Xi^k(\G,\PSL(d,\C))$, albeit sometimes singular, inherit a natural complex structure.

\subsection{The laminated conformal action associated to a hyperconvex representation}\label{sec:tangent}
To a $k$-hyperconvex representation $\rho:\G\to\PSL_d(\C)$, we can naturally  associate an element $\rho^k\in\Hom(\G,\MG)/\MG$, constructed as follows. While most of the construction in this section works for general hyperbolic groups (and are carried out in that generality in \cite{FPV2}), we restrict here, as customary in the paper, to a torsion-free cocompact lattice $\G<\PSL_2(\R)$ and freely identify $\deG$ with $\deH$.

Analogously to \cite[Section 2]{FPV2}, we consider the partial flag manifolds
\begin{align*}
&\calF_{k-1,k,k+1}=\{(U,V,W)\in\mc{F}_{k-1}(\C)\times\mc{F}_{k}(\C)\times\mc{F}_{k+1}(\C)\left|\;U< V< W\right.\}\\
&\calF_{k-1,k+1}=\{(U,W)\in\mc{F}_{k-1}(\C)\times\mc{F}_{k+1}(\C)\left|\;U< W\right.\},
\end{align*}
so that the  continuous surjection
\[
\calF_{k-1,k,k+1}\to \calF_{k-1,k+1}
\]  
is a fiber bundle with fibers isomorphic to $\mb{CP}^1$.
It can be naturally identified with the canonical fiber bundle $\mathcal B\to\calF_{k-1,k+1}$ with fiber $\mb{P}(W/U)$ over the pair $(U,W)\in\calF_{k-1,k+1}$.

We use the map $t\in \deG \mapsto (t^{k-1}, t^{k+1}) \in \mathcal {F}_{k-1,k+1}$ to pullback a $\CP$-bundle $\mathcal B_\rho^k \to \deG$.
For each $\gamma \in \Gamma$ and $t\in \deG$, the linear map $\rho(\gamma) \in \mathrm{PSL}(d,\C)$ induces a projective isomorphism 
\begin{equation}\label{eqn: mobius cocycle}
    \rho^k(\g,t):\mb{P}(t^{k+1}/t^{k-1})\to\mb{P}((\g t)^{k+1}/(\g t)^{k-1}).
\end{equation}
Thus $\rho$ induces a $\G$ action on $\mathcal B_\rho^k$ by $\CP$-bundle automorphisms covering the natural action of $\G$ on $\deG$.

In order to define the desired conformal action $\rho^k$, we choose an identification of $\mathcal B_\rho^k$ with $\fCP$: we identify $\deG$ with $\deH$ through our chosen hyperbolization $\G<\PSL_2(\R)$ and we trivialize the bundle $\mathcal B_\rho^k \to \deG$ by finding three continuous sections in fiber-wise general position, with the aid of the tangent projections $\xi_\cdot^k$ discussed in Proposition \ref{p.tangentproj}.
Specifically we fix a triple of pairwise distinct points $x,y,z\in \deG$, and define 
\[T: \mathcal B_\rho^k \to \fCP\] 
fiber-wise as the unique bi-holomorphism with 
\[(\xi_t(x), \xi_t(y),\xi_t(z)) \mapsto (0,1,\infty)\]
in the fiber over $t$.
Then $T$ is a isomorphism of $\CP$-bundles trivializing $\mathcal B_\rho^k$. 
Conjugating the $\G$ action \eqref{eqn: mobius cocycle} on $\mathcal B_\rho^k$  by $T$ defines the desired conformal action of $\G$ on $\fCP$ still denoted $\rho_k$

We will leave, from now on, implicit the trivialization $T$ of the bundle $\mathcal B_\rho^k$ and just identify it with $\fCP$.

\begin{prop}
The class in $\Hom(\G,\MG)/\MG$ of the laminated conformal action $\rho^k\in\Hom(\G,\MG)$ associated to a $k$-hyperconvex representation $\rho:\G\to\PSL(d,\C)$  is independent on the choices.
\end{prop}
\begin{proof}
This follows readily from the fact that any two choices of sections are related by an element in $\MG$. As such, the class $[\rho^k]$ is independent of the choice of the points $x,y,z$.
\end{proof}
\begin{remark}
 We will prove in Section \ref{sec:marking} that $[\rho^k]\in\qc(\iota)$.   
\end{remark}

Note that $\rho^k$ also defines an action of $\G$ on the trivial bundle $\HH^3\times \deH$ which is isometric on the fibers and covers the natural action of $\G$ on $\deH$.

The laminated conformal action associated with a $k$-hyperconvex representation encodes the $k$-th eigenvalue gaps. Given a $k$-Anosov representation 
$\rho:\G\to\PSL(d,\C)$ we denote by
$L_\rho^k:\G\to\C$ its $k$-th eigenvalue gap, so that for every $\g\in\G$,
\begin{equation}\label{e.gap}L_\rho^k(\gamma)=\frac{\lambda^k(\rho(\gamma))}{\lambda^{k+1}(\rho(\gamma))}.
\end{equation}
Recall from Definition \ref{d.Clength} that we denote by $L_\eta:[\Gamma]-\{1\}\to\mb{C}^*$ the complex dilation spectrum of a laminated conformal action $\eta:\G\to\MG$.
It follows from the construction the $k$-laminated conformal action $\rho^k:\G\to\MG$ encodes the $k$-th eigenvalue gap.
\begin{prop}\label{p.Clength}
Let $\rho:\G\to\PSL(d,\C)$ be $k$-hyperconvex. Then
$$L_{\rho^k}=L_\rho^k.$$
\end{prop}
\begin{proof}
By definition, for $\g\in[\G]-\{1\}$, $L_{\rho^k}(\g)$ is the ratio of the two eigenvalues of the projective action of $\rho^k(\g)$ on the leaf $\CP\times \{\g^+\}$. Since such action is projectively equivalent to the action on $\P\left((\g^+)^{k+1}/(\g^+)^{k-1}\right)$, where $\g^+$ acts with eigenvalues $\lambda^k(\g),\lambda^{k+1}(\g)$, the result follows.
\end{proof}

\section{Laminated Ahlfors--Bers maps, proof of Theorem \ref{thm:IntroAB}}
\label{sec:ab maps}

Recall that we fixed a closed oriented hyperbolic surface $\mb{H}^2/\Gamma$, we denote by $M_\Gamma=\mb{H}^2\times\partial\mb{H}^2/\Gamma$ the associated hyperbolic surface lamination and by $\Chark$ the space of $k$-hyperconvex representations. 

The goal of this section is to define an Ahlfors--Bers map ${\rm AB}^k:\Chark\to\T(M_\Gamma)\times\T(\overline{M}_\Gamma)$ with the properties described in Theorems \ref{thm:IntroAB}  from the introduction (which will be proved in the section).

\subsection{Quasi-conformal conjugacy}
\label{sec:marking}
We first prove that the laminated conformal action associated to a $k$-hyperconvex representation is a quasi-conformal deformation of the standard Möbius action.
\begin{prop}\label{prop:qcconjugacy}
Suppose $\rho:\G\to\PSL(d,\C)$ is $k$-hyperconvex and let $\rho^k$ be as in \S\ref{sec:tangent}. Then 
$$[\rho^k]\in\qc(\iota).$$
\end{prop}
We prove Proposition \ref{prop:qcconjugacy} by exhibiting a quasiconformal conjugacy between the standard representation $\iota:\Gamma\to\mc{MG}$ and $\rho^k:\Gamma\to\mc{MG}$ using the laminated Douady--Earle extension (Proposition \ref{prop:DE}).
Recall from Section \ref{subsubsec:continuity of AB} that $\fL = \RP \times \partial \H^2$ and $\mathcal {LM}$ is a space of laminated markings $\fL \to \CP\times \partial \H^2$; see \eqref{e.foliatedmarking}.

\begin{lem}[{see \cite[Lemma 2.8]{FPV2}}]\label{lem: double boundary continuous into bundle}
    The map 
    $$\begin{array}{cccc}
\xi_\cdot^k:&\fL&\to&\fCP\\
    &(x,t) &\mapsto &(\xi_t^k(x), t) 
    \end{array}$$
    is a homeomorphism onto its image, and thus induces an element of $\fM$.
\end{lem}

\begin{dfn}[Laminated limit set]\label{def: double limit set}
    The {\emph{laminated limit set}} $\fL^k_\rho\subset\fCP$ is  
    the image of the map $\xi_\cdot^k$ from Lemma \ref{lem: double boundary continuous into bundle}.
\end{dfn}

\begin{proof}[Proof of Proposition \ref{prop:qcconjugacy}]
    Applying the laminated Douady--Earle extension operator from Proposition \ref{prop:DE}, we obtain a continuous lamination equivalence
    \[\widehat{\DE}(\xi_\cdot^k): \CP \times \deH \to \CP \times \deH\]
    extending $\xi_\cdot^k$.  
    By construction, $\xi_\cdot^k$ is $(\iota, \rho^k)$ equivariant:
    \[\xi^k_{\gamma t}(\gamma z) = \rho^k(\gamma,t)\xi^k_t(z),\]
    for all $\gamma \in \Gamma$ and for all $t, z\in \partial \H^2$.
    Then Proposition \ref{prop:DE} (4) implies that $\widehat{\DE}(\xi_\cdot^k)$ is $(\iota, \rho^k)$-equivariant, i.e.,
    \[ \rho^k(\gamma) \widehat{\DE}(\xi_\cdot ^k) = \widehat{\DE}(\xi_\cdot^k)\iota (\gamma)\]
    holds for all $\gamma \in \Gamma$.
    
    Using Proposition \ref{p.qc}, we know that $\xi_t^k$ is $K$-quasi-M\"obius for some $K$ independent of $t$.
    Proposition \ref{prop:DE} (3) asserts that $\widehat{\DE}(\xi_\cdot^k)$ is $K^*$-quasi-conformal and smooth away from $\fL$.
    This completes the proof that $[\rho^k] \in \qc(\iota)$.
\end{proof}

\subsection{The laminated Ahlfors--Bers map ${\rm AB}^k$ and the Ahlfors lemma}

In \S\ref{s.ABuniversal}, we defined a universal Ahlfors--Bers map.
This construction applies, in particular, to $[\rho^k]\in \qc(\iota)$ obtained from a $k$-hyperconvex representation $[\rho] \in \Xi^k(\Gamma, \PSL(d,\C))$.

\begin{dfn}[$k$-th Ahlfors--Bers Map]
Let $[\rho] \in \Xi^k(\Gamma,{\rm PSL}(d,\mb{C}))$, and define
\begin{align*}
{\rm AB}^k([\rho]) &:=\mc{AB}([\rho^k]) \in \T(M_\Gamma) \times \T(\overline M_\Gamma).
\end{align*}
This is the $k$-th {\em Ahlfors--Bers map}.
\end{dfn}
We denote by $E^k_\rho$ and $F^k_\rho$ the Riemann surface laminations such that  
$${\rm AB}^k([\rho])=\left([h^k_E:M_\Gamma\to E^k_\rho],[h^k_F:\overline{M}_\Gamma\to F^k_\rho]\right)\in\T(M_\Gamma)\times\T(\overline{M}_\Gamma).$$
From now on, we will often abuse notation  and drop both the marking and the brackets indicating the equivalence class, e.g.,
\[\mathrm{AB}^k(\rho) = (E_\rho^k,F_\rho^k)\in \T(M_\Gamma) \times \T(\overline M_\Gamma).\]
We now prove that ${\rm AB}^k$ satisfies the properties of Theorem \ref{thm:IntroAB}, beginning with  Property (2), the laminated Ahlfors Lemma. 

 Recall the definition of the marked length spectrum $\ell_E$ of a hyperbolic Riemann surface lamination (Definition \ref{d.reallength}) and of the $k$-th eigenvalue gap $L^k_\rho$ of a $k$-Anosov representation $\rho:\G\to\PSL(d,\C)$ (Equation \eqref{e.gap}).
\begin{prop}
Let $\rho:\G\to\PSL(d,\C)$ be $k$-hyperconvex. Then $$2\min\{\ell_{E^k_\rho}(\cdot),\ell_{F^k_\rho}(\cdot)\}\ge\log|L^k_\rho(\cdot)|.$$
\end{prop}
\begin{proof}
Proposition \ref{p.Clength} proves that the complex dilation spectrum associated to $\rho^k$ computes the $k$-th eigenvalue gap of $\rho$, while
Proposition \ref{p.bound} establishes the desired bound for a general quasi-conformal deformation of $\iota$.
This completes the proof.
\end{proof}

\subsection{The Ahlfors--Bers map is holomorphic}
In this section we prove that the map ${\rm AB}^k$ is holomorphic. Recall that 
the set $\Chark$ of $k$-hyperconvex representations forms an open subset of the character variety, which in turn is a (potentially singular) affine variety. 

Denote by $\pD \subset \C$ the open unit disk.
Harthog's Theorem (\cite[Theorem 14.5]{Holomorphy}) asserts that a map $f$ from $\Chark$ to an (infinite dimensional) complex Banach space is holomorphic if, for every holomorphic map $\delta:\pD\to\Chark$, the composition $f\circ\delta$ is holomorphic.

\begin{thm}
\label{thm:ab map}
The Ahlfors--Bers map
\[
{\rm AB}^k:\Xi^k(\Gamma,{\rm PSL}(d,\mb{C}))\to\T(M_\Gamma)\times\T(\overline M_\Gamma)
\]
is holomorphic. 
\end{thm}
Our proof of Theorem \ref{thm:ab map} relies on the Holomorphic Motion Theorem of Sullivan--Thurston (Theorem \ref{t.holmotion} below) and the characterization of the complex structure on the Teichm\"uller space of Riemann surface laminations recalled in \S\ref{sec:Bersembedding}. In order to apply the first result, we show that, for a holomorphic family of representations, the tangential projections vary pointwise holomorphically.

An element $g\in\PSL(d,\C)$ is \emph{$s$-proximal} if it admits a unique attracting fixed point in $\Gr_s(\C^d)$, or equivalently if $|\lambda^s(g)|>|\lambda^{s+1}(g)|$ where, as always, we order the eigenvalues with non-increasing moduli. 
By definition, if $\rho:\G\to\PSL(d,\C)$ is $s$-Anosov, then $\rho(\g)$ is $s$-proximal for every (infinite order) element $\g\in\G$. The next statement is probably known to experts; we include a proof for the lack of a convenient reference.
\begin{prop}\label{prop:attractorsarehol}
Let $g: \pD \to \PSL(d,\C)$ be a holomorphic family of $s$-proximal elements, and for $w\in\pD$, denote by $A^s(w)\in \Gr_s(\C^d)$ the attracting fixed point of $g(w)$. 
The function $A^s: \pD \to\Gr_s(\C^d)$ is holomorphic.
\end{prop}

\begin{proof}
Since $\pD$ is simply connected and ${\rm SL}(d,\mb{C})\to{\rm PSL}(d,\mb{C})$ is a covering we can lift the holomorphic map $w\to g(w)$ to a holomorphic map $w\to{\hat g}(w)$ with values in ${\rm SL}(d,\mb{C})$.

Consider the holomorphic function $\delta:\mathbb C\times \pD\to\mathbb C$ given by 
$$\delta(u,w)={\rm det}(u\mb{I}-{\hat g}(w)).$$

Let $\lambda_0:=\lambda^s({\hat g(0)})$ and find $\ep>0$ such that $|\lambda^{s+1}(\hat g(0))| < |\lambda_0|-\epsilon$. 
By assumption, we have: 
\begin{enumerate}
\item{The polynomial $\delta(u,0)$ has $s$ roots (counted with multiplicity) outside the disk $B(0,\lambda^0-\ep/2)$.}
\item{It has $d-s$  roots (counted with multiplicity) inside $B(0,\lambda^0-\ep)$.}
\end{enumerate}
We deduce that there exists a neighborhood $W$ of $0\in\pD$ such that for every $w\in W$ the polynomial $u\to\delta(u,w)$ has $s$ roots outside $B(0,\lambda_0-\ep/2)$ and $d-s$ roots in $B(0,\lambda_0-\ep)$.

Denote by $P_r$ be the space of monic homogeneous polynomials of degree $r$.
Let $F\subset P_s$, resp. $F'\subset P_{d-s}$, be the open subsets consisting of those polynomials whose roots lie in $\mb{C}-B(0,\lambda_0-\ep/2)$, resp. $B(0,\lambda_0-\ep)$. 
The product map $\psi:F\times F'\to P_d$, given by $\psi(p,q)=p\cdot q$, is injective and holomorphic, hence a biholomorphism onto its image. 
As a  result the map 
\[w \in W \mapsto \psi\inverse (\delta( u , w)) = (p^s_w(u), p^{d-s}_w(u)) \in F \times F'\]
is holomorphic. 

We claim that 
\begin{equation}\label{e.attractor}
A^s(w)={\rm ker}\left(p^s_w(\hat g(w))\right).
\end{equation}
Indeed, the characteristic polynomial $\delta(u,w)={\rm det}(u\mb{I}-\hat g(w))$ of $\hat g(w)$ splits as $\delta(u,w)=p^s_w(u)\cdot p_w^{d-s}(u)$.  Then $A^s(w)$ is the direct sum of the eigenspaces corresponding to the roots of the factor $p^s_w(u)$ and is thus the attracting fixed point of $g(w)$ in $\Gr_s(\C^d)$. 
This concludes the proof since the right-hand side in Equation \eqref{e.attractor} depends holomorphically on $w$.
\end{proof}

Let now $\rho_\cdot:\pD\to\Xi^k(\G,\PSL(d,\C))$ is a holomorphic family of $k$-hyperconvex representations. We choose a holomorphic lift $\hat\rho_\cdot:\pD\to\Hom(\G,\PSL(d,\C))$ and, for $w\in\pD$ and $s\in\{k-1, k+1, d-k\}$, we denote by $(\xi^{s})_w:\deG\to\Gr_{s}(\C^d)$  the $s$-th boundary map  of the representation $\hat\rho_w$.
It follows from Proposition \ref{prop:attractorsarehol} that $(\xi^{s})_w$ is holomorphic.
\begin{cor}\label{cor:bdryarehol}
Let 
$\hat\rho_\cdot:\pD\to\Hom(\G,\PSL(d,\C))$ be a holomorphic family of $s$-Anosov representations. Then for every $z\in\deG$, 
$$w\mapsto (\xi^s)_w(z)\in\Gr_s(\C^d)$$
is holomorphic.
\end{cor}
\begin{proof}
Since the Anosov boundary map is dynamics preserving, the claim follows directly from Proposition \ref{prop:attractorsarehol} in the case $z=\g^+$ for some $\g\in\G$. Let now $z$ be generic, and choose a sequence $\g_n\in\G$ such that $\g_n^+\to z$. It follows from \cite[Proposition 6.2]{BPS} that the holomorphic maps $w\mapsto (\xi^s)_w(\gamma_n^+)$ converge uniformly to the map $w\mapsto (\xi^s)_w(z)$, which is thus holomorphic.
\end{proof}
In the situation above we further denote by  $(\xi^{k}_t)_w:\RP\to\CP$ the tangent projection associated to the representation $\rho_w$: For this, as in Section \ref{sec:tangent}, we choose consistent trivializations of the bundles $\cal B^k_{\hat \rho_w}$ depending on the choice of three fixed points $x,y,z\in\deG$; the resulting maps don't depend on the lift $\hat \rho_w$. 
\begin{prop}\label{p.holmaps}Let $\rho_\cdot:\pD\to\Xi^k(\G,\PSL(d,\C))$ be a holomorphic family of $k$-hyperconvex representations. Then for every $t\in\deH$, and every $z\in\RP$ the function 
$$w\mapsto (\xi^k_t)_w(z)\in\CP$$
is holomorphic.
\end{prop}
\begin{proof}
By conjugating with a holomorphic family of transformations $M_w\in{\rm PSL}(d,\mb{C})$, we can normalize the representations $\hat\rho_w$ so that, for every $w\in\pD$, the subspaces $t^{k+1}:=(\xi^{k+1})_w(t)$, $t^{k-1}:=(\xi^{k-1})_w(t)$ do not depend on $w$. Denote by $\cal O\subset \Gr_{d-k}$ the open subset consisting of subspaces transverse to $t^{k-1}$ and intersecting $t^{k+1}$ in one line. Since then the map 
$$\begin{array}{cccc}
I:\cal O\subset \Gr_{d-k}(\C^d)&\to&\P(t^{k+1}/t^{k-1})\\
X&\mapsto & [X\cap t^{k+1}]
\end{array}$$
is algebraic, it follows, on the one hand, that our chosen trivializations
$T_w: \P(t^{k+1}/t^{k-1})\to \CP$ depend holomorphically on $w$, and on the other hand that the tangent projection 
$$(\xi^k_t)_w(z)=T_w\circ I\circ (\xi^{d-k})_w $$
is holomorphic being a composition of holomorphic maps.
\end{proof} 

In other words, $(\xi^k_t)_w$ gives rise to a holomorphic motion.  
\begin{dfn}[Holomorphic motion]
Let $\Lambda \subset \CP$ be a set.
A map $\xi:\Lambda\times\pD \to \mb{CP}^1$ is a \emph{holomorphic motion} if
    it satisfies the following properties:
\begin{itemize}
\item{For every $s\in\Lambda$ the map $w\mapsto \xi(s,w)$ is holomorphic.}
\item{For every $w\in\pD$ the map $s\mapsto \xi(s,w)$ is injective.}
\item{For every $s\in\Lambda$ we have $\xi(s,0)=s$.}
\end{itemize}
\end{dfn}
The following is known as the Holomorphic Motion Theorem; see \cite{FJW:holomorphicmotion} for the history and a proof.
\begin{thm}\label{t.holmotion}
Let $\xi:\Lambda\times\pD \to \mb{CP}^1$ be a holomorphic motion. Then there is a holomorphic motion $g:\mb{CP}^1\times\pD\to\mb{CP}^1$ extending $\xi$.
Moreover,
\begin{itemize}
\item $g$ is continuous.
\item For each $w\in \mathbb D$, $z\mapsto g(z,w)$ is a $\frac{1+|w|}{1-|w|}$-quasi-conformal homeomorphism.
\item The family of Beltrami differentials 
\[\mu_g(z,w) = \frac{\partial g/\partial \overline z}{\partial g/\partial z } (z,w)\]
defines a holomorphic map $\mu_g$ from $\pD$ to the unit ball in the Banach space $L^\infty (\CP)$ of essentially bounded measurable functions.
\end{itemize}
\end{thm}

   Recall from \S\ref{sec:Bersembedding} that for any $[f: M_\Gamma \to W] \in \T(M_\Gamma)$, the laminated Bers' embedding $\beta_W: \T(M_\Gamma) \to Q(\overline W)$ is a holomorphic injection. 
    In the cover of $\overline W$ corresponding to $\Gamma$, consider a leaf $\overline L \cong \overline \H^2 \times \{t\}$.
    Define 
    \begin{equation}\label{eqn: bers leaf}
        q: \T(M_\Gamma) \xrightarrow{\beta_W} Q(\overline W) \to Q(\overline L),
    \end{equation}
    where the second arrow is the restriction to a holomorphic quadratic differential on $\overline L$.
    This map is a holomorphic embedding (see \S\ref{sec:Bersembedding} or \cite[\S3]{Sullivan}).
\begin{proof}[Proof of Theorem \ref{thm:ab map}]
    Let $\rho_\cdot : \pD \to \Xi^k(\Gamma, \PSL(d,\C))$ be a holomorphic disk.
    For $w\in \pD$,  denote $\mathrm{AB}^k(\rho_w) = (E_w, F_w) \in \T(M_\Gamma) \times \T(\overline M_\Gamma)$.
    Note that $\mathrm{AB}^k \circ \rho_\cdot$ is holomorphic if and only if $w \mapsto E_w$ and $w\mapsto F_w$ are each holomorphic.
    We will give the argument that the first map is holomorphic; the other is similar.
    
    The proof is really a direct consequence of Proposition \ref{p.holmaps}, the Holomorphic Motion Theorem \ref{t.holmotion},  the construction of the Bers' embedding giving $\T(M_\Gamma)$ is complex structure (Theorems \ref{thm: Sullivan Bers embedding} and \ref{thm: Bers holomorphic}), and holomorphic depedence on parameters in the Measurable Riemann Mapping Theorem \ref{thm:mrm}.
    The details are outlined below.
    
    Consider the map $q$ defined as in \eqref{eqn: bers leaf} for $W = E_0$ and $t \in \partial \H^2$ chosen so that the leaf $L$ maps injectively into $E_0$.
    Let $\Lambda_0 = (\xi_t^k)_0(\partial \H^2)$.
    Using Proposition \ref{p.holmaps}, the map 
    \[(z,w) \in \Lambda_0 \times \pD \mapsto (\xi_t^k)_w \circ (\xi_t^k)_0\inverse(z)\in \CP\]
    is a holomorphic motion.
    By Theorem \ref{t.holmotion}, there is a holomorphic motion $g: \CP \times \pD \to \CP$ extending it.

    The leaf $L$ of $E_0$ corresponding to $t$ is identified with a component of $\CP \setminus \Lambda_0$.
    Let $u : \H^2 \to L \subset \CP\setminus \Lambda_0$ be the inverse of a uniformizing map.
    Denote by $h: \H^2 \times \pD \to \CP$ the map $h(z,w) = g(u(z),w)$.
    Define a family of Beltrami differentials on $\CP$ by the rule
    \[\mu(z,w) = \begin{cases}
        \frac{\partial h/\partial \overline z}{\partial h/\partial z}(z,w), & z \in \H^2\\
        0, & z\in \overline \H^2.
    \end{cases}\]
    Using the final bullet point of Theorem \ref{t.holmotion} and the transformation rule \eqref{eqn: beltrami composite} for Beltrami coefficients under pre-composition with a holomorphic map, the map $w\in \pD\mapsto \mu(\cdot, w) \in L^\infty(\CP)$ is holomorphic.

    Let $G^w$ be the unique normalized solution of 
    \[\frac{\partial G^w}{\partial z} \mu(\cdot, w) = \frac{ \partial G^w}{\partial \overline z}.\]
    given by Theorem \ref{thm:mrm}; since $w\mapsto \mu(\cdot, w)$ is holomorphic, the maps $G^w$ vary holomrphically in $w$.
    
    By inspection of \eqref{eqn: def schwarz}, we see that the Schwarzian derivative $\mathcal S (G^w|_{\overline \H^2}) \in Q( \overline \H^2)\cong Q(\overline L)$ varies holomorphically in $w$.
    From the construction of Bers' embedding $\beta_{E_0}$, we see that $q(\rho_w) = \mathcal S (G^w|_{\overline \H^2})$.
    Bers' embedding (with respect to any basepoint) defines the complex structure on $\T(M_\Gamma)$ (Theorem \ref{thm: Bers holomorphic}).  This concludes the proof that $\mathrm{AB}^k$ is holomorphic.
\end{proof}

\subsection{Classical Ahlfors--Bers for the irreducible representation}

We conclude the proof of Theorem \ref{thm:IntroAB} by showing the compatibility of ${\rm AB}^k$ with the classical Ahlfors--Bers map on the quasi-Fuchsian locus.

\begin{prop}
\label{pro:ab irred is ab}
Let $\rho:\Gamma\to{\rm PSL}(d,\mb{C})$ be the composition of a quasi-Fuchsian representation $\eta:\Gamma\to{\rm PSL}(2,\mb{C})$ with the irreducible representation ${\iota_d}:{\rm PSL}(2,\mb{C})\to{\rm PSL}(d,\mb{C})$. Then ${\rm AB}^k(\rho)={\rm AB}(\eta)$ for every $1\le k\le d-1$.
\end{prop}

The irreducible representation ${\iota_d}:{\rm PSL}(2,\mb{C})\to{\rm PSL}(d,\mb{C})$ is the  natural ${\rm PSL}(2,\mb{C})$-action on the projectivization $\mb{P}(\mb{C}_{d-1}[X,Y])$  of the $d$-dimensional $\mb{C}$-vector space consisting of degree $(d-1)$ homogeneous polynomials in two variables. More explicitly, for a matrix $[A]=\left[\begin{smallmatrix}a&b\\c&d\end{smallmatrix}\right]\in{\rm PSL}(2,\mb{C})$ we have
\[
[A]\cdot[Q(X,Y)]:=[Q(aX+bY,cX+dY)].
\]
Equivariant with respect to $\iota_d$ is the Veronese embedding in the full flag manifold
\[
\begin{array}{cccl}
\nu:&\mb{P}(\mb{C}_{1}[X,Y])&\to&\mc{F}(\mb{C}_{{d-1}}[X,Y])\\
&Q&\mapsto&\left(\nu^k(Q):=Q^{d-k}\cdot\mb{C}_{k-1}[X,Y]\right)_{k\le d-1}
\end{array}
\]
where $Q^{d-k}\cdot\mb{C}_{k-1}[X,Y]$ is the $k$-dimensional vector subspace consisting of the  homogeneous polynomials of degree $d-1$ that are multiples of $Q^{d-k}$. 

For every $Q\in\mb{C}_{1}[X,Y]$ and $1\le k\le d-1$ we have a tangent projection
\[\begin{array}{cccl}
\nu^k_Q:&\mb{P}(\mb{C}_{1}[X,Y])&\to&\mb{P}(\nu^{k+1}(Q)/\nu^{k-1}(Q))\\
&P&\mapsto&\left\{
\begin{array}{ll}
[\nu^k(Q)] &\text{if }P=Q,\\
{[\nu^{k+1}(Q)\cap \nu^{d-k}(P)]} &\text{otherwise}
\end{array}\right.
\end{array}
\]
In order to see that $\nu_Q^k$ is well-defined, note that if $P\neq Q$ then
\begin{equation}\label{e.VeroneseTangent}
\nu^{k+1}(Q)\cap \nu^{d-k}(P)=\langle Q^{d-1-k}P^k\rangle
\end{equation}
and $Q^{d-1-k}P^k$ is contained in $\nu^{k+1}(Q)$ but not in $\nu^{k}(Q)$ (since $Q^{d-k-1}$ divides it but $Q^{d-k}$ does not) so $Q^{d-k-1}P^k+\nu^{k-1}(Q)$ has dimension $k$ and is contained in $\nu^{k+1}(Q)$. 

\begin{lem}
\label{lem:veronese holo}
For every $Q$ and $k$ the map $\nu^k_Q$ is a biholomorphism.
\end{lem}

\begin{proof}
The crux of the proof is to check continuity, as the expression for the tangent projection given in Equation \eqref{e.VeroneseTangent} is not well defined for $P=Q$, since $Q^{d-1}\in\nu^{k-1}(Q)$. Consider a sequence $P_n$ of linear homogeneous polynomials converging to $Q$. We can write them as $P_n=Q+E_n$ with $E_n=\alpha_n X+\beta_n Y$ converging to $0$. Using the symmetry of the following argument, we may assume without loss of generality that $|\alpha_n|\ge|\beta_n|$ (up to passing to a subsequence). We have
\[
Q^{d-k-1}P_n^k=Q^{d-k-1}(Q+E_n)^k=Q^{d-k-1}E_n^k+kQ^{d-k}E_n^{k-1}+R_n
\]
with $R_n\in Q^{d-k+1}\cdot\mb{C}_{d-k+2}[X,Y]=\nu^{k-1}(Q)$. Thus
\begin{align*}
\langle Q^{d-k-1}P_n^k\rangle+\nu^{k-1}(Q) &=\langle Q^{d-k-1}E_n^k+kQ^{d-k}E_n^{k-1}\rangle+\nu^{k-1}(Q)\\
 &=\left\langle\frac{1}{\alpha_n^{k-1}}(Q^{d-k-1}E_n^k+kQ^{d-k}E_n^{k-1})\right\rangle+\nu^{k-1}(Q).
\end{align*}
Recall that $E_n^s=(\alpha_nX+\beta_nY)^s=\sum_{i\le s}{{s\choose i}\alpha_n^i\beta_n^{s-i}X^iY^{s-i}}$ so
\[
\frac{1}{\alpha_n^{k-1}}E_n^s=\sum_{i\le s}{{s\choose i}\alpha_n^{i-k+1}\beta_n^{s-i}X^iY^{s-i}}.
\]

We will show that (up to subsequences) 
\begin{equation}\frac{1}{\alpha_n^{k-1}}E_n^{k-1}\to E_\infty=X^{k-1}+\cdots\end{equation} and 
\begin{equation}
\frac{1}{\alpha_n^{k-1}}E_n^k\to 0
\end{equation} as $n\to\infty$. Before doing this computation, let us observe that this implies in particular that, up to subsequences, we have 
\[
\left\langle\frac{1}{\alpha_n^{k-1}}(Q^{d-k-1}E_n^k+kQ^{d-k}E_n^{k-1})\right\rangle+\nu^{k-1}(Q) \to \langle Q^{d-k}E_\infty\rangle+\nu^{k-1}(Q)=\nu^k(Q).
\]
As the limit does not depend on the chosen subsequence, we deduce that the whole sequence converges to it. This shows continuity of the map $\nu^k_Q$.

We now discuss the convergence of the two sequences. Consider the first term
\[
\frac{1}{\alpha_n^{k-1}}E_n^{k-1}=\sum_{i\le k-1}{{k-1\choose i}\alpha_n^{i-k+1}}\beta_n^{k-1-i}X^iY^{k-1-i}
\]
By assumption $|\alpha_n|\ge|\beta_n|$, so every term $\alpha_n^{-k+i+1}\beta_n^{k-i-1}$ has absolute value bounded above by 1. Up to subsequence we may assume that all these coefficients converge.  

Consider then the second term
\[
\frac{1}{\alpha_n^{k-1}}E_n^{k}=\sum_{i\le k}{{k\choose i}\alpha_n^{i-k+1}\beta_n^{k-i}X^iY^{k-i}}.
\]
Note that $|\alpha_n^{i-k+1}\beta_n^{k-i}|=|(\beta_n/\alpha_n)^{k-i-1}\beta_n|\le|\beta_n|$ as $|\alpha_n|\ge|\beta_n|$. Since $\beta_n\to 0$ we conclude that all coefficients are converging to 0 as well.    

In order to conclude that $\nu^k_Q$ is biholomorphic, it is enough to check that it is injective and holomorphic on $\mb{P}(\mb{C}_{1}[X,Y])-\{Q\}$. Holomorphicity is a straightforward consequence of the formula 
\[
\nu^k_Q(P)=[Q^{d-k}P^k]\in\mb{P}(\nu^{k+1}(Q)/\nu^{k-1}(Q)).
\]
As for injectivity, suppose that $[Q^{d-k-1}P_1^k]=[Q^{d-k-1}P_2^k]$, or, in other words, $Q^{d-k-1}(P_1^k-P_2^k)\in \nu^{k-1}(Q)$. Then $Q^{d-k+1}$ divides $Q^{d-k-1}(P_1^k-P_2^k)$ or, equivalently, $Q^2$ divides $P_1^k-P_2^k$. However, $P_1^k-P_2^k=\prod_{j=1}^k(P_1-\zeta_k^jP_2)$, with $\zeta_k$ a primitive $k$-th root of unity, has only simple roots, so it cannot be a multiple of $Q^2$. This proves injectivity.
\end{proof}

We can now prove Proposition \ref{pro:ab irred is ab}.

\begin{proof}[Proof of Proposition \ref{pro:ab irred is ab}]
The boundary map of the representation $\rho={\iota_d}\circ\eta$ is the composition $\xi: \partial\Gamma\xrightarrow{\zeta}\mb{CP}^1\xrightarrow{\nu}\mc{F}(\mb{C}^d)$ of the boundary map $\zeta$ of $\eta$ with the Veronese embedding $\nu$.  
Denote by $\Lambda = \zeta(\partial \Gamma) \subset \CP$. We then have, for every $z\in\deG=\deH$, $\xi^k_z(\RP)=\nu_{\zeta(z)}(\Lambda)\subset \mb P (\xi^{k+1}(z)/\xi^{k-1}(z))$.

Using  Lemma \ref{lem:veronese holo}, for every $z\in \partial\Gamma$ 
\[\nu_{\zeta(z)}^k:\CP \times \{z\} \to  \mb P (\xi^{k+1}(z)/\xi^{k-1}(z))\]
is a bi-holomorphism inducing the identification of $\Lambda \times \{z\}$ and $\xi^k_z(\deG)\subset \mb P (\xi^{k+1}(z)/\xi^{k-1}(z))$.
Continuity of  
\[(w,z)\in \CP\times \partial \Gamma \mapsto \nu_{\zeta(z)}^k(w) \in \mc B_\rho^k\]
follows because it is leafwise holomorphic and extends the continuous map 
\[(w,z)\in \Lambda \times \partial \Gamma \mapsto \xi_z^k(\zeta\inverse(w)).\]

Thus, $\nu^k$ is an equivalence of Riemann surface laminations and restricts to an equivariant equivalence 
\[\left(\CP- \Lambda \right) \times \partial \Gamma \to \mc B_\rho^k - \fL_\rho^k\]
of Riemann surface laminations.
\end{proof}
    Many more $\SL(2,\R)$-embeddings give rise to $k$-hyperconvex representations. More specifically, for integers $d_1\geq \ldots \geq d_s$ such that $d_1+\ldots+d_s=d$, denote by $\iota_{d_1,\ldots, d_s}:\SL(2,\C)\to\SL(d,\C)$ the unique representation (up to conjugacy) that preserves a splitting $\C^d=\C^{d_1}\oplus\ldots\oplus \C^{d_s}$ and acts irreducibly on each factor. The restriction of $\iota_{d_1,\ldots, d_s}$ to $\Gamma$, and more generally the composition of $\iota_{d_1,\ldots, d_s}$ with a quasi-Fuchsian representation $\eta:\G\to\SL(2,\C)$ is $k$-hyperconvex if and only if $d_1>d_2+2k$, and in this case the $s$-th boundary map, for $s\leq k+1$, is contained in $\Gr_s(\C^{d_1})\subset\Gr_s(\C^{d})$. As a result, it directly follows from Proposition \ref{pro:ab irred is ab}: 
\begin{cor}\label{cor:ab is ab}
Let $\eta:\Gamma\to{\rm PSL}(2,\mb{C})$ be a quasi-Fuchsian representation, $d_1\geq \ldots \geq d_s$ integers such that $d_1+\ldots+d_s=d$.   Then for every $k< \frac12(d_1-d_2)$, 
$${\rm AB}^k(\iota_{d_1,\ldots, d_s}\circ \eta)={\rm AB}(\eta).$$
\end{cor}

\section{Fully hyperconvex representations, proof of Theorem \ref{thm:ab intro 2}}
We now turn to fully hyperconvex representations and prove Theorem \ref{thm:ab intro 2}. Holomorphicity of the full laminated Ahlfors--Bers map
\[
{\rm AB}:=({\rm AB}^1,\cdots,{\rm AB}^{d-1}):\Charn\to(\T(M_\Gamma)\times\T(\overline M_\Gamma))^{d-1}
\]
follows directly from Theorem \ref{thm:ab map} proven in the previous section, we will discuss closedness, injectivity, as well as the characterization of the preimage of the diagonal in the next three subsections.
In this section, we continue our abuse of notation, dropping markings and brackets denoting equivalence classes, throughout.

\subsection{Properness}
We equip the Teichm\"uller space of Riemann surface laminations $\T(M_\Gamma)$ with its Teichm\"uller distance (Definition \ref{d.Teichdis}). The goal of this  subsection is to prove the following: 
\begin{thm}\label{t.properness}
The pre-image of a bounded set under the map 
$${\rm AB}:\Charn\to(\T(M_\Gamma)\times\T(\overline M_\Gamma))^{d-1}$$ 
is pre-compact. 
\end{thm}
We regard this as a good generalization of properness for the map ${\rm AB}$ which has values in the infinite dimensional (non-locally-compact) space $(\T(M_\Gamma)\times\T(\overline M_\Gamma))^{d-1}$.

We claim that Theorem \ref{t.properness} implies that $\mathrm{AB}$ is closed, as claimed in Theorem \ref{thm:IntroAB}.
Indeed, let $K\subset \Charn$ be a closed set and let $\{\rho_n\} \subset K$ be a sequence such that $\mathrm{AB}(\rho_n)$ converges in $(\T(M_\Gamma)\times\T(\overline{M}_\Gamma))^{d-1}$.
This implies that $\{\mathrm{AB}(\rho_n)\}$ is contained in a bounded set for the (product) Teichm\"uller metric, so that $\{\rho_n\}$ is precompact by the theorem.  Then any accumulation point $\rho$ is contained in $K$, and $\mathrm{AB}(\rho_n) \to \mathrm{AB}(\rho)$ holds by continuity.  Thus $\mathrm{AB}(K)$ is closed.

\medskip
The first step of the proof of Theorem \ref{t.properness} is to show that the complex dilation spectrum of a quasi-conformal deformation whose Ahlfors--Bers parameters lie in a bounded set of $\T(M_\Gamma)\times\T(\overline M_\Gamma)$ is uniformly controlled. Denote by $\Sigma=\mb{H}^2/\Gamma$ our reference hyperbolic surface.

\begin{prop}\label{p.cpxLengthbound}
Let $B\subset \T(M_\Gamma)\times\T(\overline M_\Gamma)$ be a bounded set. Then there exists $K\geq 1$ such that, for every $\rho\in\qc(\iota)$ with $\mc{AB}(\rho)\in B$, 
\[
\frac{1}{K}\ell_\Sigma(\cdot)\le\log|L_\rho(\cdot)|\le K\ell_\Sigma(\cdot).
\]   
\end{prop}

\begin{proof}[Proof of Proposition \ref{p.cpxLengthbound}]
As ${\rm AB}(\rho)=(E_{\rho},F_{\rho})$ lies in the bounded set $B$, there exists $\kappa> 0$ such that
\[
d_\T(E_{\rho},\Sigma)+d_\T(F_{\rho},\Sigma)<\kappa.
\]  
In turn this implies that representation  $\rho:\Gamma\to\mc{MG}$ is $e^{2\kappa}$-quasi-conformally conjugate to the fixed Fuchsian one $\iota$ (see  Remark \ref{rmk:conformal welding}).
 
Recall that for every $\gamma\in\Gamma$, the transformation 
\[
\rho(\gamma):\mb{CP}^1\times\{\gamma^+\}\to\mb{CP}^1\times\{\gamma^+\}
\]
is a loxodromic Möbius transformation which is conjugate to $z\mapsto L_\rho(\gamma)z$. 
Similarly, $\iota(\gamma)$ is conjugate to $z\mapsto e^{\ell_\Sigma(\gamma)}z$.

Let $f:\mb{CP}^1\to\mb{CP}^1$ be a $(\rho(\gamma),\iota(\gamma))$-equivariant $e^{2\kappa}$-quasi-conformal conjugacy.
Using a theorem of Reimann \cite{Reimann:extension} (see Theorem \ref{thm:reimann}), we conclude that $f$ extends to a $K=e^{6\kappa}$-bi-Lipschitz equivalence $\H^3/\langle \rho(\gamma)\rangle \to \H^3/\langle \iota(\gamma)\rangle$.
The result follows, since $\log|L_\rho(\gamma)|$ (respectively $\ell_\Sigma(\gamma)$) is the hyperbolic length of the core curve of the solid torus $\H^3/\langle \rho(\gamma)\rangle$ (respectively $\H^3/\langle \iota(\gamma)\rangle$).
\end{proof}
Next, we want to show that any sequence of fully hyperconvex representations all of whose Ahlfors--Bers parameters are bounded is bounded in the character variety. 

Denote by $\mc{X}$  the symmetric space of ${\rm PSL}(d,\mb{C})$. Let $F\subset{\rm PSL}(d,\mb{C})$ be a finite set of isometries of $\mc{X}$. For every $x\in\mc{X}$ the \emph{displacement of $F$ at $x$} is
\[
D(F,x):=\max_{s\in F}d_\mc{X}(x,sx).
\]
Taking the infimum over points in $\mc{X}$, we get the \emph{minimal joint displacement}
\[
D(F):=\inf_{x\in\mc{X}}D(F,x).
\]
The \emph{translation length } of an element $s\in F$ is defined as $\ell(s):=\lim_{n\to\infty}{\frac{d(x,s^nx)}{n}}$, and does not depend on $x\in\mc{X}$. We define
\[
\ell(F):=\max_{s\in F}\ell(s).
\]
The following result says that the minimal joint displacement of $F$ is controlled by the stable length of elements in $F^k$, the sets of words of length at most $k$ in elements of $F$.

\begin{prop}[{see \cite[Proposition 1.6]{BF21}}] 
\label{pro:bf}
There exist $k>0$ and $C>0$ depending only on $d$ such that for every finite set $F\subset{\rm PSL}(d,\mb{C})$ we have
\[
D(F)\le\sqrt{d}(\ell(F^{k})+C).
\]
\end{prop}

We use Proposition \ref{pro:bf} to obtain convergence in our setting. 

\begin{prop}\label{p.converges}
Let $D=B_1\times\cdots\times B_{d-1}$ be the product of the bounded sets $B_k\subset\T(M_\G)\times\T(\overline M_\G)$. Let $(\rho_n)$ in ${\rm AB}^{-1}(D)\subset \Xi^{\rm hyp}(\G,\PSL(d,\C))$ be a sequence of representations  then $(\rho_n)$ admits a convergent subsequence in $\Hom(\G,\PSL(d,\C))/\!\!/\PSL(d,\C)$. 
\end{prop}
\begin{proof}
Let $S$ be a finite set of generators for $\Gamma$. In order to show convergence of $\rho_n$ in the character variety it is enough to show that $\sup_{n\in\mb{N}}{D(\rho_n(S))}<\infty$. In fact, if this happens, we can assume, up to suitably conjugating the representations $\rho_n$, that there is a point $x\in\mc{X}$ such that
\[
D(\rho_n(S),x):=\max_{s\in S}\{d_\mc{X}(x,\rho_n(s)x)\}
\]
is bounded independently of $n$. This implies that for every $s\in S$, the sequence $\rho_n(s)$ is contained in a compact subset of ${\rm PSL}(d,\mb{C})$ and, hence, converges up to subsequences to an element $\rho(s)$.

Recall that the Riemannian translation distance of a loxodromic element $A\in{\rm PSL}(d,\mb{C})$ can be computed as
\[
\ell(A) =\lim_{n\to\infty}{\frac{d_\mc{X}(x,A^nx)}{n}}=\sqrt{\log|\lambda^1(A)|^2+\ldots+\log|\lambda^d(A)|^2}\le\sqrt{d}\log|\lambda^1(A)|,
\]
where $\lambda^1(A),\cdots,\lambda^d(A)$ denote the eigenvalues of some lift of $A$ to  $\in{\rm SL}(d,\mb{C})$ with $|\lambda_i|\ge |\lambda_{i+1}|$; note that their magnitudes and ratios do not depend on the choice of lift. 
As $L_\rho^k(\gamma)=\lambda^k(\rho(\gamma))/\lambda^{k+1}(\rho(\gamma))$ denotes the $k$-th eigenvalue gap, we have 
\[
\lambda^{k+1}(\rho(\gamma))=\frac{\lambda^k(\rho(\gamma))}{L_\rho^k(\gamma)}=\cdots=\frac{\lambda^1(\rho(\gamma))}{L_\rho^1(\gamma)\cdots L_\rho^k(\gamma)}.
\]

Using $\lambda^1\cdot\ldots\cdot\lambda^d=1$, we deduce
\[
\lambda^1(\rho(\gamma))^d=L^1_\rho(\gamma)^{d-1}\cdot L^2_\rho(\gamma)^{d-2}\cdot\ldots\cdot L_\rho^{d-1}(\gamma).
\]
The boundedness assumptions on the Ahlfors--Bers parameters and Proposition \ref{p.cpxLengthbound}  gives us a $K\ge1$ with $\log |L_\rho^i(\gamma)|\le K\ell_\Sigma(\gamma)$ for all $\gamma \in \Gamma$ and $1\le i\le d-1$.
By arithmetic and the equality above, we conclude
\[
\log |\lambda^1(\rho(\gamma))|\le \frac{d-1}{2}K\ell_\Sigma(\gamma).
\] 
Thus for every $\rho \in D$, our bounded set, and for every $\g\in\G$, $\ell(\rho(\gamma))$ is  bounded from above depending on $\ell_\Sigma(\gamma)$ and $D$.
\end{proof}

\begin{prop}
In the situation of Proposition \ref{p.converges}, the limiting representation $\rho$ is $\{1,\cdots,d-1\}$-Anosov.    
\end{prop}

\begin{proof}
By Proposition \ref{p.cpxLengthbound} and continuity, for every $k\le d-1$, we have $\log|L^k_\rho(\gamma)|\ge\ell_\Sigma(\gamma)/K$. Since the orbit map $\Gamma \to \H^2$ is a quasi-isometry, $\ell_\Sigma(\gamma)\ge|\gamma|/c-c$ for some constant $c>0$, and the conclusion follows.
\end{proof}

The Anosov property gives us a $\rho$-equivariant boundary map $\xi:\partial\Gamma\to\mc{F}(\mb{C}^d)$. Furthermore, by the continuous dependence of such maps from the representation \cite[Theorem 5.13]{GW:Anosov_DOD},   the boundary maps $\xi_n:\partial\Gamma\to\mc{F}(\mb{C}^d)$ of $\rho_n$ converges to $\xi$ uniformly. 

Recall from Proposition \ref{p.tangentproj}
 that, using the Anosov boundary map $\xi$, we can define tangent projections $\xi_z^k:\deG\to\P(z^{k+1}/z^{k-1})$  for every $z\in\deG$ and every $1\leq k\leq d-1$. The representation is hyperconvex if and only if all the maps $\xi^k_z$ are injective.

\begin{prop}
In the situation of Proposition \ref{p.converges}, the limiting representation $\rho$ is fully hyperconvex.    
\end{prop}

\begin{proof}
Fix an integer $k\leq d-1$. Given four cyclically ordered distinct points $(a,b,c,d)$ on $\partial\Gamma$, and a fifth point $t\in\deG$, we can consider the tangent projections
\[
\xi^k_t(a),\xi^k_t(b),\xi^k_t(c),\xi^k_t(d).
\]
Since the boundary map  depends continuously on the representation, we have
\[
(\xi_t^k)_n(a),(\xi_t^k)_n(b),(\xi_t^k)_n(c),(\xi_t^k)_n(d)\to\xi^k_t(a),\xi^k_t(b),\xi^k_t(c),\xi^k_t(d)
\]
Indeed we can assume, up to conjugating the representation with a sequence of elements in $\PSL(d,\C)$, that the subspaces $t^{k+1}_{\rho_n}$ do not depend on the representation.

Since ${\rm AB}^k(\rho_n)$ lies in a bounded set of $\T(M_\Gamma)\times\T(\overline{M}_\Gamma)$, the actions $\rho_n^k$ are uniformly quasi-conformally conjugated to any fixed Fuchsian one.   That is $\rho_n^k=g_n^k\iota ({g_n^k})^{-1}$ for some $K$-quasi-conformal automorphism $g_n^k\in{\rm Aut}^0(\mb{CP}^1\times\partial\mb{H}^2)$. In particular
\[
((\xi_t^k)_n(a),(\xi_t^k)_n(b),(\xi_t^k)_n(c),(\xi_t^k)_n(d))=(g_n^k(a,t),g_n^k(b,t),g_n^k(c,t),g_n^k(d,t)).
\]

As the set of $K$-quasi-conformal homeomorphism of $\mb{CP}^1$ fixing 3 points (such as $0,1,\infty$ as each $g_n^k(\cdot,t)$ does) is compact, $g_n^k(\cdot, t)$ subconverges uniformly to a quasi-conformal homeomorphism $g^k(\cdot, t)$. Hence, we have
\[
(\xi^k_t(a),\xi^k_t(b),\xi^k_t(c),\xi^k_t(d))=(g^k(a,t),g^k(b,t),g^k(c,t),g^k(d,t)).
\]
As $g^k$ is a homeomorphism, the four points $\xi^k_t(a),\xi^k_t(b),\xi^k_t(c),\xi^k_t(d)$ are distinct. 
\end{proof}

The proof of Theorem \ref{t.properness} is complete, because we have shown that any sequence $\rho_n \in \Char$ mapping to a bounded set under $\mathrm{AB}$ has a convergent subsequence and the limiting representation is fully hyperconvex.

\subsection{Injectivity}\label{subsubsection: injectivity}
The key step in the proof of injectivity of the full laminated Ahlfors--Bers map is the following result, whose proof will occupy most of the subsection.
As always, $\Gamma$ is assumed to be the fundamental group of a closed oriented surface.

\begin{thm}
\label{thm:same gaps conjugate}
Let $\rho,\rho':\Gamma\to{\rm PSL}(d,\mb{C})$ be representations such that  for every $\gamma\in\Gamma$, the matrices  $\rho(\gamma),\rho'(\gamma)$ are diagonalizable with  eigenvalues of distinct moduli. Suppose that 
\[
\frac{\lambda^j(\rho(\gamma))}{\lambda^{j+1}(\rho(\gamma))}=\frac{\lambda^j(\rho'(\gamma))}{\lambda^{j+1}(\rho'(\gamma))}
\]
for every $\gamma\in\Gamma$ and $j\le d-1$. Then $\rho$ and $\rho'$ are conjugate.
\end{thm}
That $\Gamma$ is a closed surface group is used in the proof of the following lemma.
\begin{lem}
\label{lem:lift}
Let $\Gamma_0<\Gamma$ be a subgroup of index a multiple of $d$. The restriction of every representation $\eta:\Gamma\to{\rm PSL}(d,\mb{C})$ to $\G_0$  lifts to ${\rm SL}(d,\mb{C})$.    
\end{lem}

\begin{proof}
There is a cohomology class $o\in H^2({\rm PSL}(2,\mb{C}),\mb{Z}/d\mb{Z})$ (corresponding to the central extension $\mb{Z}/d\mb{Z}\to{\rm SL}(d,\mb{C})\to{\rm PSL}(d,\mb{C})$) such that a representation $\eta:\Gamma\to{\rm PSL}(d,\mb{C})$ lifts to ${\rm SL}(d,\mb{C})$ if and only if $\eta^*o\in H^2(\Gamma,\mb{Z}/d\mb{Z})$ vanishes (see for example \cite[Lemma 2.4]{Frigerio}). 

Let $i:\Gamma_0\to\Gamma$ be the inclusion.
The degree of the map $\H^2/\Gamma_0 \to \H^2/\Gamma$ is $[\Gamma: \Gamma_0] = dk$ for some $k\ge 1$,
so $i^*:H^2(\Gamma,\mb{Z}/d\mb{Z})\to H^2(\Gamma_0,\mb{Z}/d\mb{Z})$ vanishes identically.
Thus $(\eta\circ i)^*o = i^*(\eta^*o) = 0$, from which the claim follows.
\end{proof}

\begin{proof}[Proof of Theorem \ref{thm:same gaps conjugate}]
For every $i,j$ we have 
\[
\frac{\lambda^i(\rho(\gamma))}{\lambda^j(\rho(\gamma))}=\frac{\lambda^i(\rho'(\gamma))}{\lambda^j(\rho'(\gamma))}.
\]

Fix $\Gamma_0<\Gamma$ an auxiliary finite index subgroup of index $d$. Denote by $\rho_0,\rho'_0:\Gamma_0\to{\rm SL}(d,\mb{C})$ two lifts of the restrictions of $\rho,\rho'$ to $\Gamma_0$ (which exist by Lemma \ref{lem:lift}).

Observe that  
\begin{align}
1 &=\prod_{j\le d}{\lambda^j(\rho_0(\gamma))}\label{e.prodev}\\
 &=\prod_{j\le d}{\lambda^1(\rho_0(\gamma))\frac{\lambda^j(\rho_0(\gamma))}{\lambda^1(\rho_0(\gamma))}}\nonumber\\
 &=\lambda^1(\rho_0(\gamma))^d\cdot\prod_{j\le d}{\frac{\lambda^j(\rho_0(\gamma))}{\lambda^1(\rho_0(\gamma))}}\nonumber
\end{align}
and similarly for $\rho_0'$. As $\frac{\lambda^j(\rho_0(\gamma))}{\lambda^1(\rho_0(\gamma))}=\frac{\lambda^j(\rho'_0(\gamma))}{\lambda^1(\rho'_0(\gamma))}$, we deduce that $\lambda^1(\rho_0(\gamma))^d=\lambda^1(\rho'_0(\gamma))^d$. Thus, for every $\gamma\in\Gamma_0$ there exits a $d$-th root of unity $q_\gamma$ such that 
\[
\lambda^1(\rho_0(\gamma))=\lambda^1(\rho'_0(\gamma))q_\gamma
\]
and 
\[
\lambda^j(\rho_0(\gamma))=\lambda^1(\rho_0(\gamma))\frac{\lambda^j(\rho_0(\gamma))}{\lambda^1(\rho_0(\gamma))}=\lambda^1(\rho'_0(\gamma))q_\gamma\frac{\lambda^j(\rho'_0(\gamma))}{\lambda^1(\rho'_0(\gamma))}=\lambda^j(\rho'_0(\gamma))q_\gamma.
\]

In particular, for every $\gamma\in\Gamma$ we have 
\[
{\rm tr}(\rho(\gamma))={\rm tr}(\rho'(\gamma))q_\gamma.
\]

Define the algebraic subvarieties 
\[
V_k=\left\{(A,B)\in{\rm SL}(d,\mb{C})\times{\rm SL}(d,\mb{C})\left|\,{\rm tr}(A)={\rm tr}(B)e^{2\pi k i/d}\right.\right\}.
\]

Let us now consider the representation $(\rho_0,\rho_0'):\Gamma\to{\rm SL}(d,\mb{C})\times{\rm SL}(d,\mb{C})$. By the above discussion, for any $\g\in \G_0$, we have $(\rho_0(\g),\rho_0'(\g))\in V_0\sqcup\cdots\sqcup V_{d-1}$, and thus the Zariski closure $Z<{\rm SL}(d,\mb{C})\times{\rm SL}(d,\mb{C})$ of $(\rho_0,\rho_0')(\Gamma_0)$ lies in the same union. 
Observe that $Z$ is an algebraic group with finitely many connected components; we denote by $Z_0$ the component that contains the identity matrix. 
It is a subset of the subvariety $V_0\subset{\rm SL}(d,\mb{C})\times{\rm SL}(d,\mb{C})$ that contains the identity. The pre-image $\Gamma_1=(\rho_0,\rho_0')^{-1}(Z_0)$ is a finite index subgroup of $\Gamma_0$.

By construction, for every $\gamma\in\Gamma_1$ we have ${\rm tr}(\rho_0(\gamma))={\rm tr}(\rho'_0(\gamma))$. As traces are coordinates on the ${\rm SL}(d,\mb{C})$-character variety, we conclude that the restrictions of $\rho_0,\rho_0'$ to $\Gamma_1$ are conjugate. 
Theorem \ref{thm:same gaps conjugate} follows then from the next proposition, which ensures that the conjugating element $M\in{\rm SL}(d,\mb{C})$ induces a conjugacy of $\rho$ and $\rho'$ on the whole group $\G$.
\end{proof}

\begin{prop}
\label{claim:from fi to ambient}
Let $\rho,\rho':\G\to\PSL(d,\C)$ be $k$-Anosov for $1\leq k\leq d-1$. Assume that $L_\rho^k(\g)=L_{\rho'}^k(\g)$ for all $\g\in\G$ and all $k$ and that there exists a finite index subgroup $\G_1<\G$ and $M\in\PSL(d,\C)$ such that 
$M\rho M^{-1}=\rho'$ on $\G_1$. Then $M\rho M^{-1}=\rho'$ on $\Gamma$.    
\end{prop}

\begin{proof}
Fix an element $\gamma\in\Gamma$. By the Anosov property, both $\rho(\gamma)$ and $\rho(\gamma')$ determine a splitting of $\mb{C}^d$ as a direct sum of lines, corresponding to their distinct eigenspaces ordered by the moduli of the corresponding eigenvalues. Denote these splittings by 
\[
\mb{C}^d=L_1\oplus\cdots\oplus L_d=L_1'\oplus\cdots\oplus L_d'.
\]
Note that the splittings of $\rho(\gamma),\rho'(\gamma)$ are the same as the ones associated with $\rho(\gamma^N),\rho'(\gamma^N)$ for every $N\ge 1$. 
If $N=[\Gamma:\Gamma_1]$ is the index of $\Gamma_1$ in $\Gamma$, then $\gamma^{N!}\in\Gamma_1$,  so that $\rho'(\gamma^{N!})=M\rho(\gamma^{N!})M^{-1}$.
In particular,  $ML_j=L_j'$ for every $j$. 

Then  $M\rho(\gamma)M^{-1}$ can only differ from $\rho'(\gamma)$ by the multiplication with the projective class of a matrix $D_\gamma$ which is diagonal in any basis associated with $L_j'$. However, we also know that 
\[
\frac{\lambda^i(\rho(\gamma))}{\lambda^j(\rho(\gamma))}=\frac{\lambda^i(\rho'(\gamma))}{\lambda^j(\rho'(\gamma))}
\]
for every $i,j$. Combining this information with the above discussion, we get 
\[
\frac{\lambda^i(\rho(\gamma))}{\lambda^j(\rho(\gamma))}=\frac{\lambda^i(\rho'(\gamma))}{\lambda^j(\rho'(\gamma))}\frac{\lambda^i(D_\gamma)}{\lambda^j(D_\gamma)}
\]
for every $i,j$. Thus ${\lambda^i(D_\gamma)}={\lambda^j(D_\gamma)}$ for every $i,j$ and $D_\gamma$ is a multiple of the identity, thus trivial in $\PSL_d(\C)$. As a consequence, $M\rho(\gamma)M^{-1}=\rho'(\gamma)$ for every $\gamma\in\Gamma$. 
\end{proof}

As a consequence of Theorem \ref{thm:same gaps conjugate} we get:
\begin{prop}
\label{pro:injective}
The map ${\rm AB}:\Charn\to(\T(M_\Gamma)\times\T(\overline M_\Gamma))^{d-1}$ is injective.    
\end{prop}
\begin{proof}
The $k$-th eigenvalue gap $L_\rho^k$ of the representation $\rho$ equals the complex dilation function $L_{\rho^k}$ of the associated laminated conformal action $\rho^k$ (Proposition \ref{p.Clength}). Since, in turn, $\rho^k$  arises as welding of the Ahlfors--Bers parameters of ${\rm AB}^k(\rho)$
(the map $\mc{AB}:\mc{QC}(\iota)\to\T(M_\Gamma)\times\T(\overline M_\Gamma)$ is injective by Proposition \ref{p.ABinjective}).
If $\mathrm{AB}(\rho) = \mathrm{AB}(\rho')$, then $$L^k_\rho(\cdot)=L^k_{\rho'}(\cdot)$$    
for all $k$.
As a result, we have 
\[
\frac{\lambda^i(\rho(\gamma))}{\lambda^j(\rho(\gamma))}=\frac{\lambda^i(\rho'(\gamma))}{\lambda^j(\rho'(\gamma))},
\]
so we can apply Theorem \ref{thm:same gaps conjugate} and conclude the proof.
\end{proof}

\subsection{Real locus}
We conclude the section characterizing the preimage of the diagonal $\Delta\subset\calT(M_\G)\times \calT(\overline M_\G)$, thus concluding the proof of Theorem \ref{thm:ab intro 2}. 
We will build on the following theorem, whose proof will occupy most of the subsection. While the arguments are very similar to those in the proof of Theorem \ref{thm:same gaps conjugate} neither of the two results can be deduced from the other.
\begin{thm}
\label{thm:real gaps real}
Let $\Gamma$ be the fundamental group of a closed surface. Let $\rho:\Gamma\to{\rm PSL}(d,\mb{C})$ be a representation such that $\rho(\gamma)$ is diagonalizable with $d$ eigenvalues of distinct modulus for every $\gamma\in\Gamma-\{1\}$. Suppose that 
\[
\frac{\lambda^j(\rho(\gamma))}{\lambda^{j+1}(\rho(\gamma))}\in\mb{R}
\]
for every $\gamma\in\Gamma$ and $j\le d-1$. Then $\rho$ is conjugated to ${\rm PSL}(d,\mb{R})$.
\end{thm}
\begin{proof}[Proof of Theorem \ref{thm:real gaps real}]
As in the proof of Theorem \ref{thm:same gaps conjugate}, we pass to an auxiliary finite index subgroup $\Gamma_0<\Gamma$ of index $[\Gamma:\Gamma_0]=d$ and lift the restriction $\rho_0$ of $\rho$ to $\Gamma_0$ to a representation in ${\rm SL}(d,\mb{C})$ (see Lemma \ref{lem:lift}). 

As in Equation \eqref{e.prodev} we have 
\[1=\lambda^1(\rho_0(\gamma))^d\cdot\prod_{j\le d}{\frac{\lambda^j(\rho_0(\gamma))}{\lambda^1(\rho_0(\gamma))}}\]
As $\lambda^i(\rho_0(\gamma))/\lambda^j(\rho_0(\gamma))\in\mb{R}$ for every $i,j$, we conclude that $\lambda^1(\rho_0(\gamma))^d\in\mb{R}$. 

Therefore, $\rho_0(\gamma)$ is conjugate to a matrix of the form $q(\gamma)D(\rho_0(\gamma))$ where $D(\rho_0(\gamma))$ is a diagonal matrix with real entries and $q(\gamma)$ is a $d$-th root of unity (note that $q(\gamma)\mb{I}$ is a central element in ${\rm SL}(d,\mb{C})$). Therefore, the trace of $\rho_0(\gamma)$ has the form $r(\gamma)q(\gamma)$ where $r(\gamma)$ is a real number and $q(\gamma)$ is a $d$-th root of unity.

Define the $\mb{R}$-algebraic subvarieties of ${\rm SL}(d,\mb{C})$
\[
U_j=\{A\in{\rm SL}(d,\mb{C})\left|\,\mb{R}{\rm e}({\rm tr}(A))\sin(2j\pi/d)=\mb{I}{\rm m}({\rm tr}(A))\cos(2j\pi/d)\right.\}
\]

By the above discussion $\rho_0(\Gamma_0)$ is contained in the disjoint union $ U_0\sqcup\cdots\sqcup U_{d-1}$ and the same is true for the $\mb{R}$-Zariski closure $G$ of $\rho_0(\Gamma_0)$ in ${\rm SL}(d,\mb{C})$. 
As $G$ is an algebraic group, it has finitely many connected components, and the connected component of the identity $G_0$ is contained in $U_0$. The group $\Gamma_1:=\rho^{-1}(G_0)$ is a finite index subgroup of $\Gamma_0$. As $\rho_0(\Gamma_1)\subset U_0$, we have ${\rm tr}(\rho_0(\gamma))\in\mb{R}$ for every $\gamma\in\Gamma_1$.

By \cite{Aco19} this implies that $\rho_0$ is conjugate in ${\rm SL}(d,\mb{R})$ or ${\rm SL}(d/2,{\bf H})$ where ${\bf H}$ is the division algebra of quaternions over the real numbers.
However, $\rho_0$ cannot be conjugate in ${\rm SL}(d/2,{\bf H})$ because the eigenvalues of elements in ${\rm SL}(d/2,{\bf H})$ come in conjugate pairs, while we know by assumption that the eigenvalues of $\rho_0(\g)$ have distinct moduli for every $\gamma\in\Gamma_0$. 

Let $M\in{\rm SL}(d,\mb{C})$ be such that $M\rho_0(\Gamma_1)M^{-1}\subset{\rm SL}(d,\mb{R})$.
The result is proven as soon as the next proposition is established.
\end{proof}

\begin{prop}
   In the situation above, $M\rho(\Gamma)M^{-1}\subset{\rm PSL}(d,\mb{R})$.    
\end{prop}

\begin{proof}
We proceed similarly to Proposition \ref{claim:from fi to ambient}.
Let $N=[\Gamma:\Gamma_1]$ be the index of $\Gamma_1$ in $\Gamma$. For every $\gamma\in\Gamma$ we have $\gamma^{N!}\in\Gamma_1$.

By the Anosov property, $\rho(\gamma)$ is diagonalizable and determines a splitting by eigenlines $\mb{C}^d=L_1\oplus\cdots\oplus L_d$.
Since  $\rho(\gamma^{N!})$ preserves the same splitting, $ML_j$ is an eigenline of $M\rho(\gamma^{N!})M^{-1} \in \SL(d,\R)$.  

Thus $ML_j$ intersects $\R^d<\C^d$ in a real line, and the splitting
$\mb{C}^d=ML_1\oplus\cdots\oplus ML_d$ admits a real basis. In each such basis
the matrix representing $M\rho(\gamma)M^{-1}$ is diagonal with real entries.Its $j$-th diagonal entry is equal to an $N!$-th root of the corresponding diagonal element of the (real) matrix representing $\rho(\gamma^{N!})$. Since furthermore ${\lambda^i(\rho(\gamma))}/{\lambda^j(\rho(\gamma))}\in\mb{R}$, for every $i,j$ the arguments of the complex numbers $\lambda^i(\rho(\gamma)),\lambda^j(\rho(\gamma))$ are equal or differ by $\pi$. In particular, it has the form $q D_\gamma$ where $D_\gamma$ is a real diagonal matrix and $q$ is an $N!$-th root of unity. As a consequence $M\rho(\gamma)M^{-1}\in{\rm PSL}(d,\mb{R})$, for all $\gamma \in \Gamma$, as claimed. 
\end{proof}

We can now conclude:
\begin{prop}\label{prop:diag real}
Let $\rho:\Gamma\to{\rm PSL}(d,\mb{C})$ be a fully hyperconvex representation such that ${\rm AB}^k(\rho)\in\Delta$ for every $1\le k\le d-1$. Then $\rho$ is conjugate into ${\rm PSL}(d,\mb{R})$.
\end{prop}

\begin{proof}
By Proposition \ref{prop:real dilation}, for every $k$ we have
\[
\lambda^k(\rho(\gamma))/\lambda^{k+1}(\rho(\gamma))\in\mb{R}.
\]
So we can conclude by applying Theorem \ref{thm:real gaps real}.
\end{proof}

\section{Hausdorff dimension}
\label{sec:entropy}
In this section we prove Theorem \ref{thmINTRO:Hff}. Recall from the introduction that we denote by 
\[
h^k(\rho):=\limsup_{R\to\infty}\frac{\log|\{[\gamma]\text{ conjugacy class of }\pi_1(\Sigma)|\,\log|L_\rho^k(\gamma)|\le R\}|}{R}
\]
the $k$-th root entropy of a representation $\rho:\G\to\PSL(d,\C)$.

The main step in the proof is the following:

\begin{thm}\label{thm:entropy=1}
Let $\rho:\Gamma\to{\rm PSL}(d,\mb{C})$ be  $k$-hyperconvex. Then
 \[
 h^k(\rho)\ge 1
 \]
 with equality if and only if ${\rm AB}^k(\rho)$ lies in the diagonal $\Delta\subset\T(M_\Gamma)\times\T(\overline M_\Gamma)$.
\end{thm}

Theorem \ref{thm:entropy=1} implies Theorem \ref{thm:IntroReal} via Proposition \ref{prop:real dilation}. 
We now  give a quick proof of Theorem \ref{thmINTRO:Hff}
assuming Theorem \ref{thm:entropy=1}.

\begin{proof}[Proof of Theorem \ref{thmINTRO:Hff}]
Recall from Section \ref{s.Anosov} that we denote by $\Lambda^k_\rho=\xi^k(\deG)\subset \Gr_k(\C^d)$ the $k$-th limit set. By Pozzetti--Sambarino--Wienhard \cite{PSW:HDim_hyperconvex}, we have 
\[
{\rm Hdim}(\Lambda^k_\rho)=h^k(\rho).
\]
If we further denote by $\Lambda_\rho=\xi(\deG)\subset \calF(\C^d)$ the full limit set, it was proven in Pozzetti--Sambarino \cite[Corollary 5.13]{PSamb}, that
\[
{\rm Hdim}(\Lambda_\rho)=\max_{1\le k\le d-1}\{h^k(\rho)\}.
\]

As $\Lambda_\rho^k$ is a topological circle we always have $1\le{\rm Hdim}(\Lambda^k_\rho)=h^k(\rho)$. If ${\rm Hdim}(\Lambda_\rho)=1$ then 
\[
1={\rm Hdim}(\Lambda_\rho)\geq h^k(\rho),
\]
and we deduce $h^k(\rho)=1$ for every $1\le k\le d-1$. By Theorem \ref{thm:entropy=1}, all the Ahlfors--Bers parameters ${\rm AB}^k(\rho)$ lie on the diagonal $\Delta\subset\T(M_\Gamma)\times\T(\overline M_\Gamma)$. By Proposition \ref{prop:diag real}, this implies that $\rho$ is conjugate in ${\rm PSL}(d,\mb{R})$.
\end{proof}

We now move on to the proof of Theorem \ref{thm:entropy=1}. Let us briefly sketch the argument: The main point is to show that if ${\rm AB}^k(\rho)$ is not diagonal, then there exists a marked hyperbolic surface lamination $E\in\T(M_\Gamma)$ and a number $\kappa>1$ such that
\[
\frac{1}{\kappa}\ell_E(\cdot)\ge \log|L^k_\rho(\cdot)|.
\]
In view of Theorems \ref{t.Tholozan} and \ref{thm: orbit growth rate}, this immediately implies that
\[
h^k(\rho)\ge\kappa h(E)\ge\kappa>1.
\]

In order to prove the existence of $E$, we prove a laminated version of a result by Deroin and Tholozan \cite{DT16}, which we now outline.

Recall that we associate to the hyperconvex representation $\rho$ a laminated conformal action $\rho^k$ on $\fCP$ capturing the $k$-th root spectrum. Morally the hyperbolic surface lamination $E$ arises as Candel's uniformization of the harmonic filling of its laminated limit set $\fL_\rho^k=\xi^k_\cdot(\fL)\subset\fCP$. For every point $t\in\deH$, there is a unique harmonic filling $\H^2\to\H^3$ of the marking $\xi_t^k:\dH\to\CP$ provided by the work of Benoist--Hulin \cite{BH17}. This morally induces a leafwise Riemannian structure on $M_\G$ whose curvature is bounded from above by $-1$. The Candel uniformization of this metric would be strictly larger (see e.g. \cite{Ahlfors:schwarz}), and hence its length spectrum would strictly dominate the one of the harmonic filling, which in turn weakly dominates that of $\rho$. 

There are two main issues. 
First, the harmonic filling is in general singular so the induced metric is also singular, hence we cannot apply Candel's uniformization. Instead, we apply work of Wan on the universal Teichmüller space \cite{Wan}.
Second, in order to prove strong domination we need to show that such fillings give the structure of a (smooth) hyperbolic surface lamination, which gives us the compactness needed to adapt some arguments of Deroin-Tholozan \cite{DT16} and deduce strict domination.

Before going into details, we need to make a small detour into harmonic maps.  
\subsection{Harmonic maps}
We start by recalling the definition.

\begin{dfn}[Harmonic Map]
Let $f:X\to Y$ be a smooth map between Riemannian manifolds. Denote by $\nabla^X,\nabla^Y$ the Levi-Civita connections of $X,Y$ respectively. The second fundamental form of $f$ is given, for  $x\in X$ and $u,v\in T_xX$ by
\[
\mb{I}^f(u,v):=\nabla^{Y}_{df(u)}{df(v)}-df(\nabla^X_uv)\in T_{f(x)}Y.
\]
If $f$ is an embedding on the open set $U\subset X$, then $\mb{I}^f$ is the second fundamental form of the embedded submanifold $f(U)\subset Y$, that is $\mb{I}^f(u,v)=\mb{I}_{f(U)}(df(u),df(v))$.

We say that $f$ is \emph{harmonic} if the trace of $\mb{I}^f$ vanishes, that is, for any $x\in X$ and any orthonormal basis $E_1,\cdots,E_k$ of $T_x X$, we have 
\[
\tau(f)(x):=\mb{I}^f(E_1,E_1)+\cdots+\mb{I}^f(E_k,E_k)=0.
\]
If $f$ is an embedding on $U$ and $x\in U$, then $\tau(f)(x)$ is the mean curvature of $f(U)$ at $x$.
\end{dfn}

The following result of Jost--Karcher will be useful to control the convergence of sequences of harmonic maps.

\begin{thm}[{\cite[Theorem 4.9.2]{Jost}}]
\label{thm:elliptic reg}
Let $X,Y$ be complete smooth Riemannian manifolds with bounded sectional curvatures
\[
-\omega^2\le K_X,K_Y\le\kappa^2
\]
and such that the $\mc{C}^k$-norms of the Riemann curvature tensors are bounded by $B_k$. Let $B(x,R_X)\subset X$ and $B(y,R_Y)\subset Y$ be metric balls of radii
\[
R_X\le\min\left\{{\rm inj}_x(X),\frac{\pi}{2\kappa}\right\}\,\text{ and }\,R_Y\le\min\left\{{\rm inj}_y(Y),\frac{\pi}{2\kappa}\right\}
\]
(where if $\kappa=0$ the second terms are $+\infty$). There exists 
\[
c=c(R_X,R_Y,k,\omega,\kappa,B_k,{\rm dim}(X),{\rm dim}(Y))>0
\]
such that if $u:B(x,R_X)\to B(y,R_Y)$ is a harmonic map, then
\[
|u|_{\mc{C}^k}\le c.
\]
\end{thm}

We will only apply this result for $X=\mb{H}^2$ and $Y=\mb{H}^2,\mb{H}^3$, and now specialize the discussion to the case where $X=\mb{H}^2$.
In this case, it is convenient to take into account the complex structure $J:T^*\mb{H}^2\to T^*\mb{H}^2$, which determines a splitting of $T^*\mb{H}^2\otimes\mb{C}$ into $T^*_{1,0}\mb{H}^2\oplus T^*_{0,1}\mb{H}^2$ (the eigenspaces corresponding to the eigenvalues $i,-i$ of $J\otimes 1$). In turn, this induces a natural splitting
\[T^*\mb{H}^2\otimes_{\mb{C}} T^*\mb{H}^2=T^*_{2,0}\mb{H}^2\oplus T^*_{1,1}\mb{H}^2\oplus T^*_{0,2}\mb{H}^2.\]
We can further split the 2-dimensional complex subspace $T^*_{1,1}\mb{H}^2$ into 
\[
T^*_{1,1}\mb{H}^2=\mb{C}g_{\mb{H}^2}\oplus\mb{C}\omega_{\mb{H}^2}
\]
where $g_{\mb{H}^2}(\cdot,\cdot)$ is the hyperbolic metric on $\mb{H}^2$ and $\omega_{\mb{H}^2}(\cdot,\cdot)=g_{\mb{H}^2}(\cdot, J\cdot)$ is the symplectic form of $\mb{H}^2$. In particular, every complex 2-tensor $S$ on $\mb{H}^2$ decomposes as 
\[
S=S^{(2,0)}+S^{(1,1)}+S^{(0,2)}.
\]
where $S^{(2,0)},S^{(0,2)}$ are quadratic differentials. If $S$ is real, that is, it coincides with its complex conjugate $S=\overline{S}$, then 
\[
\overline{S}^{(0,2)}=S^{(2,0)}\,\text{ and }\, S^{(1,1)}=\overline{S}^{(1,1)}.
\]
If additionally, $S$ is symmetric and non-negative, then $S^{(1,1)}=\alpha g_{\mb{H}^2}$. This leads to the following definition:

\begin{dfn}[Hopf Differential]
If $f:\mb{H}^2\to Y$ is a smooth map then the pull-back of the metric $g_{Y}$ is a symmetric 2-tensor that admits a (unique) decomposition as 
\[
f^*g_{Y}=\alpha g_{\mb{H}^2}+\Phi+\overline{\Phi}
\]
where $\alpha:\mb{H}^2\to\mb{R}$ is a non-negative smooth function and $\Phi$ is a quadratic differential, the \emph{Hopf differential} of $f$. When $f$ is harmonic, the Hopf differential is holomorphic (see for example \cite[Section 2.2.3]{DW}).
\end{dfn}

\subsubsection{Existence and uniqueness of harmonic fillings}
Given a parametrized quasicircle $\xi:\partial\mb{H}^2\to\partial\mb{H}^3$ we seek a harmonic map $f:\mb{H}^2\to\mb{H}^3$ that extends $\xi$. We use the following particular case of a result of Benoist--Hulin \cite{BH17} establishing the existence and uniqueness of harmonic fillings.

\begin{thm}[{\cite{BH17}}]
\label{thm:bh}
For every $c>0$ there exists $B>0$ such that the following holds: Let $\xi:\partial\mb{H}^2\to\mb{CP}^1$ be the boundary extension of a $c$-quasi-isometric embedding $g:\mb{H}^2\to\mb{H}^3$. Then there exists a unique harmonic map $f:\mb{H}^2\to\mb{H}^3$ such that
\[
\sup_{x\in\mb{H}^2}{d_{\mb{H}^3}(f(x),g(x))}<B.
\]
In particular, $f$ continuously extends $\xi$. 
\end{thm}

The statement quantifies the distance between $f,g$ {\em uniformly} in terms of $c$ as opposed to the one in \cite{BH17}. We discuss in Appendix \ref{sec:appendix a} how to deduce it from their work.

\subsubsection{Uniformization}
Let  $f:\mb{H}^2\to\mb{H}^3$ be harmonic. The pull-back $f^*g_{\mb{H}^3}$ is a non-negative symmetric 2-tensor, which might be degenerate at some points. A result of Sampson \cite[Corollary of Theorem 3]{Sam78} ensures that it is non-degenerate on a dense open subset of $\mb{H}^2$.

Recall that we can write
\[
f^*g_{\mb{H}^3}=\alpha g_{\mb{H}^2}+\Phi+\overline{\Phi}
\]
where $\Phi$ is a holomorphic quadratic differential on $\mb{H}^2$ and $\alpha$ is a non-negative smooth function. We would like to find a hyperbolic metric $g'$ on $\mb{H}^2$ of the form 
\[
g'=\alpha'g_{\mb{H}^2}+\Phi+\overline{\Phi}
\]
where $\alpha'$ is a smooth positive function. The following result gives us a sufficient condition to be able to solve this problem: 

\begin{thm}[Wan \cite{Wan}]\label{t.Wang}
\label{thm:harmonic parametrization}
Let $\Phi$ be a bounded holomorphic quadratic differential on $\mb{H}^2$. There exists a unique (up to composition with isometries) harmonic quasi-conformal diffeomorphism $u:\mb{H}^2\to\mb{H}^2$ with Hopf differential $\Phi$. The pull-back $u^*g_{\mb{H}^2} = g_\Phi$ is a complete (hyperbolic) metric and has the form
\[
g_\Phi=\alpha_\Phi g_{\mb{H}^2}+\Phi+\overline{\Phi}.
\]
\end{thm}

Let us comment now about the regularity of the dependence of $u$ on $\Phi$. 
\begin{prop}\label{p.convergence uniformization}
For every $k\ge 1$ and compact set $D\subset\mb{H}^2$, if $\Phi_n\to\Phi$ uniformly on compact sets, then, up to composition with an isometry, $u_n$ converges to $u$ in $\mc{C}^k(D,\mb{H}^2)$. 
\end{prop}
\begin{proof}
Up to the composition with an isometry, we can assume that $u_n(o)=o$ for a fixed basepoint $o\in\mb{H}^2$ and {$d_ou_n\to M$ for some linear map $M$}. 
As the maps $u_n$ are all uniformly quasi-conformal (the quasi-conformal parameter is a continuous function of the Hopf differential) they are also uniformly quasi-Lipschitz (see for example \cite{Kiernan}). In particular, they map the bounded set $D$ to a fixed ball $B(o,R)$ for some $R$. Hence, by Theorem \ref{thm:elliptic reg}, for every $k\ge 1$ the $\mc{C}^k$-norm of $u_n$ is uniformly bounded. This implies that up to passing to subsequences we have that $u_n\to u'$ in the $\mc{C}^k$-topology on compact sets where $u':\mb{H}^2\to\mb{H}^2$ is harmonic and satisfies $u'(o)=o$ and {$d_ou'=M$}. As the convergence is $\mc{C}^2$, the Hopf differentials $\Phi_n$ converge to the Hopf differential of the limit, which therefore coincides with $\Phi$. By the uniqueness part of Theorem \ref{t.Wang}, we conclude that $u'$ is the unique harmonic map associated with $\Phi$. Lastly, as the limit $u'$ does not depend on the chosen subsequence, we conclude that the whole sequence $u_n$ converges to it.  
\end{proof}
\subsection{Geometry of harmonic fillings}
The source of the gap between the induced and the uniformized metric mentioned in the sketch of the proof of Theorem \ref{thm:entropy=1} will rely on a strong maximum principle. 
The following lemmas of Deroin--Tholozan \cite{DT16} are both consequences of the strong maximum principle. First, the induced metric has always curvature strictly smaller than $-1$ unless it is totally geodesic.
\begin{lem}[{see \cite[Lemma 2.5]{DT16}}] 
Let $f:\mb{H}^2\to\mb{H}^3$ be a harmonic map and let $U\subset\mb{H}^2$ be the subset where $f^*g_{\mb{H}^3}$ is non-degenerate. For all $x\in U$ the scalar curvature of $(U,f^*g_{\mb{H}^3})$ is bounded by $\kappa(f^*g_{\mb{H}^3})\le -1$. Furthermore $\kappa(f^*g_{\mb{H}^3})(x)<-1$ unless the second fundamental form of $f(U)$ vanishes at $f(x)$.
\end{lem}
For the second fact, recall that 
\[
f^*g_{\mb{H}^3}=\alpha g_{\mb{H}^2}+\Phi+\overline{\Phi}.
\]
By simple computations one gets 
\[
{\rm tr}_{g_{\mb{H}^2}}(f^*g_{\mb{H}^3})=\alpha
\]
and
\[
{\rm det}_{g_{\mb{H}^2}}(f^*g_{\mb{H}^3})=\alpha^2-4|\Phi|^2.
\]
As $f^*g_{\mb{H}^3}\ge 0$, we always have ${\rm det}_{g_{\mb{H}^2}}(f^*g_{\mb{H}^3})\ge 0$ which implies that the following system has always (possibly coincident) solutions $H,L:\mb{H}^2\to\mb{R}$
\[
\left\{
\begin{array}{l}
H+L=\alpha\\
HL=|\Phi|^2\\
H\ge L.
\end{array}
\right.
\]
Explicitly,
\[
H=\frac{{\rm tr}_{g_{\mb{H}^2}}(f^*g_{\mb{H}^3})+\sqrt{{\rm det}_{g_{\mb{H}^2}}(f^*g_{\mb{H}^3})}}{2}\ge\frac{{\rm tr}_{g_{\mb{H}^2}}(f^*g_{\mb{H}^3})-\sqrt{{\rm det}_{g_{\mb{H}^2}}(f^*g_{\mb{H}^3})}}{2}=L.
\]
Note that  
\[
{\rm tr}_{g_{\mb{H}^2}}(f^*g_{\mb{H}^3})\ge \sqrt{{\rm det}_{g_{\mb{H}^2}}(f^*g_{\mb{H}^3})}
\]
so both solutions $H,L$ are non-negative.

Following \cite{DT16}, in order to compare the pull-back tensor $f^*g_{\mb{H}^3}=\alpha g_{\mb{H}^2}+\Phi+\overline{\Phi}$
with the hyperbolic metric $g_\Phi:=\alpha_\Phi g_{\mb{H}}^2+\Phi+\overline{\Phi}$, we compare the solutions of their associated systems.

\begin{lem}[{see \cite[Lemma 2.6]{DT16}}]
\label{fact.LH}
Let $f:\mb{H}^2\to\mb{H}^3$ be a harmonic map with bounded Hopf differential $\Phi$. Let $g_\Phi=\alpha_\Phi g_{\mb{H}^2}+\Phi+\overline{\Phi}$ be the hyperbolic metric provided by Theorem \ref{thm:harmonic parametrization}. Consider the functions $H,H',L,L'$ on $\mb{H}^2$ defined by the systems
\[
\left\{
\begin{array}{l}
H+L=\alpha\\
HL=|\Phi|^2\\
H\ge L,
\end{array}
\right.\quad,\quad
\left\{
\begin{array}{l}
H'+L'=\alpha_\Phi\\
H'L'=|\Phi|^2\\
H'\ge L'.
\end{array}
\right..
\]
Assume that $H/H'$ attains a maximum in $\mb{H}^2$ and that $\sup\{2(H+L')\}<\infty$ on $\mb{H}^2$. Then either
\begin{itemize}
\item{$H'<H$ everywhere, in which case $f^*g_{\mb{H}^3}<g_\Phi$, or}
\item{$H'=H$ everywhere, in which case $f^*g_{\mb{H}^3}=g_\Phi$ and $f(\mb{H}^2)$ is a totally geodesic copy of $\mb{H}^2$ in $\mb{H}^3$.}
\end{itemize} 
\end{lem}

As $g_\Phi$ is non-degenerate, we have $H'>0$ so the quotient $H/H'$ is well-defined. 

The fact that $H'<H\Rightarrow f^*g_{\mb{H}^3}<g_\Phi$ is purely algebraic: 
If $0=HL=H'L'$, then $L=L'=0$ and $\alpha_\Phi=H'<H=\alpha$. 
If $HL=H'L'>0$ then \[L=|\Phi|^2/H \text{ and }L'=|\Phi|^2/H'\] and since $L\le H$ and $L'\le H'$, it holds that $L,L'\le|\Phi|$. As the function $x\to x+|\Phi|^2/x$ is decreasing on $(0,|\Phi|)$ we deduce again that $\alpha<\alpha_\Phi$.

The assumptions of Lemma \ref{fact.LH} are slightly different from the ones of \cite[Lemma 2.6]{DT16}: First in \cite{DT16}, the existence of a hyperbolic metric of the form $g_\Phi=\alpha_\Phi g_{\mb{H}^2}+\Phi+\overline{\Phi}$ is guaranteed by the results of \cite[Section 1.2]{DT16} which rely on the fact that $f$ is equivariant under a representation of a uniform lattice of ${\rm Isom}^+(\mb{H}^2)$. We instead use Theorem \ref{t.Wang} and assume that the Hopf differential of $f$ is bounded. 
Second in \cite{DT16} the existence of a maximum for $H/H'$ as well as the boundedness of $2(H+L')$ is immediate by equivariance, we instead assume both as hypotheses.
With the three additional assumptions, the proof of Lemma \ref{fact.LH} goes through exactly as in \cite[Lemma 2.6]{DT16}.

\subsection{Laminations from harmonic fillings, the proof of Theorem \ref{thm:entropy=1}}
The Reimann extension operator provides a natural way to extend quasi-conformal maps of $\mb{CP}^1$ to bi-Lipschitz diffeomorphism of $\mb{H}^3$. Let ${\rm QC}(\mb{CP}^1)$ be the space of quasi-conformal homeomorphism of $\mb{CP}^1$ endowed with the topology of uniform convergence. 
\begin{thm}[{Reimann \cite{Reimann:extension}}]
\label{thm:reimann}
There exists a map
\[
\mc{R}:{\rm QC}(\mb{CP}^1)\to{\rm Diff}^+(\mb{H}^3)
\]
with the following properties:
\begin{enumerate}
\item{$\mc{R}$ is continuous;}
\item{$\mc{R}(g)$ continuously extends $g$, and if $g$ is equivariant, so is $\mc{R}(g)$;}
\item{if $g:\mb{CP}^1\to\mb{CP}^1$ is $K$-quasi-conformal, then $\mc{R}(g)$ is $K^3$-bi-Lipschitz.} 
\end{enumerate}
\end{thm}

Let $\rho^k:\Gamma\to\mc{MG}$ be the laminated conformal action on $\mb{CP}^1\times\partial\mb{H}^2$ associated to the $k$-hyperconvex representation $\rho$. Recall that, by Proposition \ref{prop:qcconjugacy}, $\rho^k$ is quasi-conformally conjugated to the natural embedding $\iota:\Gamma\to\mc{MG}$.   This is,  $\rho^k=g\iota g^{-1}$ for some $K$-quasi-conformal $g:\mb{CP}^1\times\partial\mb{H}^2\to\mb{CP}^1\times\partial\mb{H}^2$ extending the tangent projection $\xi_\cdot ^k : \fL \to \fCP$.

Using $g$ and the Reimann operator $\mathcal R$,  we produce a continuous equivariant family of quasi-isometric maps that extend $\xi^k_t$.
Define 
\[R_t:=\mathcal R (g_t)|_{\H^2\times \{t\}}:\mb{H}^2\to\mb{H}^3,\]
and observe that for every $t\in\deH$, $\mathcal R (g_t)$ is $K^3$-bi-Lipschitz.  Thus $R_t$ is a $K^3$-quasi-isometric embedding that extends the restriction  of $g_t$ to $\mb{RP}^1$, which is $\xi^k_t$.  By the continuity properties of the Reimann extension operator $\mc{R}(\cdot)$, the family $R_t$ varies continuously with respect to uniform convergence.

We denote by $f_t:\mb{H}^2\to\mb{H}^3$ the unique harmonic extension of $\xi^k_t$ provided by  Theorem \ref{thm:bh}, which, by construction, stays at uniformly bounded distance from $R_t$ and define
\[
F:\mb{H}^2\times\partial\mb{H}^2\to\mb{H}^3\times\partial\mb{H}^2
\]
by $F(\cdot,t):=(f_t(\cdot),t)$.

\begin{lem}
\label{lem:harmonic equivariant}
$F$ is $(\iota,\rho^k)$-equivariant.    
\end{lem}     

\begin{proof}
The transformations $f_{\gamma t}\gamma$ and $\rho^k(\gamma,t)f_t$, being the pre- and post-compositions of a harmonic map with isometries of the domain and target are still harmonic. Furthermore, they extend to the boundary to the same map $\xi^k_{\gamma t}\gamma=\rho^k(\gamma,t)\xi^k_t$. By the uniqueness part of Theorem \ref{thm:bh}, we conclude that they are equal $f_{\gamma t}\gamma=\rho^k(\gamma,t)f_t$.
\end{proof}

\begin{lem}
\label{lem:harmonic continuous}
$F$ is continuous. The derivatives along the leaves vary continuously.    
\end{lem}

\begin{proof}
We want to show that for every compact domain $D\subset\mb{H}^2$ and $k\ge 0$ if $t\to t_0$ then $f_t\to f_{t_0}$ in the $\mc{C}^k$ topology on $\mc{C}^k(D,\mb{H}^3)$. 

By Theorem \ref{thm:bh}, the harmonic extension $f_t$ satisfies $d(f_t,R_t)\le B$ where $B$ only depends on $K$. In particular, the image $f_t(D)\subset\mb{H}^3$ lies in the $2B$-neighborhood of the image $R_{t_0}(D)$ for every $t$ sufficiently close to $t_0$. By Theorem \ref{thm:elliptic reg}, we deduce that for every $k\ge 1$ the $\mc{C}^k$ norms of $f_t$ are uniformly bounded. 

As $D$ and $k$ are arbitrary, the above discussion implies that, up to subsequences, $f_t$ converges in the $\mc{C}^k$ topology on compact subsets of $\mb{H}^2$ to a harmonic function $f':\mb{H}^2\to\mb{H}^3$. Note that, as $R_t\to R_{t_0}$ uniformly on compact sets and $d(f_t,R_t)\le B$ for every $t$, we have $d(f',R_{t_0})\le B$. Since $f'$ is a harmonic function at a uniformly bounded distance from the quasi-isometry $R_{t_0}$, we have that $f'$ continuously extends to $\partial R_{t_0}:\partial\mb{H}^2\to\partial\mb{H}^3$. By the uniqueness part of Theorem \ref{thm:bh}, we can conclude that $f'=f_t$. Lastly, as the limit $f'=f_{t_0}$ does not depend on the chosen subsequence of the $f_t$'s, we deduce that the whole sequence $f_t$ converges to $f_{t_0}$ in the $\mc{C}^k$-topology on compact sets of $\mb{H}^2$.
\end{proof}

As a consequence of Lemmas \ref{lem:harmonic equivariant} and \ref{lem:harmonic continuous}, $F$ determines a $\Gamma$-invariant transversely continuous symmetric 2-tensor $f_t^*g_{\mb{H}^3}$ on $\mb{H}^2\times\partial\mb{H}^2$. Each $f_t^*g_{\mb{H}^3}$ decomposes as
\[
f_t^*g_{\mb{H}^3}=\alpha_tg_{\mb{H}^2}+\Phi_t+\overline{\Phi}_t
\]
where $\alpha_t$ is a non-negative function and $\Phi_t$ is the Hopf differential of the harmonic map $f_t$. By Lemma \ref{lem:harmonic continuous}, $\Phi_t$ vary continuously in $t$ with respect to the $\mc{C}^\infty$ topology. Furthermore, by invariance and cocompactness of the action, we have $|\Phi_t|\le B$ for some $B$ and for every $t\in\partial\mb{H}^2$. 

By Theorem \ref{thm:harmonic parametrization}, there exists a unique family of hyperbolic metrics of the form 
\[
g_{\Phi_t}=\alpha_{\Phi_t}g_{\mb{H}^2}+\Phi_t+\overline{\Phi}_t
\]
all whose derivatives vary continuously with respect to $t$. Moreover, by the uniqueness part of Theorem \ref{t.Wang}, the family is also $\Gamma$-invariant that is 
\[
\gamma^*g_{\Phi_t}=g_{\Phi_{\gamma^{-1}t}}.
\]
Hence, it defines a hyperbolic surface lamination structure on $M_\Gamma=\mb{H}^2\times\partial\mb{H}^2/\Gamma$. 

\begin{lem}
\label{lem:harmonic dominate}
Either $f_t^*g_{\mb{H}^3}<g_{\Phi_t}$ for every $t$ or there exists $t$ such that $f_t(\mb{H}^2)$ is a totally geodesic plane in $\mb{H}^3$.    
\end{lem}

\begin{proof}
Let $(H_t,L_t)$ and $(H'_t,L'_t)$ be the functions associated to $\alpha_t,\Phi_t$ and $\alpha_{\Phi_t},\Phi_t$ respectively as in Lemma \ref{fact.LH}. By Lemmas \ref{lem:harmonic continuous} and \ref{lem:harmonic equivariant}, the functions $\mb{H}^2\times\partial\Gamma\to\mb{R}$ defined by
\[
(x,t)\to H_t(x),L_t(x),H'_t(x),L'_t(x)
\]
are continuous and $\Gamma$-invariant. By cocompactness of $\Gamma\curvearrowright\mb{H}^2\times\partial\Gamma$, the function $(x,t)\to H_t(x)/H'_t(x)$ attains a maximum over $\mb{H}^2\times\partial\Gamma$ on a slice $\mb{H}^2\times\{t\}$ and the function $(x,t)\to 2(H_t(x)+L'_t(x))$ is bounded. Therefore, the harmonic map $f_t:\mb{H}^2\to\mb{H}^3$ satisfies all the assumptions of Lemma \ref{fact.LH}.

By Lemma \ref{fact.LH}, either the maximum $\max_{(x,t)}\{H_t(x)/H'_t(x)\}$ is smaller than one, and $f_t^*g_{\mb{H}^3}<g_{\Phi_t}$ everywhere 
or it is 1 and realized on the slice $\mb{H}^2\times\{t\}$. On this slice we have that $f_t(\mb{H}^2)$ is a totally geodesic plane in $\mb{H}^3$.  
\end{proof} 

\begin{lem}
\label{lem:harmonic tot geo}
If $f_t(\mb{H}^2)$ is a totally geodesic plane for some $t$ then it is totally geodesic for every $t$. In particular ${\rm AB}^k(\rho)\in\Delta\subset\T(M_\Gamma)\times\T(\overline M_\Gamma)$.   
\end{lem}

\begin{proof}
By $\Gamma$-equivariance and continuity of $F$, the set of $t\in\partial\mb{H}^2$ for which $f_t(\mb{H}^2)$ is a totally geodesic plane is $\Gamma$-invariant and closed. Since every $\Gamma$-orbit is dense, the claim follows.
\end{proof}

We can now prove Theorem \ref{thm:entropy=1}.
\begin{proof}[Proof of Theorem \ref{thm:entropy=1}]
By Lemmas \ref{lem:harmonic dominate} and \ref{lem:harmonic tot geo}, either $f_t(\mb{H}^2)$ is a totally geodesic plane for every $t\in\deH$ in which case ${\rm AB}^k(\rho)\in\Delta$ or $f_t^*g_{\mb{H}^3}<g_{\Phi_t}$ for every $t\in\deH$, in which case, by continuity, $\Gamma$-invariance, and cocompactness of the action $\Gamma\curvearrowright\mb{H}^2\times\deH$, there exists $\kappa<1$ such that $f_t^*g_{\mb{H}^3}<\kappa g_{\Phi_t}$ for every $t$. In particular, by comparing length spectra, we get
\[
\log|L_\rho^k(\cdot)|\le \ell_{(M_\Gamma,f_t^*g_{\mb{H}^3})}(\cdot)\le \kappa\ell_{(M_\Gamma,g_{\Phi_t})}(\cdot).
\]
and, therefore, we have
\[
h^k(\rho)\ge \frac{1}{\kappa}h(M_\Gamma,g_{\Phi_t})\ge\frac{1}{\kappa}
\]
where in the last inequaltity we used that $h(M_\Gamma,g_{\Phi_t})\ge 1$ (Theorem \ref{thm: orbit growth rate}).
\end{proof}

\appendix

\section{Harmonic fillings}
\label{sec:appendix a}

Here we deduce from the arguments in \cite{BH17} the following strengthening of their main result. We only need the special case in which $X=\H^2$ and $Y=\H^3$, but we include a proof in the general case as it doesn't present additional difficulties. 

\begin{thm*}[{Benoist--Hulin \cite{BH17}}]
Let $X,Y$ be rank one symmetric spaces. For every $c>0$ there exists $B>0$ such that the following holds: Let $\xi:\partial X\to\partial Y$ be the boundary extension of a $c$-quasi-isometric embedding $g:X\to Y$. Then there exists a unique harmonic map $f:X\to Y$ such that
\[
\sup_{x\in X}{d_Y(f(x),g(x))}<B.
\]
In particular, $f$ continuously extends $\xi$. 
\end{thm*}

We adopt the same notations as in \cite{BH17}. We follow the main steps of their argument. First, they show that we can and will assume that the quasi-isometric embedding $g$ is smooth and the covariant derivatives $Dg,D^2g$ are uniformly bounded in terms of $c$.
\begin{prop*}[{\cite[Proposition 3.4]{BH17}}]
Let $X,Y$ be rank one symmetric spaces. For every $c>0$ there exists $c'>0$ (only depending on $c$ and $X,Y$) such that the following holds: Let $g:X\to Y$ be a $c$-quasi-isometric embedding. There exists a smooth map $g':X\to Y$ with $d(g,g')\le 2c$ and $||Dg'||,||D^2g'||\le c'$.
\end{prop*}

Fix a base-point $o\in X$ and let $B(o,R)$ be the ball of radius $R$ centered at $o$. By general theory, there exists a unique harmonic map $h_R:B(o,R)\to Y$ whose restriction to the sphere $\partial B(o,R)$ coincides with $g$.
The strategy of \cite{BH17} is to show that the maps $h_R$  converge to a harmonic map $h:X\to Y$ which stays at a bounded distance from $g:X\to Y$. The core of the argument is to show the following two estimates:

\begin{prop}[{\cite[Proposition 3.8]{BH17}}]
\label{prop:bh3 pt 8}
Let $X,Y$ be rank one symmetric spaces. For every $c>0$ and every smooth map $g:X\to Y$ with $||Dg||,||D^2g||\le c$,  the unique harmonic map $h_R:B(o,R)\to Y$ agreeing with $g$ on the sphere $\partial B(o,R)$ satisfies
\[
d(h_R(x),g(x))\le 8c^2{\rm dim}(X)d(x,\partial B(o,R)). 
\]
\end{prop}

This estimate is already explicit in $c>0$. The only part of their argument with non-explicit constant is potentially \cite[Proposition 3.6]{BH17}  where they prove:

\begin{prop*}[{\cite[Proposition 3.6]{BH17}}]
With notation as above, there exists a constant $M>0$ such that for every $R\ge 1$ one has $d(h_R,g)\le M$.
\end{prop*}
A priori, $M$ might depend on the initial boundary data $\xi$ and on the specific extension $g$. With some bookkeeping, the conclusion can be made stronger, that is, the constant $M$ can be chosen to depend only on $X,Y$ and $c$:

\begin{prop}[{Effective version of \cite[Proposition 3.6]{BH17}}]
\label{pro:bh3 pt 6}
Let $X,Y$ be rank one symmetric spaces. For every $c>0$ there exists a constant $M>0$ such that, for every smooth $c$-quasi-isometric embedding $g:X\to Y$ with $||Dg||,||D^2g||\le c$, the unique harmonic map $h_R:B(o,R)\to Y$ agreeing with $g$ on the sphere $\partial B(o,R)$ satisfies\[
d(h_R(x),g(x))\le M
\]
for every $R\ge 1$.
\end{prop}

We now explain how to extract this enhanced version from their arguments. We first choose the constants wisely: we fix a (large) threshold $T>1$ such that the following four conditions are satisfied:
\begin{enumerate}
\item{We impose
\[
T^{1/3}\le\frac{T}{16c^2{\rm dim}(X)}
\]
so that Inequality (4.3) of the paper
\[
1\le r_R\le\frac{1}{16c^2{\rm dim}(X)}\rho_R
\]
is satisfied for $r_R=\rho_R^{1/3}$ for every $\rho_R>T$.}
\item{We further impose
\[
\frac{1}{3c^2}-2^{12}c T^{-1/3}{\rm dim}(X)\ge\frac{1}{4c^2}=\sigma_0.
\]
}
\end{enumerate}
The next two requirements come from the first observation of \cite[Lemma 4.7]{BH17} which says that there exists a constant $\ep_0$ only depending on $\sigma_0=1/4c^2$ such that every subset of the standard sphere $\mb{S}^{{\rm dim}(X)-1}$ with  Lebesgue measure (normalized to be a probability measure) at least $\sigma_0$ contains two points whose angular distance is at least $\ep_0$. For this $\ep_0$ we ask the following:
\begin{enumerate}
\setcounter{enumi}{2}
\item{For the constant $A$ of \cite[Lemma 2.2]{BH17}, only depending on $X,Y$, and  the constant $\ep_0$ mentioned above (that depends only on $\sigma_0=1/4c^2$),
\[
\frac{(A+1)c}{\sin(\ep_0/2)^2}\le T^{1/3}.
\]
}
\item{For every $t\ge T$
\[
2\left(4e^{c/4}e^{-t^{1/3}/8c}+\frac{8t^2}{\sinh(t)}\right)<e^{-A}\left(\frac{\ep_0}{4}\right)^{2c}.
\]
}
\end{enumerate}

Proposition \ref{pro:bh3 pt 6}, and thus Theorem \ref{thm:bh}, follows once Claim \ref{claim:A3} is established.

\begin{claim}\label{claim:A3}
For every $R\ge T$ we have 
\[
\sup_{x\in B(o,R)}\{d_Y(h_R(x),g(x))\}\le T.
\]
\end{claim}

\begin{proof}
Set $\rho_R:=\sup_{x\in B(o,R)}\{d_Y(f(x),g(x))\}$, and let $x_R\in B(o,R)$ be a point where the supremum is achieved. We proceed by contradiction and assume that $\rho_R>T$.

By Proposition \ref{prop:bh3 pt 8}, we have 
\[
d(x_R,\partial B(o,R))\ge\frac{\rho_R}{8{\rm dim}(X)c^2}
\]
and, by choice (1), we have 
\[
\frac{\rho_R}{16{\rm dim}(X)c^2}\ge\rho_R^{1/3}.
\]
In particular, $B(o,R-1)$ contains $B(x_R,\rho_R^{1/3})$.

As in \cite[Definition 4.1]{BH17}, we write the functions $g$ and $h_R$ in exponential coordinates $\exp_{y_R}:T_{y_R}Y\to Y$ around $y_R:=g(x_R)$ 
\begin{align*}
 &g(z)=\exp_{y_R}(\rho_g(z)v_g(z)),\\
 &h_R(z)=\exp_{y_R}(\rho_h(z)v_h(z)),\\
 &h_R(x_R)=\exp_{y_R}(\rho_Rv_R)
\end{align*}
and single out on the sphere $S(x_R,\rho_R^{1/3})$ the sets
\begin{align*}
 &U_R=\{z\in S(x_R,\rho_R^{1/3})\left|\,\rho_h(z)\ge\rho_R-\rho_R^{1/3}/2c\right.\},\\
 &V_R=\{z\in S(x_R,\rho_R^{1/3})\left|\,\rho_h(z_t)\ge\rho_R/2\text{ for all $z_t$ on the geodesic $[x_R,z]$}\right.\},
\end{align*}
and define $W_R:=U_R\cap V_R$.

Denote by $\sigma$ the unique probability measure on the sphere $S(x_R,\rho_R^{1/3})$ invariant under the transitive action of the group of isometries of $X$ that stabilize $x_R$. More explicitly, $S(x_R,\rho^{1/3})$ is isometric to a standard sphere of dimension ${\rm dim}(X)-1$ of a certain radius and $\sigma$ is the Lebesgue measure on it rescaled so that it is a probability measure. As in \cite[Lemma 4.4]{BH17} we set $r_R:=\rho_R^{1/3}$ (observe again that by our choice (1) and the fact $\rho_R\ge T$, the inequality (4.3) is satisfied) and get
\begin{align*}
\sigma(W_R) &\ge\frac{1}{3c^2}-2^{12}{\rm dim}(X)c\rho_R^{-1/3}\\
 &\ge\frac{1}{3c^2}-2^{12}{\rm dim}(X)c T^{-1/3}\ge\frac{1}{{\color{blue}}4c^2}
\end{align*}
where the last inequality comes from our choice (2).

Then, by the observation above our choice (3), as $\sigma(W_R)\ge \sigma_0$, we can find two points $z_1,z_2\in W_R$ with angular distance (denoted by $\theta(\cdot,\cdot)$) at least $\ep_0$. 

Denote by $(x_1|x_2)_{x_3}$ (resp. $(y_1|y_2)_{y_3}$) the Gromov product of the points $x_1,x_2\in X$ based at $x_3\in X$ (resp. $y_1,y_2\in Y$ based at $y_3\in Y$). By \cite[Lemma 2.1.a]{BH17} we have
\begin{align*}
(x_R|z_1)_{z_2},(x_R|z_2)_{z_1} &\ge d(x_R,z_j)\sin(\theta(z_1,z_2))^2\\
 &\ge\rho_R^{1/3}\sin^2(\ep_0/2)
\end{align*}
and by our choice (3) the latter is greater than
\[
\rho_R^{1/3}\sin^2(\ep_0/2)\ge(A+1)c.
\]
Thus
\[
\min\{(x_R|z_1)_{z_2},(x_R|z_2)_{z_1}\}\ge(A+1)c.
\]
Therefore, by \cite[Lemma 2.2]{BH17}, we deduce
\[
\min\{(y_R|g(z_1))_{g(z_2)},(y_R|g(z_2))_{g(z_1)}\}\ge 1.
\]
We thus get
\begin{flalign*}
\theta(v_g(z_1),v_g(z_2)) &\ge e^{-(g(z_1)|g(z_2))_{y_R}} &\text{by \cite[Lemma 2.1.c]{BH17},}\\
&\ge e^{-A}e^{-c(z_1|z_2)_{x_R}} &\text{ by \cite[Lemma 2.2]{BH17}},\\
&\ge e^{-A}\left(\frac{\ep_0}{4}\right)^{2c} &\text{by \cite[Lemma 2.1.b]{BH17}.}
\end{flalign*}

By the triangle inequality, we have 
\[
\theta(v_g(z_1),v_g(z_2))\le\theta(v_g(z_1),v_R)+\theta(v_g(z_2),v_R).
\]

As $z_1,z_2\in W_R$ and $W_R=U_R\cap V_R$ (and Inequality (4.3) in \cite{BH17}  is again satisfied by our choice (1)), we can apply both \cite[Lemma 4.5]{BH17} and \cite[Lemma 4.6]{BH17} to conclude that
\begin{align*}
\theta(v_g(z_j),v_R) &\le\theta(v_g(z_j),v_h(z_j))+\theta(v_h(z_j),v_R)\\
 &\le 4e^{c/4}e^{-\rho_R^{1/3}/8c}+\frac{8\rho_R^2}{\sinh(\rho_R)}.
\end{align*}

Combining the previous inequalities one gets
\begin{align*}
e^{-A}\left(\frac{\ep_0}{4}\right)^{2c} &\le\theta(v_g(z_1),v_g(z_2))\\
 &\le \theta(v_g(z_1),v_R)+\theta(v_g(z_{\color{blue}2}),v_R)\\
 &\le 2\left(4e^{c/4}e^{-\rho^{1/3}_R/8c}+\frac{8\rho_R^2}{\sinh(\rho_R)}\right).
\end{align*}

As $\rho_R\ge T$, by our choice (4) we get a contradiction.
\end{proof}

\section{Entropy and orbital growth rates}\label{sec:appendix thermo}

We explain here how to deduce Theorem \ref{thm: orbit growth rate} from Theorem \ref{t.Tholozan}.
Many of the ingredients that we require are scattered throughout the literature, though not always in a published format or in the level of generality that we require.
We have tried to keep this section as self-contained as possible.

\subsection{Preliminaries on reparameterizations}
Denote as in \S\ref{subsec: entropy} by $\phi$ the geodesic flow on $M_\Gamma \cong T^1\H^2/\Gamma$.
We require some preliminary notions on reparameterizations of $\phi$ and entropy, mostly following \cite[\S3]{BCLS} and \cite[\S1]{Tholozan:notes}. 

Let $r: M_\Gamma \to \R_{>0}$ be a potential, i.e., a continuous positive function.
We will construct a flow $\phi^r$ whose speed differs from that of $\phi$ pointwise according to $r$.  
To do so, first define 
\begin{equation}\label{eqn: def kappa}
    \kappa_r(x,t) = \int_0^t r(\phi_s (x))~ds.
\end{equation}
Note that $\kappa_r$ satisfies the following cocycle property
\[\kappa_r(x,s+t) = \kappa_r(\phi_s(x), t) + \kappa_r(x,s).\]
For each $x\in M_\Gamma$, \[\kappa_r(x, \cdot) : \R \to \R\] is continuous, increasing, and proper, hence has a continuous, increasing inverse 
\begin{equation}\label{eqn: alphar}
    \alpha_r(x,\cdot): \R \to \R.
\end{equation}
The \emph{reparameterization} $\phi^r$ of $\phi$ by $r$ is given by
\begin{equation}\label{eqn: reparameter}
    \phi^r_t(x) = \phi_{\alpha_r(x,t)}(x).
\end{equation}

We collect some facts about $\phi^r$ in the following lemma; see \cite[\S3]{BCLS}.
\begin{lemma}\label{lem: reparam facts}
    Let $r: M_\Gamma \to \R_{>0}$ be continuous, and let $\phi^r$ be the reparameterizaiton defined as in \eqref{eqn: reparameter}.
    \begin{enumerate}
        \item A point $x \in M_\Gamma$ is $\phi$-periodic with period $T$ if and only if  $x$ is $\phi^r$-periodic with period $\kappa_r(x,T)$.
        \item The assignment \[\mu \mapsto \frac{r\mu}{\int r~d\mu}\]
        defines a homeomorphism between spaces of $\phi$-invariant and $\phi^r$-invariant  Borel probability measures on $M_\Gamma$.
    \end{enumerate}
\end{lemma}

Our goal is to assign to our marked Riemann surface lamination $f: M_\Gamma \to W$ a reparameterization $\phi^r$ which is continuously conjugate to the flow $\psi$ whose entropy was computed in \S\ref{subsec: entropy}; in particular the topological entropy and orbital growth rates of $\phi^r$ are the same as those of $\psi$.
We first construct an orbit equivalence between $\psi$ and $\phi$. 

\begin{lemma}\label{lem: orbit equivalence}
    There is a continuous lamination equivalence $g : M_\Gamma \to W $, leafwise homotopic to $f$, mapping orbits of $\psi$ to orbits of $\phi$, i.e., a continuous orbit equivalence. 
\end{lemma}

\begin{proof}
    Lift $f$ to a map $\tilde f$ between the covers $\H^2\times \partial \H^2$ and $\widetilde W$ corresponding to $\Gamma$ and $f_*\Gamma$, respectively.
    Every point $(x,t) \in \H^2\times \partial \H^2$ lies on a geodesic line $[t, y]$ directed from $t$ to $ y \in \partial \H^2$.  To the right of $x$ on $[t,y]$, there is an orthogonal geodesic ray pointing towards $z \in \partial \H^2$.  
    Using the boundary extension from Lemma \ref{lem: uniform bi-Lipschitz on leaves}, the triple $(t, y,z)$ maps to a triple of points on $\partial \tilde f (\H^2\times \{t\})$.  
    Define $\tilde g(x,t)$ as the orthogonal projection of $\tilde f (z)$ to the line $[\tilde f(t), \tilde f (y)]$.
    
    The construction is $\Gamma$-equivariant, continuous, and every $\tilde \phi$-orbit is mapped homeomorphically to a $\tilde \psi$-orbit so that $\tilde g$ descends to a homeomorphic lamination map $g: M_\Gamma \to W$.  
    Also, $\tilde f$ and $\tilde g$ are $\Gamma$-equivariantly leafwise homotopic via the leafwise straight line homotopy, so that $g$ has all of the desired properties.
\end{proof}

Denote by ${}_g\psi : M_\Gamma \times \R \to M_\Gamma$ the continuous flow $(x,t) \mapsto g\circ \psi_t\circ g\inverse (x)$, and note that 
    \[ \phi_t (x) = {}_g\psi_{\beta(x,t)}(x)\]
for some continuous function \[\beta: M_\Gamma \times \R \to \R\]
satisfying the $\phi$-cocycle condition
\[\beta(x,s+t) = \beta(x,s) + \beta(\phi_s(x), t).\]
In particular, $\beta(x,\cdot)$ is an increasing homeomorphism of $\R$ for all $x$.
More concisely, we write $\phi = {}_g\psi_\beta$.

Since $g$ is merely continuous, $\beta$ need not be differentiable along orbits, hence may not be equal to $\kappa_r$ for some potential $r: M_\Gamma \to \R_{>0}$, as in \eqref{eqn: def kappa}.  
In the proof of the next lemma, we use an averaging procedure to construct a potential $r$ such that ${}_g\psi_{\kappa_r}$ is conjugate to ${}_g\psi_{\beta}$.
The proof we present here is contained within \cite[\S\S1.2--1.4]{Tholozan:notes}.
\begin{lemma}\label{lem: reparam conj}
    There is a continuous function $r: M_\Gamma \to \R_{>0}$ and a homeomorphism $h: M_\Gamma \to M_\Gamma$ such that 
    \[ {}_g\psi \circ h = h\circ \phi^r. \]
    In particular, $\psi$ is topologically conjugate to $\phi^r$.
\end{lemma}

\begin{proof}
    
    By continuity of $\beta$ and compactness of $M_\Gamma$, there is a $T>0$ such that  $\beta(x,T)>0$ for all $x \in M_\Gamma$.
    Define $r(x) = \frac{1}{T} \beta(x,T)$, and a $\phi$-cocycle 
    \[\kappa_r(x,t) = \frac{1}{T} \int_0^tr(\phi_s(x))~ds,\]
    as in \eqref{eqn: def kappa}.
    
    A computation shows that $\kappa_r$ and $\beta$ are \emph{Liv{\v s}ic cohomologous}, i.e., there is a continuous function $G: M_\Gamma \to \R$ such that 
    \[\kappa_r(x,t) - \beta (x,t) = G(\phi_t(x)) - G(x), \]
    for all $x$ and $t$.
    Indeed, one can show that
    \[G(x) = \frac{1}{T}\int_0^T\beta(x, s)~ds.\]
    Let $\alpha_r(x,\cdot)$ be the inverse to $\kappa_r(x,\cdot)$, as in \eqref{eqn: alphar}.
    
    We now verify that the continuous map
    \[h(x) = {}_g\psi_{G(x)}(x) \]
    satisfies
    \[{}_g\psi_s(h(x)) = h(\phi^r_s(x))\]
    for all $x$ and $s$, where $\phi^r = \phi_{\alpha_r}$ (see \eqref{eqn: reparameter}).
    Indeed, for a given pair $(x,s)$, let $t$ be such that $s = \kappa_r(x,t)$ and compute
    \begin{align*}
        {}_g\psi_s(h(x)) &={}_g\psi_{s+G(x)}(x)\\
        & = {}_g\psi_{\kappa_r(x,t)+G(x)}(x)\\
        &= {}_g\psi_{\beta(x,t) + G(\phi_t(x))}(x)\\
        &= {}_g\psi_{G(\phi_t(x))}\circ\phi_t(x) \\
        & = h(\phi_t(x)).
    \end{align*}
    Since $s = \kappa_r(x,t)$, we have $\alpha_r(x,s) = t$.
    Then $\phi_t(x) = \phi_{\alpha_r(x,s)}(x) = \phi^r_s(x)$. Together with the previous computation, this shows that ${}_g\psi \circ h = h\circ \phi^r$.

    Since $h\inverse$ is given by $x\mapsto {}_g\psi_{-G(x)}(x)$, we have shown that ${}_g\psi$ is topologically conjugate to $\phi^r$, hence that $\psi$ is conjugate to $\phi^r$, proving the lemma. 
\end{proof}

\subsection{Entropy of reparameterizations}\label{subsec: entropy of reparam}

The following proposition states that the topological entropy of a reparameterization $\phi^r$ (defined as in \eqref{eqn: reparameter}) is a continuous function of the reparameterization potential $r$. 

\begin{prop}\label{prop: entropy continuous}
    Let $r_n: M_\Gamma \to \R_{>0}$ be a sequence of potentials converging uniformly to a potential $r_\infty$.    Then 
    \[\lim_{n \to \infty} h(\phi^{r_n}) = h(\phi^{r_\infty}).\]
\end{prop}
\begin{remark}
In the sequel we only use lower semi-continuity:
\[\liminf_{n\to \infty }h(\phi^{r_n})\ge h(\phi^{r_\infty}).\]    
\end{remark}

\begin{proof}
    The variational principle for topological entropy of a continuous flow on a compact metric space asserts, for any potential $r$, that
    \[h(\phi^r) = \sup_{\nu \in \mathbb P(M_\Gamma)^{\phi^r}}h(\phi^r,\nu),\]
    where $\mathbb P(M_\Gamma)^{\phi^r}$ is the space of $\phi^r$-invariant probability measures on $M_\Gamma$ and $h(\phi^r, \nu)$ is the measure theoretic entropy of $\phi^r$ with respect to $\nu$; see \cite[\S\S9-10]{VO:ergodic}.  

    Using Lemma \ref{lem: reparam facts} (2), we can write
    \[h(\phi^r) = \sup_{\mu \in \mathbb P(M_\Gamma)^{\phi}} h\left( \phi^r, \frac{r\mu}{\int r~d\mu}\right) = \sup_{\mu \in \mathbb P(M_\Gamma)^{\phi}} \frac{h( \phi^r, r\mu)}{\int r~d\mu} .\]
    By Abramov's formula (\cite{Abramov:entropy} or \cite{Ito:entropy}), we have 
    $h(\phi^r, r\mu) = h(\phi,\mu)$ for all $\mu \in \mathbb P(M_\Gamma)^\phi$.
    Thus we obtain the variational formula\footnote{This is essentially \cite[Lemma 2.4]{Sambarino:quantitative}.}
    \begin{equation}\label{eqn: variational}
        h(\phi^r)= \sup_{\mu \in \mathbb P(M_\Gamma)^{\phi}} \frac{h( \phi,\mu)}{\int r~d\mu}.
    \end{equation}

    Let $\ep>0$ be given and find $\mu_\ep$ such that
    \[ h(\phi^{r_\infty}) - \frac{h(\phi,\mu_\epsilon)}{\int r_\infty ~d\mu_\epsilon} <\ep.\]
    Since we have uniform convergence $r_n \to r_\infty$ of continuous positive functions,
    for $n$ large enough, we have 
    \[\int r_\infty ~d\mu_\ep > \int r_n ~d\mu_\ep - \ep >0.\]
    Then 
    \[h(\phi^{r_\infty}) \le \frac{h(\phi, \mu_\ep)}{\int r_\infty~d\mu_\ep}+\ep<\frac{h(\phi, \mu_\ep)}{\int r_n~d\mu_\ep-\ep}+\ep \le h(\phi^{r_n})+O(\ep),\]
    holds for $n$-large enough.
    Since $\ep>0$ was arbitrary, this proves that 
    \[h(\phi^{r_\infty})\le \liminf_{n\to \infty}h(\phi^{r_n})\]
    
    Now we show that 
    \[\limsup_{n\to \infty}h(\phi^{r_n}) \le h(\phi^{r_\infty}).\]
    Let $\min r_\infty >\ep>0$ be given, and find measures $\mu_n \in \mathbb P (M_\Gamma)^\phi$ such that 
    \[h(\phi^{r_n})\le \frac{h(\phi,\mu_n)}{\int r_n ~d\mu_n}+\ep\]
    for all $n$.
    Since $\mathbb P(M_\Gamma)^\phi$ is compact, we may pass to a subsequence (without renaming) and assume that $\mu_n \to \mu_\infty$.
    
    Since $\phi$ is \emph{expansive}, the function $\mu \mapsto h(\phi,\mu)$ is upper semi-continuous (e.g., \cite[\S9]{VO:ergodic}).  Thus, if $n$ is large enough, we have 
    \[h(\phi, \mu_n) \le h(\phi, \mu_\infty) +\ep.\]

    Since $\|r_n- r_\infty\|_\infty \to 0$ and $\mu_n \to \mu_\infty$, if $n$ is large enough, we have 
    \[ \int r_n ~d\mu_n >\int r_\infty ~d\mu_\infty - \ep >0.\]
    Thus for all $n$ large enough, we have 
    \[h(\phi^{r_n})\le \frac{h(\phi,\mu_n)}{\int r_n ~d\mu_n} +\epsilon < \frac{h(\phi, \mu_\infty)+\epsilon}{\int r_\infty~d\mu_\infty - \epsilon}+\epsilon \le h(r^{r_\infty})+O(\epsilon).\]
    Since $\ep>0$ was arbitrary, we conclude the proposition.
\end{proof}

We are now ready to combine the above ingredients to deduce Theorem \ref{thm: orbit growth rate}.

\begin{proof}[Proof of Theorem \ref{thm: orbit growth rate}]
By Lemma \ref{lem: reparam conj}, $\psi$ is topologically conjugate to $\phi^r$ for some continuous potential $r$, so it suffices to prove the theorem for flows of this form.

Since $\phi$ is a topologically transitive Anosov flow on $M_\Gamma$, it can be described as the suspension of a subshift of finite type with H\"older continuous roof function, i.e., it admits a strong Markov coding \cite{Bowen2, Bowen:symbolic}.
This is still true if $r$ is H\"older, and classical results give $h(\phi^r) = H(\phi^r)$  in this case \cite{Pollicott}.
Otherwise, we consider a sequence of H\"older continuous $r_n$ converging uniformly to $r=r_\infty$.

We claim that 
\begin{equation}\label{eqn: orbit growth cont}
    \lim_{n\to \infty} H(\phi^{r_n}) = H(\phi^{r_\infty}).
\end{equation}
Indeed, let $\ep>0$ be given and find $n$ large enough that
\[1-\epsilon <\frac{r_\infty}{r_n} \text{ and } 1-\epsilon < \frac{r_n}{r_\infty}\]
holds.
Let $\gamma$ be a periodic $\phi$-orbit with period $\ell(\gamma)$ and denote by $\ell_n (\gamma)$ its period for the flow $\phi^{r_n}$, for $n=1, 2, ..., \infty$.
Using Lemma \ref{lem: reparam facts} (1), 
\[\ell_n (\gamma) = \int_0^{\ell(\gamma)} r_n (\phi_s(x)) ~ds.\]
Using the lower bounds on $r_\infty/r_n$, we find that 
\[\ell_\infty (\gamma) < R \Rightarrow \ell_n (\gamma) < \frac{R}{1-\ep}.\]
Similarly, 
\[\ell_n (\gamma) < R(1-\ep) \Rightarrow \ell_\infty (\gamma) < R.\]
Thus 
\[\#\{\gamma:  \ell_n(\gamma)<R(1-\ep)\} \le \#\{\gamma:  \ell_\infty(\gamma)<R\} \le  \#\{\gamma:  \ell_n(\gamma)<R/(1-\ep)\}. \]
It follows then that 
\[(1-\ep)H(\phi^{r_n}) \le H(\phi^{r_\infty}) \le \frac{H(\phi^{r_n})}{1-\ep}\]
holds for all $n$-large enough.
Since $\ep>0$ was arbitrary, we conclude the claim that $\lim_{n\to \infty} H(\phi^{r_n}) = H(\phi^{r_\infty}) = H(\phi^r)$.

Using Proposition \ref{prop: entropy continuous}, \eqref{eqn: orbit growth cont}, and the fact that $H(\phi^{r_n}) = h(\phi^{r_n})$ for all $n$, we obtain 
\[h(\phi^r) \le \liminf_{n\to \infty}h(\phi^{r_n}) = \liminf_{n\to \infty}H(\phi^{r_n}) = H(\phi^r).\]
Combining this with Theorem \ref{t.Tholozan}, that $h(\phi^r)\ge 1$, completes the proof of the theorem.
\end{proof}

\bibliographystyle{amsalpha.bst}
\bibliography{references}

\Addresses

\end{document}

%% file: macros.tex
\newcommand{\MG}{\mathcal{MG}}

\newcommand{\Hom}{{\rm Hom}}

\newcommand{\deH}{\partial\HH^2}

\newcommand{\fL}{\mathcal L}
\newcommand{\fCP}{\CP \times \dH}
\newcommand{\fCPL}{\CP \times \dH-\fL}
\newcommand{\fCPgL}{\CP \times \dH-g(\fL)}
\newcommand{\fCPgLprime}{\CP \times \dH-g'(\fL)}
\newcommand{\fCPLmG}{\quotient{\fCPL}{ \iota(\Gamma)}}

\newcommand{\fM}{\mathcal{LM}}
\newcommand\quotient[2]{
	\mathchoice
	{
		\text{\raise.5ex\hbox{$#1$}\big/\lower.5ex\hbox{$#2$}}%
	}
	{
		\text{\raise.25ex\hbox{$#1$}\big/\lower.25ex\hbox{$#2$}}%
	}
	{
		#1\,/\,#2
	}
	{
		#1\,/\,#2
	}
}
\newcommand{\Chark}{\Xi^k(\G,\PSL(d,\C))}
\newcommand{\pD}{\mathbb D}